\documentclass[11pt]{amsart}
\usepackage{amsmath, amsfonts, amsthm, amssymb}
\usepackage{tikz-cd}

\usepackage{cite}
\usepackage{xcolor}

\usepackage[all,cmtip,line]{xy}
\usepackage{array}
\usepackage{enumerate}
\usepackage{srcltx} % for inverse search in LINUX -- does not cause any problems in Windows

\newtheorem{thm}{Theorem}[section]
\newtheorem{lem}{Lemma}[section]
\newtheorem{prop}[lem]{Proposition}
\newtheorem{cor}[lem]{Corollary}

\newtheorem{defn}[lem]{Definition}
\newtheorem{rem}[lem]{Remark}

\numberwithin{equation}{section}

\newcommand{\bR}{ \mathbb{R}} %%%% real
\newcommand \eps{\varepsilon}

\newlength{\originalbase}
\title{An improvement of regularity result for pseudo Calabi flow}

\author{Jingrui Cheng,  Junhao Tian}

\begin{document}
\maketitle

\begin{abstract}
In this paper,  we observe that if the initial data of pseudo Calabi flow has volume form $C^0$ close to a smooth one,  then the flow is immediately smooth for $t>0$.  As an application,  we show that if the initial data has volume form $C^0$ close to that of a cscK metric,  then the pseudo Calabi flow exists for $t\in (0,+\infty)$.  We also prove similar improvement of regularity and long time existence result for pseudo Calabi flow on a Fano manifold when the volume form is bounded and the class is close to $c_1(M)$.  
\end{abstract}

\section{Introduction}

The study of canonical metrics in K\"ahler geometry has long been a central theme in complex differential geometry, with roots tracing back to Calabi's seminal program in the 1950s.  In the absense of holomorphic vector fields,  these canonical metrics reduces to K\"ahler metrics with constant scalar curvature (cscK). This program has shaped a vast body of research, linking complex analysis, algebraic geometry, and nonlinear partial differential equations. \\

Let $M$ be a compact K\"ahler manifold,  $\omega_0$ is a K\"ahler form on $M$.  From the $\partial\bar{\partial}$-lemma,  the K\"ahler metric which is in the same cohomology class as $[\omega_0]$ can be represented as $\omega_{\varphi}:=\omega_0+\sqrt{-1}\partial\bar{\partial}\varphi,\,\varphi\in C^2(M)$.  Then the cscK equation reads:
\begin{equation}\label{cscK}
R(\omega_{\varphi})=\underline{R},\,\,\,\omega_0+\sqrt{-1}\partial\bar{\partial}\varphi>0.
\end{equation}
In the above,  $R(\omega_{\varphi})$ denotes the scalar curvature of the metric $\omega_{\varphi}$,  and $\underline{R}$ is a topological constant that only depends on $[\omega_0]$.\\

In the seminal work by Calabi \cite{calabi82},  he also proposed a flow to find a cscK metric within a K\"ahler class by solving for 
\begin{equation}\label{Calabi-flow}
\partial_t\varphi=R(\omega_{\varphi})-\underline{R}.
\end{equation}
In addition to being a possible approach to find cscK metrics,  the study of the flow equation (\ref{Calabi-flow}) is also of independent interest.  The short time existence of (\ref{Calabi-flow}) was studied in \cite{chenhe05}.  \cite{chenhe05} also proved the long time existence of (\ref{Calabi-flow}) when the initial data is close to cscK under smooth topology.  
However,  the analytical study of (\ref{Calabi-flow}) is much harder than its elliptic version (\ref{cscK}).  The apriori estimates developed in Chen-Cheng \cite{cc1} do not work for (\ref{Calabi-flow}).  \\

In view of the difficulties of studying Calabi flow (\ref{Calabi-flow}) as a 4th order fully nonlinear parabolic equation,  Chen-Zheng \cite{Chen-Zheng} proposed to study the so-called pseudo-Calabi flow,  which is formally of second order:
\begin{equation}\label{PCF}
\partial_t\varphi=\Delta_{\varphi}^{-1}(\underline{R}-R(\omega_{\varphi})).
\end{equation}
In the above,  $\Delta_{\varphi}^{-1}(\underline{R}-R(\omega_{\varphi}))$ denotes the function $f_{\varphi}$ such that $\Delta_{\varphi}f_{\varphi}=\underline{R}-R(\omega_{\varphi})$.  
Here $\Delta_{\varphi}$ denotes the Laplace operator under the metric $\omega_{\varphi}$.  The choice of normalization of $f_{\varphi}$ (up to an additive constant) is not that important,  because different choices of normalization will only lead to adding a function of $t$ to $\varphi$,  which has no effect on the metric $\omega_{\varphi}$.

Note that (\ref{PCF}) can be re-written as:
\begin{equation}\label{PCF-2}
\partial_t\varphi=\log\frac{\omega_{\varphi}^n}{\omega_0^n}+P,\,\,\Delta_{\varphi}P=\underline{R}-tr_{\omega_{\varphi}}(Ric(\omega_0)).
\end{equation}
Here $tr_{\omega_{\varphi}}(Ric(\omega_0))$ denotes the trace of $Ric(\omega_0)$ under the metric $\omega_{\varphi}$.  We are going to choose the normalization of $P$ as:
\begin{equation*}
\int_MP\omega_{\varphi}^n=0.
\end{equation*}
Note that (\ref{PCF-2}) is apparently of second order,  which suggests that it might be easier to study than (\ref{Calabi-flow}) from analytic point of view. 
Chen-Zheng \cite{Chen-Zheng} studied (\ref{PCF}) (or equivalently (\ref{PCF-2})) and they proved:
\begin{thm}\label{t1.1}
Let $(M,\omega_0)$ be a compact K\"ahler manifold,  then:
\begin{enumerate}
\item Let $\varphi_0\in PSH(\omega_0)\cap C^{2,\alpha}(M)$ for some $0<\alpha<1$.  Then there exists $T_0>0$,  such that there exists a pseudo Calabi flow which exists on $M\times [0,T_0]$ with $\varphi_0$ being the initial data.  Moreover,  the flow is $C^{\infty}$ smooth for $t>0$.
\item If $\omega_0$ is a cscK metric,  then there exists $\eps_0>0$,  such that if $||\varphi_0||_{2,\alpha}\le \eps_0$,  the pseudo Calabi flow starting from $\varphi_0$ exists on $M\times [0,+\infty)$.  Moreover,  the flow converges to a cscK metric (possibly different from $\omega_0$) as $t\rightarrow \infty$.
\end{enumerate}
\end{thm}
Chen-Zheng \cite{Chen-Zheng} also conjectured that the pseudo Calabi flow exists on $M\times [0,+\infty)$ for any smooth initial data unconditionally.  \\

In this paper,  we partially generalize the elliptic estimates developed in Chen-Cheng \cite{cc1} to study (\ref{PCF}).  However,  we are not able to prove the unconditional long time existence.  The difficulty is mainly related to the fact that the linearized complex Monge-Ampère equation lacks similar $C^{0,\alpha}$ estimate as in the real version,  which was proved by Caffarelli-Guterriez \cite{CG}.  What we are able to prove is the following improvement of regularity result.
\begin{prop}\label{p1.1}
Let $(M,\omega_0)$ be a compact K\"ahler manifold.  Then there exists $\eps_0>0$ depending only on the initial data,  such that for any $\varphi_0\in PSH(M,\omega_0)$ with $||\log\frac{\omega_{\varphi_0}^n}{\omega_0^n}||_0\le \eps_0$,  the following holds for the pseudo Calabi flow starting from $\varphi_0$:
\begin{enumerate}
\item There exists $T_0>0$ depending only on the background metric such that the flow exists on $M\times [0,T_0]$,
\item For any $\eps_1>0$,  one can bound all derivatives of the solution on $M\times (\eps_1,T_0]$.
\end{enumerate}
\end{prop}

As an application of Proposition \ref{p1.1} we show that if the initial data is a small perturbation of a cscK metric in the sense of volume forms,  then the pseudo Calabi flow exists globally.  More precisely,  we have:
\begin{thm}\label{t1.2}
Let $(M,\omega_0,J)$ be a cscK manifold.  Then there exists a constant $\eps_0>0$ small enough,  such that for any initial data $\omega\in [\omega_0]$ with $1-\eps_0\le \frac{\omega^n}{\omega_0^n}\le 1+\eps_0$,  the pseudo Calabi flow starting from $\omega$ exists on $M\times [0,+\infty)$.  Moreover,  the flow converges to a cscK metric (possibly different from $\omega_0$) exponentially fast as $t\rightarrow \infty$.
\end{thm}

If the manifold admits a K\"ahler-Einstein metric and if the class is close to the canonical class,  then we no longer need the closeness of the volume form.  In this case,  we have:
\begin{thm}\label{t1.3}
Let $(M,J,\omega_0)$ be a compact K\"ahler-Einstein manifold with $Aut^0(M,J)=0$.  Then for any $\Gamma>0$,  there exists a neighborhood $\mathcal{U}$ of $[\omega_0]$ in $H^{1,1}(M,\bR)$ such that for any K\"ahler metric $\omega$ whose class is in $\mathcal{U}$ and $\frac{1}{\Gamma}\le \frac{\omega^n}{\omega_0^n}\le \Gamma$,  the pseudo Calabi flow starting from $\omega$ exists on $M\times [0,+\infty)$ and converges to a cscK metric as $t\rightarrow \infty$. 
\end{thm}

Now we explain the organization of this paper.  

In Section 2,  we explain some background and preliminaries.

Section 3 focuses on $C^0$ estimates,  as given by the following diagram.
		\begin{center}
			$||P||_0$ bounded $\ \  \Longleftrightarrow\ \  ||\log\frac{\omega_{\varphi}^n}{\omega_0^n}||_0$ bounded $\ \  \Longrightarrow\ \  ||\varphi ||_{C^{0,\alpha}}$ bounded \\[4mm]
		\end{center} 
Here $P$ is defined in (\ref{PCF-2}).  

Section 4 proves higher order estimates,  which are shown below:
\begin{center}
	\begin{tikzcd}
		\text{$P$ is close to a continuous function}  \arrow[rd] &       &     &    \\
           & {||\log\frac{\omega_{\varphi}^n}{\omega_0^n}||_{W^{1,p}}} \arrow[r] & {||\varphi||_{C^{2,\alpha}}} \arrow[r] & {||\varphi||_{C^{k,\alpha}}} \\
		{||\log\frac{\omega_{\varphi}^n}{\omega_0^n}(0)||<\eps} \arrow[ru]     &     &      &                             
	\end{tikzcd}\\[4mm]
	\end{center}	

Section 5 is devoted to the proof of Theorem \ref{t1.2}.  The key assumption needed is the smallness of $\log\frac{\omega_{\varphi}^n}{\omega_0^n}$.

Section 6 is devoted to the proof of Theorem \ref{t1.3}.  The key assumption needed is the closeness of the function $P$ to a continuous funtion.

\section{preliminary and background}

We denote $\mathcal{H}$ to be the space of K\"ahler potentials,  namely $\mathcal{H}=\{\varphi\in C^{\infty}(M):\omega_0+\sqrt{-1}\partial\bar{\partial}\varphi>0\}$.  One can view $\mathcal{H}$ as an infinite dimensional Riemannian manifold under the metric (see \cite{Dona96},  \cite{Mabuchi} and \cite{semmes}):
\begin{equation}\label{2.1}
(f,g)=\int_Mf\bar{g}\omega_{\varphi}^n,\,\,\varphi\in \mathcal{H},\,\,f,\,g\in T_{\varphi}\mathcal{H}=C^{\infty}(M).
\end{equation}

In the study of cscK metrics,  the $K$-energy plays a crucial role,  which was first defined in terms of its Frechet derivative,  given as:
\begin{equation*}
\delta K(\varphi)=\int_M\delta\varphi(\underline{R}-R(\omega_{\varphi}))\omega_{\varphi}^n.
\end{equation*}
Let $\chi$ be a closed $(1,1)$ form,  we put:
\begin{equation*}
\delta J_{\chi}(\varphi) = \int_M \delta \varphi (tr_\varphi\chi - \underline \chi)\ \omega_\varphi^n.
\end{equation*}

Assuming $J_{\chi}(0)=0$, integrate along the line segment $\varphi_t = t\varphi$, the explicit expression of $J$--functional can be written as

\begin{align}
	J_\chi(\varphi) &= \int_0^1 \frac{d}{dt}J_\chi(t\varphi)\,dt  \notag\\
	&=\int_0^1 \int_M \varphi \big(\chi \wedge \omega_{t\varphi}^{\,n-1} - \underline \chi\, \omega_{t\varphi}^{\,n}\big)\, dt \notag\\
	&=\int_0^1 \int_M \varphi \big( \sum_{j=0}^{n-1} \binom{n-1}{j} t^j (\sqrt{-1}\partial\bar\partial\varphi)^{j} \wedge \chi \wedge \omega^{n-1-j} - \underline \chi\, \omega_{t\varphi}^{\,n}\big)\, dt \notag\\ 
	&= \sum_{j=0}^{n-1} \frac{j+1}{n+1} \int_M \sqrt{-1}\partial\varphi\wedge\bar\partial\varphi \wedge \chi^{\,j} \wedge \omega_\varphi^{\,n-1-j} 
\end{align}\\

It is observed in Chen \cite{chen00} that the $K$-energy admits the following decomposition:
\begin{align}
	K(\varphi) &=\int_M \log\frac{\omega_{\varphi}^n}{\omega_0^n} \omega_\varphi^n +J_{-Ric}(\varphi)
\end{align}
The first term is called Entropy and second term is called Energy which is continuous under uniform convergence of the potential function $\varphi$. 

The Calabi flow can be viewed as the gradient flow of $K$-energy under the Riemmanian structure (\ref{2.1}).  On the other hand,  it has been observed in Chen-Zheng \cite{Chen-Zheng} that pseudo Calabi flow (\ref{PCF}) can also be viewed as a gradient flow of $K$-energy under an alternative Riemannian structure,  given as:
\begin{equation*}
(f,g)=\int_M\nabla_{\varphi}f\cdot\nabla_{\varphi}g\omega_{\varphi}^n,\,\,\,\varphi\in \mathcal{H},\,\,f,\,g\in T_{\varphi}\mathcal{H}=C^{\infty}(M).
\end{equation*}

Throughout this paper,  we are going to denote:
\begin{equation*}
F=\log\frac{\omega_{\varphi}^n}{\omega_0^n}.
\end{equation*}

Note that one has: (denoting $\omega_0=\sqrt{-1}g_{i\bar{j}}dz_i\wedge d\bar{z}_j$)
\begin{equation*}
R(\omega_{\varphi})=-g_{\varphi}^{i\bar{j}}\partial_{i\bar{j}}\log\det(g_{a\bar{b}}+\varphi_{a\bar{b}})=-\Delta_{\varphi}F+tr_{\omega_{\varphi}}Ric(\omega_0).
\end{equation*}
Therefore,  one can re-write the pseudo Calabi flow in the following form:
\begin{equation*}
\begin{split}
\partial_t\varphi&=\Delta_{\varphi}^{-1}(\underline{R}-R(\omega_{\varphi}))=\Delta_{\varphi}^{-1}(\Delta_{\varphi}F+\underline{R}-tr_{\omega_{\varphi}}Ric(\omega_0))=F+\Delta_{\varphi}^{-1}(\underline{R}-tr_{\omega_{\varphi}}Ric(\omega_0))\\
&=F+P,\,\,\Delta_{\varphi}P=\underline{R}-tr_{\omega_{\varphi}}Ric(\omega_0),\,\int_MP\omega_{\varphi}^n=0.
\end{split}
\end{equation*}

\textbf{Acknowledgment:} The second author would like to express gratitude to Professor Claude LeBrun, Xiuxiong Chen,  Kai Zheng, Yulun Xu for inspiring discussions and helpful comments.  This work is partially supported by the Simons foundation Award with ID: 605796.

\section{$C^0$ estimate along Pseudo Calabi flow}

In this section, we show some $C^0$ estimates of the solution of Pseudo Calabi Flow. \\ 

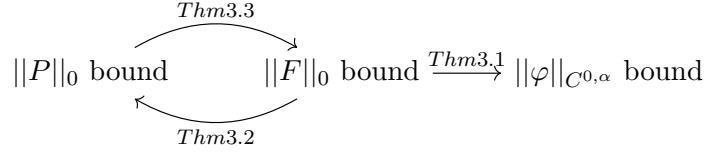
\begin{center}
\begin{tikzcd}
{||P||_0 \text{ bound}} \arrow[r, "Thm \ref{MT22}", bend left] & {||F||_0 \text{ bound}} \arrow[r, "Thm \ref{TMDD}"] \arrow[l, "Thm \ref{MT21}", bend left] & {||\varphi||_{C^{0,\alpha}} \text{ bound}}
\end{tikzcd} \\[6mm]
\end{center}

\begin{thm}[Theorem $A^*$ in \cite{DDGH2008}]\label{TMDD}
	Let $\varphi$ be a solution to $(\omega_0+dd^c\varphi)^n=e^F\omega_0^n$,  with $e^F\in L^p(\omega_0^n)$ for some $1<p<\infty$.  Then for any $0<\alpha<\frac{2p-2}{np+p-1}$,  there exists $C>0$ such that:
	\begin{equation*}
	||\varphi||_{0,\alpha}\le  C,
	\end{equation*}
	where $C$ depends on the background metric,  $||e^F||_{L^p}$,  $p$,  and $n$. \\
\end{thm}

\noindent
Our goal in Section 2 is to establish Theorem \ref{MT21} and \ref{MT22}. \\

\begin{thm}\label{MT21}
	Let $\varphi$ be a smooth $\omega_0$-psh function,  and $P$ solve the following linear equation:
	\begin{align}
		\Delta_{\varphi}P=-tr_{\omega_{\varphi}}(Ric(\omega_0))+\underline{R},\,\,\,\int_MP\omega_{\varphi}^n=0. 
	\end{align}
	Then there exists a constant $C$ such that $||P||_0\le  C$,  where $C$ depends on upper bounds of $||\frac{\omega_{\varphi}^n}{\omega_0^n}||_{L^p}$ for some $p>1$, $||\frac{\omega_{0}^n}{\omega_{\varphi}^n}||_{L^1}$ and Ricci curvature of background metric. \\
\end{thm}

\begin{thm}\label{MT22}
	Let $\varphi(t)$ be a solution of Pseudo Calabi flow on $[0,T)\times M$. If 
	\begin{align}
		\max_{t\in [0,T)} ||P(t)||_0 < V_1 , \, |Ric(\omega_0)|<\Lambda
	\end{align}
	then there exist some constants $\alpha_i$ depending on $V_1, \Lambda$ such that 
	\begin{align}
		||\varphi(t)||_0 \le  ||\varphi(0)||_0 e^{\alpha_1t}+\alpha_2 t
	\end{align}
	and 
	\begin{align}
		||F(t)||_0 \le   ||F(0)||_0 e^{\alpha_3 t} +\alpha_4 t
	\end{align} \\
\end{thm}

\noindent
We start with the following energy bound along PCF. 

 \begin{lem}\label{Lem21}
 	Let $\varphi$ be a solution of PCF on $M\times [0,T]$.
	\begin{enumerate}
	\item There exists $\lambda>0$ large enough,  depending only on the background metric,  such that $t\mapsto K(\varphi(\cdot,t))+\lambda J_{\omega_0}(\varphi(\cdot,t))$ is bounded for $t\in [0,T]$,  where the bound depends on the initial data and  $T$.  
	\item $t\mapsto \int_M|\nabla_{\varphi}\varphi(\cdot,t)|^2_{\omega_{\varphi}}\omega_{\varphi}^n$ is bounded on $[0,T]$,  with a bound depending on the initial data,  the background metric and $T$.\\
	\end{enumerate}
 \end{lem}

\begin{proof}[Proof of Lemma \ref{Lem21}]
	First we observe that $K$ energy is decreasing along the flow,  indeed:
\begin{align}
	\frac{dK(\varphi)}{dt} &=\int_M\partial_t\varphi(\underline{R}-R(\omega_{\varphi}))\omega_{\varphi}^n \notag\\
	& =\int_M(F+P)\Delta_{\varphi}(F+P)\omega_{\varphi}^n =-\int_M|\nabla_{\varphi}(F+P)|^2\omega_{\varphi}^n.
\end{align}\\
On the other hand,
\begin{align}
	\frac{dJ_{\omega_0}(\varphi)}{dt} &=\int_M\partial_t\varphi(tr_{\omega_{\varphi}}\omega_0-n)\omega_{\varphi}^n \notag\\
	& =\int_M(F+P)(-\Delta_{\varphi}\varphi)\omega_{\varphi}^n =\int_M\nabla_{\varphi}(F+P)\cdot \nabla_{\varphi}\varphi\;\omega_{\varphi}^n.
\end{align}\\
Therefore we see that:
\begin{align}
	\frac{d}{dt}\big(K(\varphi)+\lambda J_{\omega_0}(\varphi)\big) \le  \int_M\frac{\lambda^2}{2}|\nabla_{\varphi}\varphi|^2\omega_{\varphi}^n =\frac{\lambda^2}{2}\int_Md^c\varphi\wedge d\varphi\wedge \omega_{\varphi}^{n-1}.
\end{align}\\
Notice that 
\begin{align}
	J_{\omega_0}(\varphi) &= \frac{1}{V}\sum_{j=0}^{n-1} \frac{j+1}{n+1} \int_M i\partial\varphi\wedge\bar\partial\varphi \wedge \omega_0^{\,j} \wedge \omega_\varphi^{\,n-1-j} \notag\\
	&\ge  \frac{1}{V} \frac{1}{n+1} \int_M i\partial\varphi\wedge\bar\partial\varphi \wedge \omega_\varphi^{\,n-1} \notag\\
	& \ge  c_n\int_Md^c\varphi\wedge d\varphi\wedge \omega_{\varphi}^{n-1}
\end{align}
Also it controls the $d_1$-distance (See\cite{Darvas1402}). Therefore, with $\lambda$ sufficiently large, we have:
\begin{align}
	K(\varphi)+\lambda J_{\omega_0}(\varphi)\ge  \frac{\lambda}{2}J_{\omega_0}(\varphi).
\end{align}\\
So we obtain:
\begin{align}
	\frac{d}{dt}\big(K(\varphi)+\lambda J_{\omega_0}(\varphi)\big)	&\le  C\big(K(\varphi)+\lambda J_{\omega_0}(\varphi)\big).
\end{align}\\
Hence by Gronwall lemma,  $K(\varphi)+\lambda J_{\omega_0}(\varphi)$ is bounded on any finite time interval.  Also we see that $\int_M|\nabla_{\varphi}\varphi|^2\omega_{\varphi}^n$ is bounded on any finite time interval. \\
\end{proof}

\begin{proof}[Proof of Theorem \ref{MT21}]
	
	Due to the normalization of $P$ that $\int_MP\omega_{\varphi}^n=0$,  we have:
	\begin{align}
	&\int_M|\nabla_{\varphi}P|^2\omega_{\varphi}^n =\int_MP(-\underline{R}+tr_{\omega_{\varphi}} Ric(\omega_0))\omega_{\varphi}^n \notag \\
	&\le  C_3 \int_M|P|(1+tr_{\omega_{\varphi}}\omega_0)\omega_{\varphi}^n \notag \\
	&\le  C_3 \int_M\sqrt{P^2+1} (1+n-\Delta_{\varphi}\varphi)\omega_{\varphi}^n \notag \\
	&\le  C_4 \int_M|P|\omega_{\varphi}^n+\int_M\frac{P}{\sqrt{P^2+1}} \nabla_{\varphi}P\cdot \nabla_{\varphi}\varphi\;\omega_{\varphi}^n \notag\\
	& \le  C_4 \int_M|P|\omega_{\varphi}^n + \frac{1}{2} \int_M|\nabla_\varphi P|^2\omega_{\varphi}^n + \frac{1}{2} \int_M|\nabla_\varphi \varphi |^2\omega_{\varphi}^n
	\end{align}\\
According to \cite{GPSS2023}, upper bounds of $||\frac{\omega_{\varphi}^n}{\omega_0^n}||_{L^p}$ implies a uniformly Sobolev constant bound 
\begin{align}
	\int_M |P|^2\omega_\varphi^n \le  C_s \int_M|\nabla_\varphi P|^2 \omega_{\varphi}^n 
\end{align}
Hence we have 
\begin{align}
	\int_M |P|^2\omega_\varphi^n \le  2C_4 C_s \int_M|P|\omega_{\varphi}^n +C_s \int_M|\nabla_\varphi \varphi |^2\omega_{\varphi}^n \notag\\
	\le  \frac{1}{2}\int_M |P|^2\omega_\varphi^n + (C_4C_s)^2+ C_s \int_M|\nabla_\varphi \varphi |^2\omega_{\varphi}^n
\end{align}\\
By lemma \ref{Lem21}, $\int_M|\nabla_\varphi \varphi |^2\omega_{\varphi}^n$ is bounded. Hence we get a bound of $\int_M |P|^2\omega_\varphi^n$. \\

\noindent
Once we have this, we next derive the $L^{\infty}$ bound of $P$. To see this, let $\psi$ be the solution to the auxiliary problem:
\begin{align}
(\omega_0+dd^c\psi)^n=\frac{e^F\sqrt{F^2+1}}{\int_Me^F\sqrt{F^2+1}\omega_0^n}\omega_0^n,  \qquad \sup_M\psi=0.
\end{align}\\
For convenient we denote $A=\int_Me^F\sqrt{F^2+1}\omega_0^n$, which can be bounded by an $L^p$ bound of $e^F$. Now we calculate:
\begin{align}
	&\Delta_{\varphi}(P+\varepsilon_0\psi-C_0\varphi) =\underline{R}-tr_{\omega_{\varphi}}Ric(\omega_0)+\varepsilon_0 \Delta_{\varphi}\psi-C_0\Delta_{\varphi}\varphi \notag \\
	&=\underline{R}-tr_{\omega_{\varphi}} Ric(\omega_0)+\varepsilon_0tr_{\omega_{\varphi}}\omega_{\psi} -\varepsilon_0tr_{\omega_{\varphi}}\omega_0-C_0n+C_0tr_{\omega_{\varphi}}\omega_0 
\end{align}\\
Fix a small constant $\varepsilon_0$, and choose $C_0$ large enough so that
\begin{align}
	\frac{1}{2}C_0> |Ric(\omega_0)|+\varepsilon_0 +1
\end{align}\\
Then 
\begin{align}
	&\Delta_{\varphi}(P+\varepsilon_0\psi-C_0\varphi) \ge  -C_5 +\frac{C_0}{2}tr_{\omega_{\varphi}}\omega_0 +\varepsilon_0 n\Big(\frac{\omega_{\psi}^n}{\omega_{\varphi}^n}\Big)^{\frac{1}{n}} \notag \\
	& =-C_5 +\frac{C_0}{2}tr_{\omega_{\varphi}}\omega_0 +\varepsilon_0nA^{-\frac{1}{n}}(F^2+1)^{\frac{1}{2n}}.
\end{align}\\
Assume that $P+\varepsilon_0\psi-C_0\varphi$ achieves maximum at $p_0$, we define $\eta_{p_0}$ being a cut-off function: Let $0<\theta_0<1$ be sufficiently small to be specified later,  $\eta_{p_0}$ to be a smooth function on $M$ such that $\eta_{p_0}(p_0)=1$, and $\eta_{p_0}=1-\theta_0$ on $B_{r_0}^c(p_0)$ with the following conditions 
\begin{align}\label{equacutoff}
	|\nabla\eta_{p_0}| \le  \frac{4\theta_0}{r_0}, \quad |\nabla^2\eta_{p_0}|\le  \frac{4\theta_0}{r_0^2}
\end{align}\\
Then we can calculate:
\begin{align}
	&\Delta_{\varphi}\big(e^{\delta_0(P+\varepsilon_0\psi-C_0(\varphi-\sup_M\varphi))}\eta_{p_0}\big) e^{-\delta_0(P+\varepsilon_0\psi-C_0(\varphi-\sup_M\varphi))} \notag \\
	&\ge \left( \delta_0\Delta_{\varphi}(P+\varepsilon_0\psi-C_0(\varphi-\sup_M\varphi)) +\delta_0^2|\nabla_{\varphi}(P+\varepsilon_0\psi-C_0(\varphi-\sup_M\varphi))|^2\right)\eta_0 \notag \\
	&\quad+2\delta_0\nabla_{\varphi}(P+\varepsilon_0\psi-C_0(\varphi-\sup_M\varphi)) \cdot \nabla_{\varphi}\eta_{p_0}+\Delta_{\varphi}\eta_{p_0}  
\end{align}\\
Notice the following completion of square 
\begin{align}
	&\delta_0^2|\nabla_{\varphi}(P+\varepsilon_0\psi-C_0(\varphi-\sup_M\varphi))|^2 \eta_{p_0} \notag\\
	&\quad + 2\delta_0\nabla_{\varphi}(P+\varepsilon_0\psi-C_0(\varphi-\sup_M\varphi)) \cdot \nabla_{\varphi}\eta_{p_0} + \frac{|\nabla_\varphi\eta_{p_0}|^2}{\eta_{p_0}} \ge  0
\end{align}\\
We then have 
\begin{align}
	&\Delta_{\varphi}\big(e^{\delta_0(P+\varepsilon_0\psi-C_0(\varphi-\sup_M\varphi))}\eta_{p_0}\big) e^{-\delta_0(P+\varepsilon_0\psi-C_0(\varphi-\sup_M\varphi))}\notag\\
	& \ge  \left( -\delta_0C_5 +\frac{C_0\delta_0}{2}tr_{\omega_{\varphi}}\omega_0 +\varepsilon_0n\delta_0A^{-\frac{1}{n}}(F^2+1)^{\frac{1}{2n}}\right)\eta_{p_0} +\Delta_\varphi\eta_{p_0} - \frac{|\nabla_\varphi\eta_{p_0}|^2}{\eta_{p_0}} 
\end{align}\\
Let $\theta_0$ be small enough so that 
\begin{align}
	\Delta_\varphi\eta_{p_0} - \frac{|\nabla_\varphi\eta_{p_0}|^2}{\eta_{p_0}} \le  \frac{4\theta_0}{r_0^2}tr_{\omega_{\varphi}}\omega_0 (1+\frac{1}{1-\theta_0} ) \le  (1-\theta_0) tr_{\omega_{\varphi}}\omega_0
\end{align}\\
Choosing $\delta_0=1$, $C_0$ sufficiently large so that $\frac{C_0\delta_0}{2}(1-\theta_0) \ge  2$, we see that:
\begin{align}
	\Delta_{\varphi}\big(e^{\delta_0(P+\varepsilon_0\psi-C_0(\varphi-\sup_M\varphi))}\eta_{p_0} \big) \ge  e^{\delta_0(P+\varepsilon_0\psi-C_0(\varphi-\sup_M\varphi))} \Big(-C_5+\varepsilon_0n A^{-\frac{1}{n}}(F^2+1)^{\frac{1}{2n}}\Big)\eta_{p_0}
\end{align}\\
Therefore, we can use Aleksandrov-Bakelman-Pucci (ABP) maximum principle (See Lemma 5.5 in \cite{cc1}) to get:
\begin{align}
	&\sup_{B_{r_0}(p_0)}e^{\delta_0(P+\varepsilon_0\psi-C_0(\varphi-\sup_M\varphi))} \le  \sup_{\partial B_{r_0}(p_0)}e^{\delta_0(P+\varepsilon_0\psi-C_0(\varphi-\sup_M\varphi))} \notag \\
	&\quad+C_nr_0\bigg(\int_{B_{r_0}(p_0)}e^{2F}\cdot e^{2n\delta_0(P+\varepsilon_0\psi-C_0(\varphi-\sup_M\varphi))} \cdot\big(-C_5+\varepsilon_0nA^{-\frac{1}{n}}(F^2+1)^{\frac{1}{2n}}\big)_+^{2n}\omega_0^n\bigg)^{\frac{1}{2n}}
\end{align}\\
Where we denote $x_+=\max(0,x)$. Therefore,  the integrand is non-zero only on the domain where  
\begin{align}
	-C_5+\varepsilon_0nA^{-\frac{1}{n}}(F^2+1)^{\frac{1}{2n}} \le  0
\end{align}\\
i.e. 
\begin{align}
	\sqrt{F^2+1}\le  \left( \frac{C_5}{\varepsilon_0 n}\right)^n A
\end{align}\\
We get that $F$ is bounded by $A$ on the set where the integrand is nonzero. Hence
\begin{align}
	&\sup_Me^{P+\varepsilon_0\psi-C_0(\varphi-\sup_M\varphi)} \le  \frac{C_6 r_0}{\theta_0}\bigg(\int_Me^{2n(P+\varepsilon_0\psi-C_0(\varphi-\sup_M\varphi))}\omega_0^n\bigg)^{\frac{1}{2n}} \notag \\
	&\le  \frac{C_6 r_0}{\theta_0}\sup_M\frac{e^{P+\varepsilon_0\psi-C_0(\varphi-\sup_M\varphi)}}{\big(1+(P+\varepsilon_0\psi-C_0(\varphi-\sup_M\varphi))_+\big)^{\frac{1}{2n}}}
	\bigg(\int_M(P+\varepsilon_0\psi-C_0(\varphi-\sup_M\varphi))_+\omega_0^n\bigg)^{\frac{1}{2n}}
\end{align}\\
Since we already know that $P$ has $L^2(\omega_\varphi)$ bound, and $\varphi$ has $\alpha$-invariant bound 
\begin{align}
	\int_M(P+\varepsilon_0\psi-C_0(\varphi-\sup_M\varphi))_+\omega_0^n \le  \int_M \big(|P|^2e^F + e^{-F} + \frac{C_0}{\alpha} e^{\alpha(\sup \varphi -\varphi)}\big) \, \omega_0^n
\end{align}\\
We see that the integral on the right hand side is bounded. Therefore, we get:
\begin{align}
	\sup_Me^{P+\varepsilon_0\psi-C_0(\varphi-\sup_M\varphi)} \le  C_7 \sup_M\frac{e^{P+\varepsilon_0\psi-C_0(\varphi-\sup_M\varphi)}}{\big(1+(P+\varepsilon_0\psi-C_0(\varphi-\sup_M\varphi))_+\big)^{\frac{1}{2n}}}.
\end{align}\\
Therefore we get an upper bound for $P+\varepsilon_0\psi-C_0(\varphi-\sup_M\varphi)$. Running the same argument with $-P+\varepsilon\psi-C_0(\varphi-\sup_M\varphi)$ (In order not to make the process too lengthly, we omit the details), We then get an upper bound for $|P|+\varepsilon_0\psi-C_0(\varphi-\sup_M\varphi)$. When $N_p$ bounded, by Theorem \ref{TMDD}, $\psi, \varphi$ is uniformly bounded, hence $|P|+\varepsilon_0\psi-C_0(\varphi-\sup_M\varphi)$ is bounded, which proves the bound of $||P||_0$. \\ 

\end{proof}

\noindent
Now we proceed to the proof of Theorem \ref{MT22}. 

\begin{proof}[Proof of Theorem \ref{MT22}]
	Let $\nu=1$ or $-1$. The following estimate can be done for both case.  Assume that the function $\nu F-\lambda \varphi$ achieves maximum at $(p_0,t_0)\in M\times [0,T]$. Let $\eta$ be a cut-off function on $M$ as defined in \eqref{equacutoff}.	Let $\delta_0>0$, we compute:
	\begin{align}
		& (\partial_t-\Delta_\varphi)(e^{\delta_0 (\nu F-\lambda \varphi)}\,\eta) e^{-\delta_0 (\nu F-\lambda \varphi)} = \eta \left(\delta_0 (\partial_t-\Delta_\varphi) (\nu F-\lambda \varphi) - \delta_0^2 |\nabla_\varphi (\nu F -\lambda\varphi)|^2 \right) \notag\\
		& \qquad\qquad\qquad\qquad\qquad\qquad\qquad\qquad\qquad + 2\delta_0 \nabla_\varphi(\nu F-\lambda\varphi)\cdot \nabla_\varphi\eta - \Delta_\varphi\eta \notag\\ 
		& \qquad\qquad \le   \eta \delta_0 (\partial_t-\Delta_\varphi) (\nu F-\lambda \varphi) + \frac{|\nabla_\varphi\eta|^2}{\eta} - \Delta_\varphi\eta \notag\\
		& \qquad\qquad \le  \eta \delta_0 \left( - \nu\, tr_\varphi Ric(\omega_0) +\nu \underline R -\lambda(F+P-n+tr_{\omega_{\varphi}}\omega_0 )  \right) + \frac{16\theta_0}{r_0^2} tr_{\omega_{\varphi}}\omega_0 
	\end{align}\\
	We choose 
	\begin{align}
		\lambda>|Ric(\omega_0)|+2,\qquad \theta_0 < \frac{1}{1000} r_0^2\delta_0
	\end{align}
	 then 
	\begin{align}
		(\partial_t-\Delta_\varphi)(e^{\delta_0 (\nu F-\lambda \varphi)}\,\eta) \le  e^{\delta_0 (\nu F-\lambda \varphi)}\,\eta \delta_0 \big( C_6- \lambda(F+P)\big)
	\end{align}\\
	Now we apply the parabolic ABP on $B_{r_0}(p_0)\times [0,t_0]$, and see that: 
	\begin{align}
		& \sup_{B_{r_0}(p_0)\times [0,t_0]}\big(e^{\delta_0(\nu F-\lambda\varphi)}\eta\big)\le  \sup_{\partial B_{r_0}(p_0)\times [0,t_0] \cup B_{r_0}(p_0)\times \{0\}}\big( e^{\delta_0(\nu F-\lambda\varphi)}\eta \big) \notag\\
		& +C_7\bigg(\int_{B_{r_0}(p_0)\times [0,t_0]}e^{\frac{2n+1}{n}F}\cdot e^{(2n+1)\delta_0(\nu F-\lambda\varphi)} (\eta \delta_0)^{2n+1} \big( C_6- \lambda(F+P)\big)^{2n+1}_+ \,\omega_0^n \bigg)^{\frac{1}{2n+1}}
	\end{align}	\\
	The integral of the second line above is on the domain where 
	\begin{align}
		C_6- \lambda(F+P)>0
	\end{align}
	$P$ is bounded, $F$ is also bounded from above on this domain 
	\begin{align}
		F\le  \frac{C_6}{\lambda}+||P||_0
	\end{align}\\
	 We choose
	\begin{align}
		\delta_0 < \min \{\ \frac{1}{2n} ,\ \frac{\alpha \big(M,[\omega_0]\big)}{(2n+1)\lambda}\  \}
	\end{align}
	where $\alpha \big(M,[\omega]\big)$ is the $\alpha$-invariant (See \cite{tian87}). Then we have  
	\begin{align}
		&\int_{B_{r_0}(p_0)\times [0,t_0]}e^{\frac{2n+1}{n}F}\cdot e^{(2n+1)\delta_0(\nu F-\lambda \varphi)} (\eta\delta_0)^{2n+1}(C_6-\lambda(P+F))_+^{2n+1} \omega_0^n \notag \\
		&\le  \int_{B_{r_0}(p_0)\times [0,t_0]}\exp\big((\frac{2n+1}{n}+(2n+1)\delta_0\nu)(\frac{C_6}{\lambda}+||P||_0)\big)\cdot e^{-(2n+1)\delta_0\lambda\varphi}\delta_0^{2n+1}\omega_0^n \notag \\
		&=\exp\big((\frac{2n+1}{n}+(2n+1)\delta_0\nu)(\frac{C_6}{\lambda}+||P||_0)\big)\delta_0^{2n+1} \notag \\
		&\qquad\qquad\qquad\qquad  \cdot\int_{B_{r_0}(p_0)\times [0,t_0]}e^{-(2n+1)\delta_0\lambda(\varphi-\sup_M\varphi(\cdot,t))}\cdot e^{-(2n+1)\delta_0\lambda \sup_M\varphi(\cdot,t)}\omega_0^n \notag \\
		&\le  C_8\int_0^{t_0}e^{-(2n+1)\delta_0\lambda\sup_M\varphi(\cdot,t)}dt.
	\end{align} \\
	The second inequality used the $\alpha$-invariant estimate, as long as we choose $\delta_0$ small enough and $C_8$ depends also on $||P||_0$. Moreover,  our choice of $\delta_0$ guarantees $\frac{2n+1}{n}+(2n+1)\delta_0\nu>0$,  regardless whether $\nu=1$ or $-1$. Therefore we have 
	\begin{align}
		& \sup_{B_{r_0}(p_0)\times [0,t_0]}\big(e^{\delta_0(\nu F-\lambda\varphi)}\eta\big) - \sup_{\partial B_{r_0}(p_0)\times [0,t_0] \cup B_{r_0}(p_0)\times \{0\}}\big( e^{\delta_0(\nu F-\lambda\varphi)}\eta \big) \notag\\
		&\qquad\qquad\qquad\qquad\qquad\qquad\qquad\qquad \le  C_9  \int_{M\times [0,t_0]} e^{-(2n+1)\delta_0\lambda\sup\varphi}\,\omega_0^n
	\end{align}\\
	Next it only remains to estimate $\sup_M\varphi(\cdot,t)$ from below. Define 
	\begin{align}
		I(\varphi)=\frac{1}{n+1}\int_M\varphi\sum_{j=0}^n\omega_{\varphi}^{n-j}\wedge \omega_0^j
	\end{align}\\
	Then we know that:
	\begin{align}
		\frac{d}{dt}I(\varphi)=\int_M\partial_t\varphi\omega_{\varphi}^n=\int_M(F+P)\omega_{\varphi}^n=\int_MF\omega_{\varphi}^n\ge  0.
	\end{align}\\
	The last inequality is due to that $xe^x\ge  e^x-1$. In particular,  $I(\varphi(\cdot,t))\ge  I(\varphi_0)$.  On the other hand,  $I(\varphi(\cdot,t))\le  \sup_M\varphi(\cdot,t)\,\cdot\,$vol($M$). Therefore we have 
	\begin{align}
		\sup_{B_{r_0}(p_0)\times [0,t_0]}\big(e^{\delta_0(\nu F-\lambda\varphi)}\eta\big) - \sup_{\partial B_{r_0}(p_0)\times [0,t_0] \cup B_{r_0}(p_0)\times \{0\}}\big( e^{\delta_0(\nu F-\lambda\varphi)}\eta \big) \le  C_{10} t_0
	\end{align}\\
	which means
	\begin{align}
		||e^{\delta_0(\nu F-\lambda\varphi)}|_{t=t_0} ||_0\le  ||e^{\delta_0(\nu F-\lambda\varphi)}|_{t=0} ||_0 + \frac{C_{10}}{\theta_0} t_0
	\end{align}
	Notice that this inequality hold for both $\nu=1$ or $-1$ and any $t_0$, hence 
	\begin{align}
		||e^{F(t)}||_0 \le  e^{\alpha_1 \varphi(t)} \bigg(||e^{\alpha_2 F(0)}|| + C_{11}t\bigg)
	\end{align}
	Notice that $\log (e^a+b)\le  a+b$ for $a,b\ge  0$, then we have
	\begin{align}
		||F(t)||_0 \le  \alpha_1 ||\varphi(t)||_0 +\alpha_2 ||F(0)||_0 +C_{11}t
	\end{align}
	According to the definition of PDE, 
	\begin{align}
		\partial_t \varphi = F+P \le  \alpha_1 ||\varphi(t)||_0 +\alpha_2 ||F(0)||_0 +C_{11}t + ||P||_0
	\end{align}
	By Gronwall, we have 
	\begin{align}
		||\varphi(t)||_0 \le  ||\varphi(0)||_0 e^{\alpha_3 t} + C_{12}t
	\end{align}
	and 
	\begin{align}
		||F(t)||_0 \le   ||F(0)||_0 e^{\alpha_4 t} + C_{13}t
	\end{align} \\
\end{proof}

\section{Higher order estimate\\}\label{Sec3}

In this section we will estimate the higher order derivatives.  The following diagram shows the outline of regularity improvement.  \\

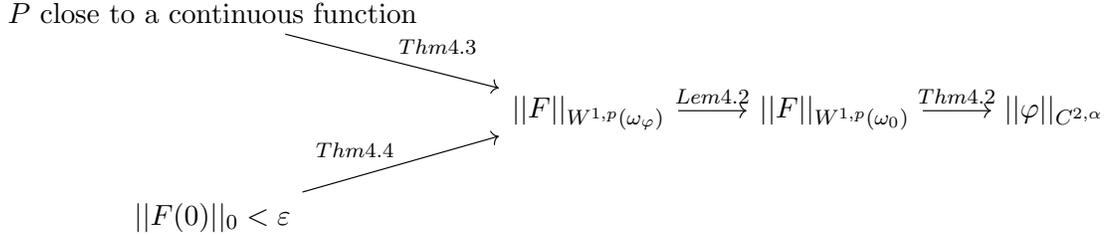
\begin{tikzcd}
\text{$P$ close to a continuous function}  \arrow[rd, "Thm \ref{MT31}"] &     &      &       &      \\
   & {||F||_{W^{1,p}(\omega_\varphi)}} \arrow[r, "Lem \ref{Lem32}"] & {||F||_{W^{1,p}(\omega_0)}} \arrow[r, "Thm \ref{TM32}"] & {||\varphi||_{C^{2,\alpha}}}  \\
{||F(0)||_0<\eps} \arrow[ru, "Thm \ref{MT32}"]    &     &    &    & 
\end{tikzcd} \\

Chen-Zheng \cite{Chen-Zheng} have shown that higher derivative bounds of $\varphi$ follow from $C^{2,\alpha}$ bound. 

\begin{thm}[Theorem 5.37 in \cite{Chen-Zheng}]\label{TM31}
	Let $\varphi$ be a solution of Pseudo Calabi Flow on $[0,T)\times M$. If there exist a constant so that 
	\begin{align}
		\max_{t\in [0,T)} ||\varphi(t)||_{C^{2,\alpha}}< C
	\end{align}
	then for any $\eps,k$, we have 
	\begin{align}
		\max_{t\in [\eps ,T)} ||\varphi(t)||_{C^{k,\alpha}}< C (\eps,k)\\
	\end{align}
\end{thm}
Chen-He reduced the $C^{2,\alpha}$ bound of the K\"ahler potential $\varphi$ to the $W^{1,p}$ bound of the volume form,  with $p>2n$.
\begin{thm}[Section 4 in \cite{ChHe2011}]\label{TM32}
	Let $\varphi$ be a solution of the Monge-Amp\'ere equation on a compact K\"ahler manifold $(M,\omega_0)$ and $C$ be a constant so that 
	\begin{align}\label{equa34}
		(\omega_0+dd^c \varphi)^n=e^F \omega_0, \quad ||F||_{W^{1,p}(\omega_0)}<C , \quad p>2n.
	\end{align}
	then there exist a constant $C_1$ depending on $C,p$ and background metric such that 
	\begin{align}
		||\varphi||_{C^{2,\alpha}} < C. 
	\end{align}\\
\end{thm}

\noindent
To compare two different norms $W^{1,p}(\omega_0)$ and $W^{1,p}(\omega_\varphi)$, we first need to estimate $tr_{\omega_{\varphi}}\omega_0$ in term of $|\nabla_\varphi F|$. \\

\begin{lem}[\textbf{Estimate of $tr_{\omega_{\varphi}}\omega_0$}]\label{Lem31}
	Let $(\varphi,F)$ satisfy the Monge-Amp\'ere equation 
	\begin{align}
		(\omega_0+dd^c \varphi)^n=e^F \omega_0^n
	\end{align}\\
	Fix a constant $V$. If $||F||_0 \le  V$, then there exist a constant $C$ depend on $p,V$, and the background metric $\omega_0$, such that
	\begin{align}
		\int_M (tr_{\omega_{\varphi}}\omega_0)^{p+1}\omega_\varphi^n \le  C\left( \int_M |\nabla_\varphi F|^{2(p+1)}\omega_\varphi^n +1 \right)
	\end{align}\\
\end{lem}

\begin{proof}

Firstly
\begin{align}
	& \Delta_\varphi (e^{-\lambda\varphi-\alpha F}tr_{\omega_{\varphi}}\omega_0)e^{\lambda\varphi+\alpha F} = \left(-\Delta_\varphi (\lambda\varphi+\alpha F)+|\nabla_\varphi (\lambda\varphi+\alpha F)|^2\right)tr_{\omega_{\varphi}}\omega_0  \notag\\
	&\qquad\qquad\qquad\qquad +\Delta_\varphi tr_{\omega_{\varphi}}\omega_0 - 2\nabla_\varphi(\lambda\varphi+\alpha F)\cdot\nabla_\varphi tr_{\omega_{\varphi}}\omega_0
\end{align}\\
Recall 
\begin{align}
	& \Delta_\varphi tr_{\omega_{\varphi}}\omega_0 = \frac{R_{i\bar ik\bar k}}{(1+\varphi_{i\bar i})(1+\varphi_{k\bar k})} + \frac{|\varphi_{i\bar jk}|^2}{(1+\varphi_{i\bar i})(1+\varphi_{j\bar j})^2(1+\varphi_{k\bar k})} -\frac{F_{i\bar i}-Ric_{i\bar i}}{(1+\varphi_{i\bar i})^2} \notag\\
	&\qquad\qquad \ge  -C (tr_{\omega_{\varphi}}\omega_0)^2 + \frac{|\nabla_\varphi tr_{\omega_{\varphi}}\omega_0|^2}{tr_{\omega_{\varphi}}\omega_0}+ \frac{dd^c F\wedge \omega_0\wedge\omega_\varphi^{n-2}}{\omega_\varphi^n}  - \Delta_\varphi F\cdot tr_{\omega_{\varphi}}\omega_0 
\end{align}\\
And notice the following completion of square 
\begin{align}
	\frac{|\nabla_\varphi tr_{\omega_{\varphi}}\omega_0|^2}{tr_{\omega_{\varphi}}\omega_0} - 2\nabla_\varphi(\lambda\varphi+\alpha F)\cdot\nabla_\varphi tr_{\omega_{\varphi}}\omega_0 + |\nabla_\varphi (\lambda\varphi+\alpha F)|^2 tr_{\omega_{\varphi}}\omega_0  \ge  0
\end{align}\\
We then get 
\begin{align}
	& \Delta_\varphi (e^{-\lambda\varphi-\alpha F}tr_{\omega_{\varphi}}\omega_0)e^{\lambda\varphi+\alpha F} \ge  \left(\lambda tr_{\omega_{\varphi}}\omega_0 -\lambda n  \right)tr_{\omega_{\varphi}}\omega_0 \notag\\
	&\qquad -C (tr_{\omega_{\varphi}}\omega_0)^2 + \frac{dd^c F\wedge \omega_0\wedge\omega_\varphi^{n-2}}{\omega_\varphi^n}  - (\alpha+1) \Delta_\varphi F\cdot tr_{\omega_{\varphi}}\omega_0
\end{align}\\
Set $u=e^{-\lambda\varphi-\alpha F}tr_{\omega_{\varphi}}\omega_0$ 
\begin{align}
	\Delta_\varphi u\ge  \big( (\lambda-C)tr_{\omega_{\varphi}}\omega_0 -\lambda n  \big)u  +e^{-\lambda\varphi-\alpha F}\frac{dd^c F\wedge \omega_0\wedge\omega_\varphi^{n-2}}{\omega_\varphi^n} -  (\alpha+1) \Delta_\varphi F\cdot u
\end{align}\\
Let $p>1$,  we have:
\begin{align}
	\Delta_\varphi u^p = p u^{p-1}\Delta_\varphi u+p(p-1)u^{p-2}|\nabla_\varphi u|^2
\end{align}
Integrate with respect to $\omega_\varphi^n$, we have
\begin{align}
	&\int_M p(p-1)u^{p-2}|\nabla_{\varphi}u|^2\omega_{\varphi}^n = - \int_M pu^{p-1}\Delta_\varphi u \omega_\varphi^n  \notag \\
	&\qquad \le  -\int_M pu^p \big( (\lambda-C)tr_{\omega_{\varphi}}\omega_0-\lambda n\big) \omega_\varphi^n + (\alpha+1) \int_M pu^p \Delta_\varphi F \omega_\varphi^n  \notag\\
	&\quad\quad\quad\quad\qquad - \int_M pu^{p-1} e^{-\lambda\varphi-\alpha F}dd^c F\wedge \omega_0\wedge\omega_\varphi^{n-2}
\end{align}
We integrate by part to handle the last two terms
\begin{align}
	& \int_M pu^p \Delta_\varphi F \omega_\varphi^n = -\int_M p^2 u^{p-1}\nabla_\varphi u \cdot\nabla_\varphi F\,\omega_\varphi^n \notag\\
	&\qquad\qquad  \le  \int_M \frac{p(p-1)}{4(\alpha+1)} u^{p-2}|\nabla_{\varphi}u|^2\omega_{\varphi}^n + \int_M 8(\alpha+1)p^2 u^p |\nabla_\varphi F|^2\omega_\varphi^n
\end{align}
Next 
\begin{align}
	&- \int_M pu^{p-1} e^{-\lambda\varphi-\alpha F}dd^c F\wedge \omega_0\wedge\omega_\varphi^{n-2} \notag\\
	&\qquad = \frac{1}{\alpha} \int_M pu^{p-1} e^{-\lambda\varphi-\alpha F}dd^c(-\lambda\varphi-\alpha F)\wedge \omega_0\wedge\omega_\varphi^{n-2} \notag\\
	&\qquad\qquad\qquad  +\frac{\lambda}{\alpha} \int_M pu^{p-1} e^{-\lambda\varphi-\alpha F}dd^c\varphi \wedge \omega_0\wedge\omega_\varphi^{n-2} \notag\\
	&\qquad = -\frac{1}{\alpha} \int_M pu^{p-1} e^{-\lambda\varphi-\alpha F}d(-\lambda\varphi-\alpha F)\wedge d^c(-\lambda\varphi-\alpha F)\wedge \omega_0 \wedge \omega_\varphi^{n-2} \notag\\
	&\quad\quad\qquad -\frac{1}{\alpha} \int_M p(p-1)u^{p-2} e^{-\lambda\varphi-\alpha F}du\wedge d^c(-\lambda\varphi-\alpha F)\wedge \omega_0\wedge\omega_\varphi^{n-2} \notag\\
	&\quad\quad\qquad +\frac{\lambda}{\alpha} \int_M pu^{p-1} e^{-\lambda\varphi-\alpha F}\omega_0\wedge\omega_\varphi^{n-1} - \frac{\lambda}{\alpha} \int_M pu^{p-1} e^{-\lambda\varphi-\alpha F} \omega_0^2 \wedge\omega_\varphi^{n-2} \notag\\
	&\qquad = -\frac{1}{\alpha} \int_M pu^{p-1} e^{-\lambda\varphi-\alpha F}\left( |\nabla_\varphi (-\lambda\varphi-\alpha F)|^2 tr_{\omega_{\varphi}}\omega_0 -\sum_i \frac{|(-\lambda\varphi-\alpha F)_i|^2}{(1+\varphi_{i\bar i})^2} \right)\omega_\varphi^{n} \notag\\
	&\quad -\frac{1}{\alpha} \int_M p(p-1)u^{p-2} e^{-\lambda\varphi-\alpha F} \left( \big( \nabla_\varphi (-\lambda\varphi-\alpha F)\cdot\nabla_\varphi u \big) tr_{\omega_{\varphi}}\omega_0 -\sum_i \frac{(-\lambda\varphi-\alpha F)_i u_{\bar i}}{(1+\varphi_{i\bar i})^2} \right) \omega_\varphi^{n} \notag\\
	&\quad\quad +\frac{\lambda}{\alpha}\int_M pu^p \omega_\varphi^n - \frac{\lambda}{\alpha} \int_M pu^{p-1} e^{-\lambda\varphi-\alpha F} \left( (tr_{\omega_{\varphi}}\omega_0)^2 -\sum_i \frac{1}{(1+\varphi_{i\bar i})^2} \right) \omega_\varphi^{n} \notag\\
	&\qquad \le  \ \int_M \frac{p(p-1)^2}{4\alpha} u^{p-3} e^{-\lambda\varphi-\alpha F} \left( |\nabla_\varphi u|^2 tr_{\omega_{\varphi}}\omega_0 - \frac{|u_i|^2}{(1+\varphi_{i\bar i})^2} \right)\omega_\varphi^n  +\frac{\lambda}{\alpha}\int_M pu^p \omega_\varphi^n \notag\\
	&\qquad \le  \ \int_M \frac{p(p-1)^2}{4\alpha} u^{p-2}  |\nabla_\varphi u|^2 \omega_\varphi^n  +\frac{\lambda}{\alpha}\int_M pu^p \omega_\varphi^n
\end{align}
So we get 
\begin{align}
	&\int_M p(p-1)u^{p-2}|\nabla_{\varphi}u|^2\omega_{\varphi}^n \le  -\int_M pu^p \big( (\lambda-C)tr_{\omega_{\varphi}}\omega_0-\lambda n\big) \omega_\varphi^n + \frac{\lambda}{\alpha}\int_M pu^p \omega_\varphi^n \notag\\
	&\quad\quad + \int_M 8(\alpha+1)p^2 u^p |\nabla_\varphi F|^2\omega_\varphi^n  + \int_M \left( \frac{p(p-1)^2}{4\alpha} + \frac{p(p-1)}{4(\alpha+1)}\right) u^{p-2}  |\nabla_\varphi u|^2 \omega_\varphi^n 
\end{align}
By choosing $\alpha=p$ and $\lambda>2C$ 
\begin{align}
	\frac{\lambda}{2} \int_M pu^p tr_{\omega_{\varphi}}\omega_0 \omega_\varphi^n \le  C(p) \int_M u^p |\nabla_\varphi F|^2 \omega_\varphi^n + \lambda \int_M u^p\omega_\varphi^n \notag\\ 
	\le  c \int_M u^{p+1}\omega_\varphi^n +C(p) \int_M |\nabla_\varphi F|^{2(p+1)}\omega_\varphi^n
\end{align} 
where $c$ is a small constant such that 
\begin{align}
	c < \frac{1}{4} \lambda\cdot \min e^{\lambda\varphi+\alpha F}
\end{align}
which implies 
\begin{align}
	\int_M u^p tr_{\omega_{\varphi}}\omega_0 \omega_\varphi^n \le  \frac{C(p)}{\lambda} \int_M |\nabla_\varphi F|^{2(p+1)}\omega_\varphi^n
\end{align}
equivalently 
\begin{align}
	\int_M (tr_{\omega_{\varphi}}\omega_0)^{p+1}\omega_\varphi^n \le  C(p,||F||_0)\left( \int_M |\nabla_\varphi F|^{2(p+1)}\omega_\varphi^n +1 \right)
\end{align}\\
\end{proof}

\begin{cor}\label{Lem32}
	Let $\varphi$ be a solution of the Monge-Amp\'ere equation on a compact K\"ahler manifold $(M,\omega_0)$ with $||F||_0$ uniformly bounded  
	\begin{align}
		(\omega_0+dd^c \varphi)^n=e^F \omega_0^n, \quad ||F||_0< V
	\end{align}
	Let $p\in (1,\infty]$ be a constant, then there exist $C_1, C_2,q$ depending on $V,p,n$ and background metric such that 
	\begin{align}
		||F||_{W^{1,p}(\omega_0)} < C_1 ||F||_{W^{1,q}(\omega_\varphi)} +C_2
	\end{align}\\
\end{cor}

\begin{proof}
	In local normal coordinate respect to $\omega_0$
	\begin{align}
		F_{i}F_{\bar i} \le  (1+\varphi_{i\bar i})^2 + \bigg( \frac{1}{1+\varphi_{i\bar i}} F_{i}F_{\bar i} \bigg)^2
	\end{align}
	Hence 
	\begin{align}
		|\nabla F|^2\le  C (tr_\varphi \omega_0)^{2n-2} + |\nabla_\varphi F|^4
	\end{align} 
	Integral on manifold  and notice that $e^F$ bounded, we have 
	\begin{align}
		& \int_M |\nabla F|^{2p}\omega_0^n \le  C \bigg( \int_M (tr_\varphi \omega_0)^{2p(n-1)}\omega_\varphi^n + \int_M |\nabla_\varphi F|^{4p}\, \omega_\varphi^n \bigg) \notag\\
		& \qquad \le  C(p,||F||_0)\left( \int_M |\nabla_\varphi F|^{4p(n-1)}\omega_\varphi^n +  \int_M |\nabla_\varphi F|^{4p}\omega_\varphi^n +1 \right)
	\end{align}
\end{proof}

\noindent
By Corollary \ref{Lem32}, to get higher order estimate, we just need to estimate $||F||_{W^{1,p}(\omega_\varphi)}$. In the following subsections, we estimate $||F||_{W^{1,p}(\omega_\varphi)}$ in two cases.
	\begin{itemize}
		\item when $P$ is close to a continuous function (Definition \ref{Def31}). 
		\item The initial volume form $|F|$ is small.
	\end{itemize}

\subsection{When $P$ is close to a continuous function}

\begin{defn}\label{Def31}
	Let $\delta_*>0$. Let P be a function defined on $M\times [0,T)$. We say that $P$ is $\delta_*$-- close to a continuous function with modulus $\omega(r)$, if for any $t\in [0,T)$, there exists a continuous function $\tilde{P}(t)$ with modulus of continuity $\omega(r)$, such that $|P(t)-\tilde{P}(t)|<\delta_*$. \\
\end{defn}
\noindent
This following Proposition is a standard result in geometric analysis.

\begin{prop}[Geodesic mollification on a compact manifold]\label{prop:geodesic-mollifier}
Let $(M,g)$ be a smooth compact Riemannian manifold of dimension $n$ with injective radius $r_{inj}(M,g)>0$. 
Let $f\in C^0(M)$ be a bounded, uniformly continuous function with modulus of continuity $\omega:[0,\infty)\to[0,\infty)$, that is,
\begin{align}
	|f(x)-f(y)|\le  \omega\!\big(d_g(x,y)\big)\quad\forall x,y\in M,  \qquad \omega(r)\to 0\text{ as }r\to0.
\end{align}
Then for each $0<r<\tfrac12 r_{inj}(M,g)$ one can construct a smooth function $f_r\in C^\infty(M)$ (the \emph{geodesic mollification of $f$ at scale $r$}) satisfying:
\begin{align}
	\|f_r-f\|_{L^\infty(M)} &\le  \omega(r), \\
	 \|\nabla f_r\|_{L^\infty(M)} &\le  \frac{C}{r}\|f\|_{L^\infty(M)}, \\
	\|\nabla^2 f_r\|_{L^\infty(M)} &\le  \frac{C}{r^{2}}\|f\|_{L^\infty(M)}, 
\end{align}
where the constant $C>0$ depends only on $(M,g)$ and the dimension $n$. 
Moreover, $f_r\to f$ uniformly as $r\to0$, and $f_r$ has modulus of continuity at most $2\,\omega(r)$. \\
\end{prop}
\noindent
By this result we have the following Corollary. \\
\begin{cor}\label{Cor30}
	If a function $P$ is $\delta_*$--close to a continuous function with modulus $\omega(r)$ as in Definition \ref{Def31}, then there is a $C^2$ function  $\widetilde P$ and constant $C$ such that for any $r>0$ (Small than injective radius)
	\begin{align}
		|P-\widetilde P|<\delta_*+\omega(r),\ \ |D \widetilde P|<\frac{C}{r},\ \ |D^2\widetilde P|<\frac{C}{r^2}
	\end{align} \\ 
\end{cor}

\noindent
The goal of this subsection is to prove the following theorem. \\

\begin{thm}\label{MT31}
	Let $\varphi$ be a solution to the PCF on $M\times [0,T)$,  such that for some $C_1>0$:
	\begin{align}
		\max_{t\in [0,T)}||F(\cdot,t)||_0\le  C_1.
	\end{align}
	Let $p>1$,  then there exists $\delta_*>0$ depending only on $p$,  the background metric,  $C_1$,  such that if we additionally assume that $P(\cdot,t)$ is $\delta_*$-close to a continuous function with modulus $\omega(r)$,  uniformly for $t\in [0,T)$,   then for any $\eps >0$,  one has:
	\begin{align}
		\max_{t\in [\eps ,T)}||\nabla_{\varphi} F(\cdot,t)||_{L^p(\omega_{\varphi}^n)}\le  C.
	\end{align}
	Here $C$ depends only on background metric,  $p$,  $n$,  $C_1$,  $\eps$,  and the modulus $\omega(r)$.\\
\end{thm}

\begin{lem}\label{l4.6N}
Let $\varphi$ be a solution to PCF on $M\times (0,T)$ with $T<+\infty$.  Assume that there exists $p_n>0$ large enough (depending only on $n$,  such that for some $C>0$,  
\begin{align}
	\sup_{t\in (0,T)}||F(\cdot,t)||_0\le  C,\,\,\,\sup_{t\in (0,T)}||\nabla_{\varphi}F(\cdot,t)||_{L^{p_n}(\omega_{\varphi}^n)}\le  C, 
\end{align}
then all the derivatives of $\varphi$ can be estimate for $t\in[\eps ,T)$ for any $\eps >0$.  In particular,  the PCF can be extended beyond $T$. 
\end{lem}

\begin{proof}
	By Lemma \ref{Lem31} and Corollary \ref{Lem32}, we get 
	\begin{align}
		\sup_{t\in (0,T)}||\nabla F(\cdot,t)||_{L^{q_n}(\omega_{0}^n)}\le  C
	\end{align}
	Then by Theorem \ref{TM31} and \ref{TM32}, we have upper bound of $||\varphi||_{k,\alpha}$ for any $k$ on $[\eps,T)$. Hence we can extract a sequence $\varphi(t_i)$ smoothly converge to a function, denote as $\varphi(T)$. By Theorem \ref{t1.1},  part (1),  there exists a short time solution starting at $\varphi(T)$.  Therefore,  we get a solution to the pseudo Calabi flow which exists on $M\times [0,T+\eps_0)$ for some $\eps_0>0$.\\
\end{proof}

\begin{cor}\label{Cor31}
	 Let $\varphi$ be a solution to PCF on $M\times (0,T)$ such that for some $C_1>0$, $\sup_{t\in (0,T)}||F(\cdot,t)||_0\le C_1$,  then there exists $\delta_*>0$ depending only on the background metric and $C_1$ such that if $P$ is $\delta_*$-close to a continuous function with modulus $\omega(r)$,  then one can bound all derivatives of $\varphi$ for $t\in (\eps ,T)$ for any $\eps >0$,  and the PCF can be extended beyond $T$.\\ 
\end{cor}

As a consequence of Theorem \ref{MT22} (a $C^0$ bound of $P$ implies a $C^0$ bound of $F$) as well as Proposition \label{prop:geodesic-mollifier},  we get:

\begin{cor}\label{Cor32}
	Let $\varphi$ be a solution of pseudo Calabi Flow on $M\times [0,T]$.  Assume that $\sup_{t\in [0,T]}||P(\cdot,t)||_{C^{0,\alpha}(M)}<\infty$,  then the pseudo Calabi flow can be extended beyond $T$.
\end{cor}

\noindent
The rest of this subsection is devoted to the proof of Theorem \ref{MT31}. We start the estimate from the first order estimate of potential function. \\

\begin{lem}[\textbf{Estimate of $|\nabla\varphi|^2$}]\label{Lem311}
		Let $(\varphi,F,P)$ be a solution of Pseudo Calabi Flow on $[0,T)\times M$ satisfying the conditions in Theorem \ref{MT31}, then for any $p>1$,  there exists a constant $C$ depending only on $||F||_0, p, T$ and the background metric,  such that 
	\begin{align}
		 \sup_{t\in (0,T)} t^{1+np} \int_M |\nabla\varphi|^{2p}\omega_0^n \le  C\big(||F||_0,p,T\big)
	\end{align} \\
\end{lem}

\begin{lem}[Differential inequality for Lemma \ref{Lem311}]\label{Lem312}
	Denote $H(F,\varphi)=-\lambda \varphi+\varphi^2+\delta_1F^2$. For sufficiently large $\lambda, K$ and small $\delta_1$ (depending on background metric and $C^0$ bound of $F$), we have a point-wise estimate 
	\begin{align}\label{K1}
	&(\partial_t-\Delta_{\varphi})(e^{H(F,\varphi)}(|\nabla\varphi|^2+K))\cdot e^{-H(F,\varphi)}  \le  2\nabla P\cdot\nabla\varphi -\frac{1}{2}(n+\Delta\varphi)       \notag \\
	&\quad\quad\quad\quad\quad\quad +(|\nabla\varphi|^2+K)\bigg( C -\frac{\lambda}{4}tr_{\omega_{\varphi}}\omega_0 -|\nabla_\varphi\varphi|^2 -\delta_1|\nabla\varphi F|^2 - c|\nabla\varphi|^{\frac{2}{n}}\bigg)  
\end{align} \\
where $C$ is a constant also depending on background metric and $C^0$ bound of $F$. \\
\end{lem}

\begin{proof}[Proof of Lemma \ref{Lem311}]
Set $u=e^{H(F,\varphi)}(|\nabla\varphi|^2+K)$ and consider the following equality for very large $p$
\begin{align}
	(\Delta_\varphi-\partial_t)(\frac{1}{p}u^p) = u^{p-1}(\Delta_\varphi-\partial_t) u + (p-1)u^{p-2}|\nabla_\varphi u|^2_\varphi
\end{align}\\
Integral with respect to $ \omega_\varphi^n = e^{F} \omega_0^n $, we have
\begin{align}
	\int_M (p-1) u^{p-2} |\nabla_\varphi u|^2_\varphi e^{F} \omega_0^n
	= \int_M u^{p-1} (\partial_t - \Delta_\varphi) u \, e^{F} \omega_0^n 
	- \frac{1}{p} \int_M \partial_t u^p \, e^{F} \omega_0^n
\end{align}\\
We estimate the second term of right hand side by the following:
\begin{align}
	& -\int_M \partial_t u^p \, e^{F} \omega_0^n = \int_M u^p \, \partial_t e^{F} - \partial_t(u^p e^{F}) \, \omega_0^n \notag \\
	&\quad =  \int_M u^p (\Delta_\varphi F - tr_{\omega_{\varphi}} Ric(\omega_0) + \underline{R}) e^{F} \omega_0^n - \partial_t\big( \int_M u^p e^{F} \omega_0^n \big)
\end{align}\\
We use integration by parts to handle the term $ \Delta_\varphi F $
\begin{align}
	\int_M u^p \Delta_\varphi F \, e^{F} \omega_0^n = - \int_M p u^{p-1} \langle \nabla_\varphi F, \nabla_\varphi u \rangle_\varphi \, e^{F} \omega_0^n
\end{align}\\
Putting those together, we get:
\begin{align}
	&\int_M (p-1) u^{p-2} |\nabla_\varphi u|^2_\varphi e^{F} \omega_0^n \notag\\
	&\quad\quad =  \int_M u^{p-1} (\partial_t - \Delta_\varphi) u \, e^{F} \omega_0^n + \frac{1}{p}  \int_M \left( u^p \, \partial_t e^{F} - \partial_t(u^p e^{F}) \right) \omega_0^n  \notag \\
	&\quad\quad =  \int_M u^{p-1} \left\{ (\partial_t - \Delta_\varphi) u - \langle \nabla_\varphi F, \nabla_\varphi u \rangle_\varphi \right\} e^{F} \omega_0^n \notag\\
	& \quad\quad\quad\quad + \frac{1}{p}   \int_M u^p (-tr_{\omega_{\varphi }}Ric(\omega_0) + \underline{R}) e^{F} \omega_0^n - \frac{1}{p} \left(\partial_t \int_M u^p e^{F} \omega_0^n  \right)
\end{align}	\\
By using AM-GM inequality 
\begin{align}
	\big|u^{p-1}\langle \nabla_\varphi F_\varphi, \nabla_\varphi u \rangle_\varphi\big| \le  \frac{p-1}{2} u^{p-2}|\nabla_\varphi u|^2+ \frac{1}{2(p-1)}u^p|\nabla_\varphi F|^2
\end{align}
we get 
\begin{align}\label{K3}
	& \partial_t\big( \int_M u^p\omega_\varphi^n\big)\le  - \int_M \frac{p(p-1)}{2} u^{p-2} |\nabla_\varphi u|^2_\varphi\omega_\varphi^n \notag\\
	&\quad\quad\quad +\int_M pu^{p-1}\big[(\partial_t-\Delta_\varphi)u+  \frac{1}{2(p-1)}|\nabla_\varphi F|^2 u  \big]\omega_\varphi^n \notag\\
	&\quad\quad\quad +\int_M u^p(|Ric(\omega_0) |tr_{\omega_{\varphi}}\omega_0+\underline R)\omega_\varphi^n
\end{align}
Insert the inequality (\ref{K1}) here, we have 
\begin{align}
	& \partial_t\big( \int_M u^p\omega_\varphi^n\big)\le  - \int_M \frac{p(p-1)}{2} u^{p-2} |\nabla_\varphi u|^2_\varphi\omega_\varphi^n +\int_M u^p(|Ric(\omega_0)|tr_{\omega_{\varphi}}\omega_0+\underline R)\omega_\varphi^n \notag\\
	&\quad +\int_M pu^{p}\bigg( C -\frac{\lambda}{4}tr_{\omega_{\varphi}}\omega_0 -|\nabla_\varphi\varphi|^2 -\delta_1|\nabla_\varphi F|^2 - c|\nabla\varphi|^{\frac{2}{n}} +  \frac{1}{2(p-1)}|\nabla_\varphi F|^2   \bigg)\omega_\varphi^n \notag\\
	& \quad\quad\quad\quad\qquad +\int_M pu^{p-1}e^{H(F,\varphi)}\big[2\nabla P\cdot\nabla\varphi -\frac{1}{2}(n+\Delta\varphi)\big]\omega_\varphi^n
\end{align}
We can also assume $\lambda>8|Ric(\omega_0)|$ because we just need $\lambda$ sufficiently large in our calculations. We also choose $p$ large enough that 
\begin{align}
	\frac{\delta_1}{2}>\frac{1}{2(p-1)}
\end{align}
Then we will get 
\begin{align}\label{K2}
	& \partial_t\big( \int_M u^p\omega_\varphi^n\big)\le  - \int_M \frac{p(p-1)}{2} u^{p-2} |\nabla_\varphi u|^2_\varphi\omega_\varphi^n \notag\\
	&\qquad\qquad +\int_M pu^{p-1}e^{H(F,\varphi)}\big[2\nabla P\cdot\nabla\varphi -\frac{1}{2}(n+\Delta\varphi)\big]\omega_\varphi^n \notag\\
	&\qquad\qquad\quad +\int_M pu^{p} \big(C_1 -\frac{\lambda}{8}tr_{\omega_{\varphi}}\omega_0 -|\nabla_\varphi\varphi|^2  -\frac{\delta_1}{2} |\nabla_\varphi F|^2 - c|\nabla\varphi|^{\frac{2}{n}}  \big) \omega_\varphi^n
\end{align}\\
We need to use the condition that $P$ is $\delta_*$-close to a continuous function with modulus $\omega(r)$.  Let $\tilde{P}$ be the $C^2$ function given by Corollary \ref{Cor30}.  
\begin{align}
 \int_M 2p u^{p-1} e^{H} \nabla \varphi \cdot \nabla P \, \omega_{\varphi}^n = \int_M 2p u^{p-1} e^{H} \nabla \varphi \cdot \nabla \tilde{P} \, \omega_{\varphi}^n + \int_M 2p u^{p-1} e^{H} \nabla \varphi \cdot \nabla(P - \tilde{P}) \, \omega_{\varphi}^n
\end{align}\\
In the above, we have:
\begin{align}
	\int_M 2p u^{p-1} e^{H} \nabla \varphi \cdot \nabla \tilde{P} \, \omega_{\varphi}^n \le  \int_M C_2 p u^{p - \frac{1}{2}} |\nabla \tilde{P}| \, \omega_{\varphi}^n \le  \int_M p u^{p - \frac{1}{2}} \frac{C_{3}}{r} \, \omega_{\varphi}^n.
\end{align}\\
Moreover,
\begin{align}
	&\int_M 2p u^{p-1} e^{H} \nabla \varphi \cdot \nabla(P - \tilde{P}) \, \omega_{\varphi}^n = \int_M 2p u^{p-1} e^{H} \nabla \varphi \cdot \nabla(P - \tilde{P}) e^F \omega_0^n \notag \\ 
	&\quad = -\int_M 2p(p - 1) u^{p - 2} e^{H} \nabla u \cdot \nabla \varphi (P - \tilde{P}) \, \omega_{\varphi}^n -\int_M 2p u^{p - 1} e^{H} \Delta \varphi (P - \tilde{P}) \, \omega_{\varphi}^n \notag \\ 
	&\quad\quad\qquad -\int_M 2p u^{p - 1} e^{H} \nabla \varphi \cdot \nabla (H+F)\ (P - \tilde{P}) \, \omega_{\varphi}^n
\end{align}\\
Then we can estimate as follows:
\begin{align}
	&-\int_M 2p(p - 1) u^{p - 2} e^{H} \nabla u \cdot \nabla \varphi (P - \tilde{P}) \, \omega_{\varphi}^n \le  \int_M 2p(p - 1)(\delta_* + \omega(r)) u^{p - 2}e^{H} |\nabla u||\nabla \varphi| \, \omega_{\varphi}^n \notag \\
	&\quad\quad \le  \int_M 2p(p - 1)(\delta_* + \omega(r)) u^{p - 2}e^{H} |\nabla_{\varphi} u| (n + \Delta \varphi)^{1/2} |\nabla \varphi| \, \omega_{\varphi}^n \notag \\
	&\quad\quad \le  \frac{p(p - 1)}{4} \int_M u^{p - 2} |\nabla_{\varphi} u|^2 \, \omega_{\varphi}^n + 4 p(p - 1)(\delta_* + \omega(r))^2 \int_M u^{p - 1} e^{H} (n + \Delta \varphi) \, \omega_{\varphi}^n.
\end{align}
Next
\begin{align}
	-\int_M 2p u^{p-1}e^{H} \Delta \varphi (P - \tilde{P}) \, \omega_{\varphi}^n \le  \int_M 2p(\delta_* + \omega(r)) u^{p - 1}e^H (n + \Delta \varphi) \, \omega_{\varphi}^n.
\end{align}\\
Finally,
\begin{align}
	&-\int_M 2p u^{p - 1} e^{H} \nabla \varphi \cdot \nabla (H+F)\ (P - \tilde{P}) \, \omega_{\varphi}^n \notag\\
	&\qquad\qquad \le  \int_M 2p(\delta_* + \omega(r))u^{p-1}e^H |\nabla\varphi|(n+\Delta\varphi)^{1/2}|\nabla_\varphi (H+F)|\omega_\varphi^n      \notag\\       
	&\qquad\qquad \le  \int_M 2p(\delta_* + \omega(r)) u^{p - \frac{1}{2}} |\nabla_{\varphi}(F+H)| e^{H/2} (n + \Delta \varphi)^{1/2} \, \omega_{\varphi}^n \notag \\
	&\qquad\qquad \le  \frac{p}{4} \int_M u^{p - 1}e^H (n + \Delta \varphi) \, \omega_{\varphi}^n + \int_M 4p(\delta_* + \omega(r))^2 u^p |\nabla_{\varphi} (F+H)|^2 \, \omega_{\varphi}^n
\end{align}\\
In summary, one has:
\begin{align}
	&\int_M 2p u^{p - 1}e^H \nabla \varphi \cdot \nabla P \, \omega_{\varphi}^n \le  \frac{p(p - 1)}{4} \int_M u^{p - 2} |\nabla_{\varphi} u|^2 \, \omega_{\varphi}^n + \int_M p u^{p - \frac{1}{2}} \frac{C_{3}}{r} \, \omega_{\varphi}^n \notag \\
	&\quad + \left( 4(p - 1)(\delta_* + \omega(r))^2 + 2(\delta_* + \omega(r))+\frac{1}{4} \right) \int_M p u^{p - 1} e^H (n + \Delta \varphi) \, \omega_{\varphi}^n \notag\\
	&\quad + \int_M 4p(\delta_* + \omega(r))^2 u^p |\nabla_{\varphi} (F+H)|^2 \, \omega_{\varphi}^n
\end{align} \\
Plug this back into (\ref{K2}), one has:
\begin{align}
	&\partial_t\big( \int_M u^p\omega_\varphi^n\big)\le  -\frac{1}{4} p(p-1)\int_M  u^{p-2} |\nabla_\varphi u|^2_\varphi\omega_\varphi^n + \int_M p u^{p - \frac{1}{2}} \frac{C_{3}}{r} \, \omega_{\varphi}^n  \notag\\
	&\quad + \left( 4(p - 1)(\delta_* + \omega(r))^2 + 2(\delta_* + \omega(r))-\frac{1}{4} \right) \int_M p u^{p - 1} e^H (n + \Delta \varphi) \, \omega_{\varphi}^n \notag\\
	&\quad +\int_M pu^{p} \bigg(C_1 -\frac{\lambda}{8}tr_{\omega_{\varphi}}\omega_0 -|\nabla_\varphi\varphi|^2  -\frac{\delta_1}{2} |\nabla_\varphi F|^2 - c|\nabla\varphi|^{\frac{2}{n}}  \notag\\
	&\qquad\qquad\qquad\qquad\qquad\qquad\qquad\qquad\qquad\qquad  +4(\delta_* + \omega(r))^2 |\nabla_{\varphi} (F+H)|^2  \bigg) \omega_\varphi^n 
\end{align}\\
We can choose $r$ sufficiently small, and assume $\delta_*$ small enough, so that:
\begin{align}
	4(p - 1)(\delta_* + \omega(r))^2 + 2(\delta_* + \omega(r)) < \frac{1}{4}
\end{align}
and 
\begin{align}
	4(\delta_* + \omega(r))^2 |\nabla_{\varphi} (F+H)|^2 < |\nabla_\varphi\varphi|^2 + \frac{\delta_1}{2} |\nabla_\varphi F|^2 
\end{align}
then we can get rid of some terms and get 
\begin{align}
	&\partial_t\big( \int_M u^p\omega_\varphi^n\big)\le   \int_M p u^{p-\frac{1}{2}} \frac{C_{3}}{r} \, \omega_{\varphi}^n +\int_M pu^{p} \big(C_1 - c|\nabla\varphi|^{\frac{2}{n}} \big) \omega_\varphi^n 
\end{align}\\
By Young's inequality
\begin{align}
	ab\le  \frac{1}{\gamma}a^\gamma+\frac{\gamma-1}{\gamma}b^{\frac{\gamma}{\gamma-1}}
\end{align}
for $a,b>0$ and $\gamma>1$. We can let $a=u^p$ and $\gamma=1+\frac{1}{pn}$, then choose a sufficiently large number $b$. We will get 
\begin{align}
	C_1 p u^p\le  \frac{c}{4}u^{p+\frac{1}{n}}+C_4 
\end{align}\\
Similarly  
\begin{align}
	C_3 p \frac{1}{r} u^{p-\frac{1}{2}}\le  \frac{c}{4}u^{p+\frac{1}{n}}+C_5
\end{align}\\
Where $C_4, C_5$ depends on $p$. We then obtain, with \( p \) large enough:
\begin{align}
	\partial_t\big( \int_M u^p \omega_{\varphi}^n \big) \le  - \int_M \frac{pc}{2} u^{p + \frac{1}{n}} \, \omega_{\varphi}^n + C_{6}(p)
\end{align}\\
Let \( \lambda_p > 0 \), we then compute:
\begin{align}
	\partial_t \left( t^{\lambda_p} \int_M u^p \, \omega_{\varphi}^n \right) \le  \lambda_p t^{\lambda_p - 1} \int_M u^p \, \omega_{\varphi}^n - \frac{pc}{2} t^{\lambda_p} \int_M u^{p + \frac{1}{n}} \, \omega_{\varphi}^n + C_{6}(p) t^{\lambda_p} \notag \\
	\le  \lambda_p t^{\lambda_p - 1} \int_M u^p \, \omega_{\varphi}^n - \frac{pc}{2} \left( t^{\lambda_p \frac{np}{1 + np}} \int_M u^p \, \omega_{\varphi}^n \right)^{1 + \frac{1}{np}}  + C_{6}(p) t^{\lambda_p}.
\end{align}\\
We wish to choose \( \lambda_p = 1 + np \) so that \( \lambda_p - 1 = \lambda_p \frac{np}{1 + np} \), and we now have:
\begin{align}
	\partial_t \left( t^{\lambda_p} \int_M u^p \, \omega_{\varphi}^n \right) \le  \lambda_p \left( t^{np} \int_M u^p \, \omega_{\varphi}^n \right) - \frac{pc}{2} \left( t^{np} \int_M u^p \, \omega_{\varphi}^n \right)^{1 + \frac{1}{np}} + C_{6}(p) t^{\lambda_p}.
\end{align}\\
Note that the right-hand side above is bounded from above on any finite time interval. We integrate in \( t \), and see that \( t^{\lambda_p} \int_M u^p \, \omega_{\varphi}^n \) is bounded on a finite time interval, which depends on $||F||_0, p$ and background metric. Specifically 
\begin{align}
	\sup_{t\in (0,T)} t^{\lambda_p} \int_M |\nabla\varphi|^{2p}\omega_0^n \le  C\big(||F||_0,p,T\big) 
\end{align} \\
\end{proof}

\begin{proof}[Proof of Lemma \ref{Lem312}]
First,  we compute: \\
\begin{align}
	&(\partial_t-\Delta_{\varphi})(e^{H(F,\varphi)}(|\nabla\varphi|^2+K))=e^{H(F,\varphi)}\bigg( (\partial_t-\Delta_{\varphi})H(F,\varphi)-|\nabla_{\varphi}H|^2|\bigg) (|\nabla\varphi|^2+K)\notag \\ 
	&\quad\quad\quad\quad\quad\quad\quad\quad +e^{H(F,\varphi)}(\partial_t-\Delta_{\varphi})|\nabla\varphi|^2 -2\,\nabla_{\varphi}e^{H(F,\varphi)} \cdot \nabla_{\varphi}|\nabla\varphi|^2. 
\end{align}\\
We compute this expression term by term under the choice of normal coordinates,  namely $g_{i\bar{j}}(p)=\delta_{ij},\,\nabla g_{i\bar{j}}(p)=0, \,\varphi_{i\bar{j}}(p)=\varphi_{i\bar{i}}\delta_{ij}$. We have the following: 
\begin{align}
	&(\partial_t-\Delta_{\varphi})H(F,\varphi)= (2\varphi-\lambda)(F+P-n+tr_{\omega_{\varphi}}\omega_0) \notag\\
	&\quad\quad\quad\quad +2\delta_1F(\underline R-tr_{\omega_{\varphi}} Ric(\omega_0))-2|\nabla_\varphi\varphi|^2-2\delta_1|\nabla_\varphi F|^2 \notag \\
	&\le  (\lambda+1)C_1 -(\lambda- C_1-\delta_1C_1)tr_{\omega_{\varphi}}\omega_0 -2|\nabla_\varphi\varphi|^2-2\delta_1|\nabla_\varphi F|^2
\end{align}
where $C_1$ depends on the background metric and the bound of $||F||_0,||P||_0$. We need a higher exponential of $|\nabla\varphi|$ for later use, so when we notice this: 
\begin{align}
	tr_{\omega_{\varphi}}\omega_0\ge  e^{-\frac{F}{n-1}}(n+\Delta\varphi)^{\frac{1}{n-1}}
\end{align}
we will have:
\begin{align}
	\frac{\lambda}{2}tr_{\omega_{\varphi}}\omega_0+\frac{1}{2} |\nabla_{\varphi}\varphi|^2\ge  c|\nabla\varphi|^{\frac{2}{n}}
\end{align}
for a small constant $c$. So we can have 
\begin{align}
	(\partial_t-\Delta_{\varphi})H(F,\varphi) \le  (\lambda+1)C_1 -(\frac{\lambda}{2} - C_1-\delta_1C_1)tr_{\omega_{\varphi}}\omega_0 \notag\\
	 -\frac{3}{2} |\nabla_\varphi\varphi|^2-2\delta_1|\nabla_\varphi F|^2 - c|\nabla\varphi|^{\frac{2}{n}}
\end{align}\\
Then we compute the second term 
\begin{align}
	&(\partial_t-\Delta_{\varphi})(|\nabla\varphi|^2) = \varphi_k(F+P)_{\bar k}+\varphi_{\bar k}(F+P)_k \notag\\
	&\quad\quad\quad - \frac{1}{1+\varphi_{i\bar i}}(R_{k\bar l i\bar i}\varphi_{k}\varphi_{\bar l}+|\varphi_{ik}|^2+|\varphi_{i\bar k}|^2+\varphi_k\varphi_{\bar k i\bar i}+\varphi_{\bar k}\varphi_{k i\bar i})
\end{align}\\
By differentiating the following equation with $z_k$
	\begin{align}
		F=\log\frac{\det(g_{i\bar j}+\varphi_{i\bar j})}{\det g_{i\bar j}}
	\end{align}
we get 
	\begin{align}
		F_k=\sum\frac{\varphi_{i\bar i k}}{1+\varphi_{i\bar i}}
	\end{align}\\
So we have 
\begin{align}
	& (\partial_t-\Delta_{\varphi})(|\nabla\varphi|^2) \le  \varphi_k P_{\bar k}+\varphi_{\bar k}P_k + \frac{C_2}{1+\varphi_{i\bar i}}|\nabla\varphi|^2 -\frac{|\varphi_{ik}|^2+|\varphi_{i\bar k}|^2}{1+\varphi_{i\bar i}} \notag\\
	&\quad\quad\quad\quad\quad\quad \le  \varphi_k P_{\bar k}+\varphi_{\bar k}P_k + \frac{C_2}{1+\varphi_{i\bar i}}|\nabla\varphi|^2 -\frac{|\varphi_{ik}|^2}{1+\varphi_{i\bar i}} - (n+\Delta\varphi)+2n
\end{align}\\
Where in the last line above we use the estimate
\begin{align}
	\sum_{i,k} \frac{|\varphi_{i\bar k}|^2}{1+\varphi_{i\bar i}} \ge  \sum_i \frac{\varphi_{i\bar i}^2}{1+\varphi_{i\bar i}} = n+\Delta\varphi - \sum_i \frac{2\varphi_{i\bar i}+1}{1+\varphi_{i\bar i}} \ge  n+\Delta\varphi-2n
\end{align} \\
Next we estimate the last term 
\begin{align}
	&\nabla_{\varphi}(e^{H(F,\varphi)})\cdot \nabla_{\varphi}(|\nabla\varphi|^2)=e^{H(F,\varphi)}\nabla_{\varphi}H\cdot \nabla_{\varphi}(|\nabla\varphi|^2) \notag\\
	&\quad =e^{H(F,\varphi)}Re\big(\frac{H_i\varphi_{j}\varphi_{,\bar{i}\bar{j}}+H_i\varphi_{\bar{j}}\varphi_{j\bar{i}}}{1+\varphi_{i\bar{i}}}\big)\\
	&\quad =e^{H(F,\varphi)}Re\big(\frac{H_i\varphi_j\varphi_{,\bar{i}\bar{j}}}{1+\varphi_{i\bar{i}}}\big)+e^{H(F,\varphi)}Re\big(\frac{H_i\varphi_{\bar{i}}\varphi_{i\bar{i}}}{1+\varphi_{i\bar{i}}}\big)
\end{align}\\
We notice that one of the item above can be eliminated by the following completion of square 
\begin{align}
	|\nabla_{\varphi}H|^2|\nabla\varphi|^2-2Re\big(\frac{H_i\varphi_j\varphi_{,\bar{i}\bar{j}}}{1+\varphi_{i\bar{i}}}\big)+\frac{|\varphi_{,ii}|^2}{1+\varphi_{i\bar{i}}}\ge  0
\end{align}\\
and the other term is estimated as 
\begin{align}
	& 2Re\big(\frac{H_i\varphi_{\bar{i}}\varphi_{i\bar{i}}}{1+\varphi_{i\bar{i}}}\big) = (1-\frac{1}{1+\varphi_{i\bar i}})(H_i\varphi_{\bar i}+H_{\bar i}\varphi_i) \notag\\
	&\quad = (-\lambda\varphi+\varphi^2)_i\varphi_{\bar i}+(-\lambda\varphi+\varphi^2)_{\bar i}\varphi_i + 2\delta_1F(F_i\varphi_{\bar i}+F_{\bar i}\varphi_i)-\frac{H_i\varphi_{\bar i}+H_{\bar i}\varphi_i}{1+\varphi_{i\bar i}} \notag\\
	&\quad \le  (\lambda+1)C_3|\nabla\varphi|^2 + 2\delta_1F \bigg( \eps \frac{|F_i|^2}{1+\varphi_{i\bar i}}\cdot |\varphi_i|^2 + \frac{1}{\eps}(1+\varphi_{i\bar i})\bigg)+|\nabla_\varphi H|^2+|\nabla_\varphi\varphi|^2 \notag\\
	&\quad \le  (\lambda+1)C_3|\nabla\varphi|^2 + 2\delta_1F\bigg(\eps |\nabla_\varphi F|^2\cdot |\nabla\varphi|^2+\frac{1}{\eps}(n+\Delta\varphi )\bigg)+|\nabla_\varphi H|^2+|\nabla_\varphi\varphi|^2
\end{align}\\
Now we choose $\eps$ sufficiently small such that $\eps ||F||_0<\frac{1}{2}$ and then choose $\delta_1$ small so that $\delta_1||F||_0<\frac{\eps}{4}$. Then we have 
\begin{align}
	& 2Re\big(\frac{H_i\varphi_{\bar{i}}\varphi_{i\bar{i}}}{1+\varphi_{i\bar{i}}}\big) \le  (\lambda+1)C_3|\nabla\varphi|^2 + \delta_1 |\nabla_\varphi F|^2\cdot |\nabla\varphi|^2+\frac{1}{2}(n+\Delta\varphi ) +|\nabla_\varphi H|^2+|\nabla_\varphi\varphi|^2
\end{align}\\
Combining all the calculations above, we get 
\begin{align}
	&(\partial_t-\Delta_{\varphi})(e^{H(F,\varphi)}(|\nabla\varphi|^2+K))\cdot e^{-H(F,\varphi)} \le  -K |\nabla_\varphi H|^2+ 2\nabla P\cdot\nabla\varphi + C_2 tr_{\omega_{\varphi}}\omega_0 |\nabla\varphi|^2 +2n \notag \\
	&\qquad +(|\nabla\varphi|^2+K)\bigg( (\lambda+1)C_1 -(\frac{\lambda}{2} - C_1-\delta_1C_1)tr_{\omega_{\varphi}}\omega_0 - \frac{3}{2} |\nabla_\varphi\varphi|^2-2\delta_1|\nabla\varphi F|^2 - c|\nabla\varphi|^{\frac{2}{n}}\bigg) \notag\\
	&\qquad\qquad +(\lambda+1)C_3|\nabla\varphi|^2 + \delta_1 |\nabla_\varphi F|^2\cdot |\nabla\varphi|^2-\frac{1}{2}(n+\Delta\varphi ) +|\nabla_\varphi H|^2+|\nabla_\varphi\varphi|^2
\end{align} \\
For $K>2$, and $\lambda$ large enough so that
\begin{align}
	\frac{\lambda}{4}>(1+\delta_1)C_1+C_2+1
\end{align}
then we have 
\begin{align}
	&(\partial_t-\Delta_{\varphi})(e^{H(F,\varphi)}(|\nabla\varphi|^2+K))\cdot e^{-H(F,\varphi)}  \le  2\nabla P\cdot\nabla\varphi -\frac{1}{2}(n+\Delta\varphi)       \notag \\
	&\quad\quad\quad\quad\quad\quad +(|\nabla\varphi|^2+K)\bigg( C_4 -\frac{\lambda}{4}tr_{\omega_{\varphi}}\omega_0 -|\nabla_\varphi\varphi|^2 -\delta_1|\nabla\varphi F|^2 - c|\nabla\varphi|^{\frac{2}{n}}\bigg) 
\end{align}\\
\end{proof}

\noindent
Next we are going to get the second derivative bound. \\

\begin{lem}[\textbf{Estimate of $\Delta\varphi$}]\label{Lem313}
		Let $(\varphi,F,P)$ be a solution of Pseudo Calabi Flow on $[0,T)\times M$ satisfying the conditions in Theorem \ref{MT31}, then for any $p>1$,  there exists a constant $C$ depending only on $||F||_0, p, T$ and the background metric,  such that 
		\begin{align}
		\sup_{t \in (0,T)} t^{1 + np} \int_M (n + \Delta\varphi)^p \omega_{0}^n \le  C (p,||F||_0,T)
		\end{align} \\
\end{lem}

\begin{lem}[Differential inequality for Lemma \ref{Lem313}]\label{Lem314}
	We put $H(F,\varphi) = -\lambda \varphi + \delta_0\varphi^2 + \delta_1F^2$, then for sufficiently large $\lambda$ depending on background metric and $C^0$ bound of $F$, we have the following point-wise estimate 
	\begin{align}
	&(\partial_t - \Delta_{\varphi})(e^H(n+\Delta\varphi))e^{-H} \le    \Delta P + C \notag\\
	&\quad\quad\qquad + (n+\Delta\varphi)\left(C - \frac{\lambda}{2}tr_{\omega_{\varphi}}\omega_0 - 2\delta_0|\nabla_{\varphi}\varphi|^2 - 2\delta_1|\nabla_{\varphi}F|^2 \right)
\end{align}
where $C$ depends on background metric and $C^0$ bound of $F$. \\
\end{lem}

\begin{proof}[Proof of Lemma \ref{Lem313}]
Now we set \( u = e^H(n + \Delta\varphi) \) bring this into (\ref{K3}), then we will have:
\begin{align}\label{K4}
	 & \partial_t\big( \int_M u^p\omega_\varphi^n\big)\le  - \int_M \frac{p(p-1)}{2} u^{p-2} |\nabla_\varphi u|^2_\varphi\omega_\varphi^n +\int_M u^p(|Ric|tr_\varphi\omega_0+\underline R)\omega_\varphi^n \notag\\
	 &\quad\quad\qquad\qquad\qquad +\int_M pu^{p-1}\big[(\partial_t-\Delta_\varphi)u+  \frac{1}{2(p-1)}|\nabla_\varphi F|^2 u  \big]\omega_\varphi^n \notag\\
	&\quad\qquad \le  \int_M pu^p \left(C_1 - (\frac{\lambda}{2}-|Ric|)tr_{\omega_{\varphi}}\omega_0 - 2\delta_0|\nabla_{\varphi}\varphi|^2 - (2\delta_1-\frac{1}{2(p-1)}) |\nabla_{\varphi}F|^2 \right) \notag\\
	&\quad\qquad\qquad - \int_M \frac{p(p-1)}{2} u^{p-2} |\nabla_\varphi u|^2_\varphi\omega_\varphi^n +C_2 \int_M pu^{p-1}e^H\omega_\varphi^n+ \int_M pu^{p-1}e^H\Delta P\, \omega_\varphi^n
\end{align}\\
The term $\Delta P$ can be handled using integral by part. 
\begin{align}\label{0.2N}
	& \int_M p u^{p-1} e^H \Delta P \omega_{\varphi}^n = \int_M p u^{p-1} e^H \Delta P e^F \omega_0^n \notag\\
	& = - \int_M p(p-1) u^{p-2} e^H \nabla u \cdot \nabla P \omega_{\varphi}^n - \int_M p u^{p-1} e^H \nabla H \cdot \nabla P \omega_{\varphi}^n  - \int_M p u^{p-1} e^H \nabla P \cdot \nabla F \omega_{\varphi}^n.
\end{align}\\
In the above, we estimate:
\begin{align}
	& - \int_M p(p-1) u^{p-2} e^H \nabla u \cdot \nabla P \omega_{\varphi}^n \le  \int_M p(p-1) u^{p-2} e^H |\nabla u||\nabla P| \omega_{\varphi}^n \notag\\
	&\qquad \le  \int_M p(p-1) u^{p-1} |\nabla_{\varphi} u||\nabla_{\varphi} P| \omega_{\varphi}^n \notag\\
	&\qquad \le  \int_M \frac{p(p-1)}{4} u^{p-2} |\nabla_{\varphi} u|^2 \omega_{\varphi}^n  + \int_M 4p(p-1) u^p |\nabla_{\varphi} P|^2 \omega_{\varphi}^n.
\end{align}\\
Next:
\begin{align}\label{0.4N}
	& -\int_M p u^{p-1} e^H \nabla H \cdot \nabla P \omega_{\varphi}^n - \int_M p u^{p-1} e^H \nabla P \cdot \nabla F \omega_{\varphi}^n \notag\\
	& \qquad\qquad \le  \int_M p u^{p-1} e^H C_3( |\nabla \varphi| + |\nabla F|) |\nabla P| \omega_{\varphi}^n \notag\\
	& \qquad\qquad   \le  \int_M p u^{p-1} e^H C_3( |\nabla_\varphi \varphi| + |\nabla_\varphi F|) |\nabla_\varphi P|(n+\Delta\varphi) \omega_{\varphi}^n \notag\\
	&\qquad\qquad \le  \int_M pu^p\bigg(\eps ( |\nabla_\varphi \varphi|^2 + |\nabla_\varphi F|^2)+ \frac{C_3^2}{\eps} |\nabla_\varphi P|^2 \bigg) \omega_\varphi^n
\end{align}\\
Therefore we have 
\begin{align}
	& \int_M p u^{p-1} e^H \Delta P \omega_{\varphi}^n \le  \int_M \frac{p(p-1)}{4} u^{p-2} |\nabla_{\varphi} u|^2 \omega_{\varphi}^n \notag\\
	&\qquad + \int_M pu^p\bigg(\eps ( |\nabla_\varphi \varphi|^2 + |\nabla_\varphi F|^2)+ \big(\frac{C_3^2}{\eps}+4(p-1)\big) |\nabla_\varphi P|^2 \bigg) \omega_\varphi^n
\end{align}\\
Bring this back to (\ref{K4}), we have 
\begin{align}
	&\partial_t\big( \int_M u^p\omega_\varphi^n\big)\le  \int_M pu^p \bigg(C_1 - (\frac{\lambda}{2}-|Ric|)tr_{\omega_{\varphi}}\omega_0 - (2\delta_0-\eps) |\nabla_{\varphi}\varphi|^2 \notag \\
	&\qquad\qquad\qquad\qquad\qquad\qquad\qquad\qquad\qquad - (2\delta_1-\frac{1}{2(p-1)}-\eps ) |\nabla_{\varphi}F|^2 \bigg)\,\omega_\varphi^n \notag\\
	&\qquad\qquad - \int_M \frac{p(p-1)}{4} u^{p-2} |\nabla_\varphi u|^2_\varphi\omega_\varphi^n +C_2 \int_M pu^{p-1}e^H \omega_\varphi^n \notag\\
	&\qquad\qquad\qquad\qquad\qquad + \bigg(\frac{C_3^2}{\eps}+4(p-1)\bigg) \int_M p u^p |\nabla_{\varphi} P|^2 \omega_{\varphi}^n 
\end{align}\\
Let $\delta_0=\delta_1=1$ and choose $\eps$ small and $p$ very large such that 
\begin{align}
	2\delta_0-\eps>0,\qquad 2\delta_1-\frac{1}{2(p-1)}-\eps >0
\end{align}
We can get rid of the terms $|\nabla_\varphi\varphi|^2$ and $|\nabla_\varphi F|^2$. Also notice that 
\begin{align}
	\int_M pu^{p-1}e^H \omega_\varphi^n \le  \int_M pu^{p} \omega_\varphi^n +C_4
\end{align}\\
We then get 
\begin{align}\label{K41}
	&\partial_t\big( \int_M u^p\omega_\varphi^n\big)\le  \int_M pu^p \bigg(C_5 - (\frac{\lambda}{2}-|Ric|)tr_{\omega_{\varphi}}\omega_0 \bigg)\,\omega_\varphi^n +C_4 \notag\\
	&\qquad\qquad  - \int_M \frac{p(p-1)}{4} u^{p-2} |\nabla_\varphi u|^2_\varphi\omega_\varphi^n + C_6\int_M p u^p |\nabla_{\varphi} P|^2 \omega_{\varphi}^n 
\end{align}\\
For the last term, we need to use the condition that $P$ is $\delta_*$-close to a continuous function with modulus $\omega(r)$.  Let $\tilde{P}$ be the $C^2$ function given by Corollary \ref{Cor30}.
\begin{align}\label{K5}
	\int_M p u^p |\nabla_{\varphi} P|^2 \omega_{\varphi}^n = \int_M p u^p \nabla_{\varphi} P \cdot \nabla_{\varphi} \tilde{P} \omega_{\varphi}^n + \int_M p u^p \nabla_{\varphi} P \cdot \nabla_{\varphi}(P - \tilde{P}) \omega_{\varphi}^n.
\end{align}\\
Notice 
\begin{align}
	|\nabla \tilde{P}| \le  \frac{C_{7}}{r}\  \Longrightarrow\  |\nabla_{\varphi} \tilde{P}| \le  \frac{C_{7}}{r} (tr_{\varphi} \omega_0)^{1/2} 
\end{align}
Therefore:
\begin{align}\label{K6}
	&\int_M p u^p \nabla_{\varphi} P \cdot \nabla_{\varphi} \tilde{P} \omega_{\varphi}^n \le  \int_M p \frac{C_{7}}{r} u^p |\nabla_{\varphi} P| (tr_{\omega_{\varphi} }\omega_0)^{1/2} \omega_{\varphi}^n \notag\\
	&\le  \int_M \frac{p}{2} u^p |\nabla_{\varphi} P|^2 \omega_{\varphi}^n + \int_M \frac{p}{2} \frac{C_{7}^2}{r^2} u^p tr_{\omega_{\varphi}} \omega_0 \omega_{\varphi}^n 
\end{align}\\
Hence
\begin{align}
	\int_M p u^p |\nabla_{\varphi} P|^2 \omega_{\varphi}^n \le  C(p,r) \int_M  u^p tr_{\omega_{\varphi} }\omega_0 \omega_{\varphi}^n + 2\int_M p u^p \nabla_{\varphi} P \cdot \nabla_{\varphi}(P - \tilde{P}) \omega_{\varphi}^n 
\end{align}
And
\begin{align}\label{K7}
	&\int_M p u^p \nabla_{\varphi} P \cdot \nabla_{\varphi}(P - \tilde{P}) \omega_{\varphi}^n \notag\\
	&\qquad =  \int_M\nabla_\varphi P\cdot \left(\nabla_\varphi\big( pu^p(P-\tilde P)\big)- (P-\tilde P)p(p-1)u^{p-1} \nabla_{\varphi} u \right) \omega_{\varphi}^n \notag\\
	&\qquad =-\int_M pu^p(P-\tilde P)\Delta_\varphi P \omega_{\varphi}^n -\int_M (P-\tilde P)p(p-1)u^{p-1} \nabla_\varphi P\cdot\nabla_\varphi u\, \omega_\varphi^n           \notag\\
	&\qquad \le  \int_M (\delta_* + \omega(r))pu^p C_{8}(tr_{\omega_{\varphi}}\omega_0+1)\,\omega_\varphi^n +  \int_M  \frac{p(p-1)}{16C_6} u^{p-2} |\nabla_{\varphi} u|^2 \omega_{\varphi}^n \notag\\
	&\qquad\qquad \quad  + \int_M 4C_6 p(p-1) u^p (\delta_* + \omega(r))^2 |\nabla_{\varphi} P|^2 \omega_{\varphi}^n.
\end{align}\\
Putting these together, we will have
\begin{align}
	& \int_M p u^p |\nabla_{\varphi} P|^2 \omega_{\varphi}^n \le  C(p,r) \int_M  u^p tr_{\omega_{\varphi}} \omega_0 \omega_{\varphi}^n +\int_M  \frac{p(p-1)}{8C_6} u^{p-2} |\nabla_{\varphi} u|^2 \omega_{\varphi}^n  \notag\\
	& +  \int_M 2(\delta_* + \omega(r))pu^p C_{8}(tr_{\omega_{\varphi}}\omega_0+1)\,\omega_\varphi^n  + \int_M 8C_6 p(p-1) u^p (\delta_* + \omega(r))^2 |\nabla_{\varphi} P|^2 \omega_{\varphi}^n
\end{align}\\
Then we choose $r,\ \delta_*$ small enough such that 
\begin{align}
	8C_6 p(p-1) u^p (\delta_* + \omega(r))^2 <\frac{1}{2}
\end{align}\\
We will have 
\begin{align}
	& C_6 \int_M p u^p |\nabla_{\varphi} P|^2 \omega_{\varphi}^n \le  2C_6 C(p,r) \int_M  u^p tr_{\omega_{\varphi} }\omega_0 \omega_{\varphi}^n +\int_M  \frac{p(p-1)}{4} u^{p-2} |\nabla_{\varphi} u|^2 \omega_{\varphi}^n  \notag\\
	&\qquad\qquad\qquad\qquad + \int_M 4C_6 C_{8}(\delta_* + \omega(r))pu^p (tr_{\omega_{\varphi}}\omega_0+1)\,\omega_\varphi^n   \notag\\
	&\qquad\qquad \le   \int_M C_{9}pu^p (tr_{\omega_{\varphi}}\omega_0+1)\,\omega_\varphi^n   +\int_M  \frac{p(p-1)}{4} u^{p-2} |\nabla_{\varphi} u|^2 \omega_{\varphi}^n 
\end{align}\\
Combining those together, the expression (\ref{K41}) turn to be
\begin{align}
	&\partial_t\big( \int_M u^p\omega_\varphi^n\big)\le  \int_M pu^p \bigg(C_{10} - (\frac{\lambda}{2}-|Ric|-C_{11})tr_{\omega_{\varphi}}\omega_0 \bigg)\,\omega_\varphi^n +C_{12} 
\end{align}\\
Notice that
\begin{align}
	tr_{\varphi} \omega_0 \ge  e^{-\frac{F}{n-1}} (n + \Delta\varphi)^{\frac{1}{n-1}} \ge  c (n+\Delta\varphi)^{\frac{1}{n}}-C_{13}
\end{align}\\
and choosing $\lambda$ very large so that
\begin{align}
	\frac{\lambda}{2}-|Ric(\omega_0)|-C_{11}> \max e^{H/n}
\end{align}
we get:
\begin{align}
	\partial_t\left( \int_M u^p \omega_{\varphi}^n \right) \le  -c \int_M u^{p + \frac{1}{n}} \omega_{\varphi}^n + C_{14}\int_M p u^p \omega_{\varphi}^n + C_{12} 
\end{align}
for every $p$ large enough. Using previous calculations, let \( \lambda_p > 0 \), we then compute:
\begin{align}
	& \partial_t \left( t^{\lambda_p} \int_M u^p \, \omega_{\varphi}^n \right) \le  \lambda_p t^{\lambda_p - 1} \int_M u^p \, \omega_{\varphi}^n - c t^{\lambda_p} \int_M u^{p + \frac{1}{n}} \, \omega_{\varphi}^n + C_{14}t^{\lambda_p} \int_M p u^p \omega_{\varphi}^n + C_{12}t^{\lambda_p}  \notag \\
	&\qquad\qquad \le  \lambda_p t^{\lambda_p - 1} \int_M u^p \, \omega_{\varphi}^n - c \left( t^{\lambda_p \frac{np}{1 + np}} \int_M u^p \, \omega_{\varphi}^n \right)^{1 + \frac{1}{np}} + C_{14}t^{\lambda_p} \int_M p u^p \omega_{\varphi}^n + C_{12}t^{\lambda_p} 
\end{align}\\
Choose \( \lambda_p = 1 + np \) so that \( \lambda_p - 1 = \lambda_p \frac{np}{1 + np} \), and we now have:
\begin{align}
	&\partial_t \left( t^{\lambda_p} \int_M u^p \, \omega_{\varphi}^n \right) \le  C_{15} \left( t^{\lambda_p-1} \int_M u^p \, \omega_{\varphi}^n \right) - c \left( t^{\lambda_p-1} \int_M u^p \, \omega_{\varphi}^n \right)^{1 + \frac{1}{np}} + C_{12} t^{\lambda_p}
\end{align}\\
Which means, on finite time interval
\begin{align}
	\sup_{t \in (0,T)} t^{1 + np} \int_M (n + \Delta\varphi)^p \omega_{\varphi}^n \le  C_{16}
\end{align}\\
Hence 
\begin{align}
	\sup_{t \in (0,T)} t^{1 + np} \int_M (n + \Delta\varphi)^p \omega_{0}^n \le  C(p,||F||_0,T)
\end{align}\\
\end{proof}

\begin{proof}[Proof of Lemma \ref{Lem314}]
 We compute:
\begin{align}
	&(\partial_t - \Delta_{\varphi})(e^{H(F,\varphi)}(n+\Delta\varphi))e^{-H} = (\partial_t - \Delta_{\varphi})H(F,\varphi)(n+\Delta\varphi) \notag\\
	&\quad + (\partial_t - \Delta_{\varphi})(n+\Delta\varphi) - |\nabla_{\varphi}H|^2(n+\Delta\varphi) - 2\nabla_{\varphi}H \cdot \nabla_{\varphi}(n+\Delta\varphi).
\end{align}\\
We already know that:
\begin{align}
	(\partial_t-\Delta_{\varphi})H(F,\varphi) \le  (\lambda+\delta_0)C_1 -(\lambda-\delta_0 C_1-\delta_1C_1)tr_{\omega_{\varphi}}\omega_0 -2\delta_0|\nabla_\varphi\varphi|^2-2\delta_1|\nabla_\varphi F|^2 
\end{align}\\
Following Yau's calculations \cite{Yau1978}, we get
\begin{align}\label{Yau}
	&\Delta_\varphi\Delta\varphi =\frac{1}{1+\varphi_{k\bar k}}(g^{i\bar j}\varphi_{i\bar j})_{k\bar k} = \frac{1}{1+\varphi_{k\bar k}} (R_{k\bar ki\bar i}\varphi_{i\bar i} +\varphi_{k\bar ki\bar i}) \notag \\	
	&\quad	=R_{k\bar ki\bar i}\frac{1+\varphi_{i\bar i}}{1+\varphi_{k\bar k}} + \frac{|\varphi_{k\bar l i}|^2}{(1 + \varphi_{k\bar k})(1 + \varphi_{l\bar l})} +\Delta F-\underline R \notag \\
	&\quad \ge  -C_2 tr_{\omega_{\varphi}}\omega_0\,(n+\Delta\varphi) + \frac{|\varphi_{k\bar l i}|^2}{(1 + \varphi_{k\bar k})(1 + \varphi_{l\bar l})} +\Delta F-\underline R
\end{align}\\
We also observe the following 
\begin{align}
	\frac{|(\Delta\varphi)_i|^2}{1 + \varphi_{i\bar{i}}} = \frac{\left|\sum_k \varphi_{k\bar{k}i}\right|^2}{1 + \varphi_{i\bar{i}}} \le  \frac{|\varphi_{k\bar{k}i}|^2 (n + \Delta\varphi)}{(1 + \varphi_{i\bar{i}})(1 + \varphi_{k\bar{k}})} \le  \frac{|\varphi_{k\bar{l}i}|^2 (n + \Delta\varphi)}{(1 + \varphi_{k\bar{k}})(1 + \varphi_{l\bar{l}})}.
\end{align}\\
So we have 
\begin{align}
	\Delta_\varphi\Delta\varphi \ge  -C_2 tr_{\omega_{\varphi}}\omega_0\,(n+\Delta\varphi) + \frac{|\nabla_{\varphi}(n+\Delta\varphi)|^2}{n+\Delta\varphi} +\Delta F-\underline R
\end{align}\\
Then:
\begin{align}
	(\partial_t - \Delta_{\varphi})(n+\Delta\varphi) \le  C_2(tr_{\omega_{\varphi}} \omega_0)(n+\Delta\varphi) - \frac{|\nabla_{\varphi}(n+\Delta\varphi)|^2}{n+\Delta\varphi} + \Delta P + \underline R
\end{align}\\
We note the completion of square:
\begin{align}
	|\nabla_{\varphi}H|^2(n+\Delta\varphi) + 2\nabla_{\varphi}H \cdot \nabla_{\varphi}(n+\Delta\varphi) + \frac{|\nabla_{\varphi}(n+\Delta\varphi)|^2}{n+\Delta\varphi} \ge  0.
\end{align}\\
Therefore, we get:
\begin{align}
	&(\partial_t - \Delta_{\varphi})(e^H(n+\Delta\varphi))e^{-H} \le  C_2(tr_{\omega_{\varphi}} \omega_0)(n+\Delta\varphi)  + \Delta P + \underline R \notag\\
	& (n+\Delta\varphi)\left( (\lambda+\delta_0)C_1 -(\lambda-\delta_0 C_1-\delta_1C_1)tr_{\omega_{\varphi}}\omega_0 -2\delta_0|\nabla_\varphi\varphi|^2-2\delta_1|\nabla\varphi F|^2 \right) \notag\\
\end{align}\\
We choose $\lambda$ large enough so that 
\begin{align}
	\frac{\lambda}{2}> (\delta_0+\delta_1)C_1+C_2
\end{align}\\
then it will give us 
\begin{align}
	&(\partial_t - \Delta_{\varphi})(e^H(n+\Delta\varphi))e^{-H} \le    \Delta P + C_3 \notag\\
	&\quad\quad\qquad + (n+\Delta\varphi)\left(C_3 - \frac{\lambda}{2}tr_{\omega_{\varphi}}\omega_0 - 2\delta_0|\nabla_{\varphi}\varphi|^2 - 2\delta_1|\nabla_{\varphi}F|^2 \right) 
\end{align}\\
\end{proof}

\noindent
Then we estimate $|\nabla_\varphi P|^2$.  \\

\begin{lem}[\textbf{Estimate of $|\nabla_\varphi P|^2$}]\label{Lem315}
	Let $\varphi$ be a smooth $\omega_0$-psh function.  Let $P$ solve $\Delta_{\varphi}P=-tr_{\omega_{\varphi}}Ric(\omega_0)+\underline{R}$,  and $\int_MP\omega_{\varphi}^n=0$.  Let $p>1$,  $K\ge 2$ large enough,  there exists $\delta_*>0$ small enough,  depending on $p$,  background metric,  and $||\log\frac{\omega_{\varphi}^n}{\omega_0^n}||_0$ and the choice of $K$,  such that if $P$ is $\delta_*$-close to a continuous function with modulus $\omega(r)$,  one has 
	\begin{align}
		K\int_M|\nabla_{\varphi}P|^{2(p+1)}\omega_{\varphi}^n\le C(p)\big(\int_M(1+tr_{\omega_{\varphi}}\omega_0)^{p+1}\omega_{\varphi}^n+\int_M|\nabla_{\varphi}F|^{2(p+1)}\omega_{\varphi}^n\big)+C. 
	\end{align}\\
\end{lem}

\begin{proof}[Proof of Lemma \ref{Lem315}]
 Let $K>0$,  we compute:
\begin{align}
	& \Delta_{\varphi}\big(e^{K(P-\tilde P)^2}|\nabla_{\varphi}P|^2\big)e^{-K(P-\tilde P)^2} =\left( \Delta_{\varphi}(K(P-\tilde P)^2)+|\nabla_{\varphi}(K(P-\tilde P)^2)|^2 \right)(|\nabla_\varphi P|^2)  \notag \\
	&\qquad\qquad\qquad\qquad\qquad +\Delta_{\varphi}(|\nabla_{\varphi}P|^2) +2\nabla_{\varphi}(K(P-\tilde P)^2)\cdot \nabla_{\varphi}(|\nabla_{\varphi}P|^2).
\end{align}\\
In the above,  we have:
\begin{align}
\Delta_{\varphi}(|\nabla_{\varphi}P|^2)=g_{\varphi}^{i\bar{q}}g_{\varphi}^{p\bar{j}}P_{,iq}P_{,\bar{p}\bar{j}}+g_{\varphi}^{i\bar{q}}g_{\varphi}^{p\bar{j}}P_{i\bar{j}}P_{p\bar{q}}+2\nabla_{\varphi}P\cdot \nabla_{\varphi}\Delta_{\varphi}P+g_{\varphi}^{i\bar{q}}g_{\varphi}^{p\bar{j}}(Ric(\omega_{\varphi}))_{i\bar{j}}P_pP_{\bar{q}}.
\end{align}\\
Also
\begin{align}
2\nabla_{\varphi}(K(P-\tilde P)^2)\cdot \nabla_{\varphi}(|\nabla_{\varphi}P|^2)=2Re\big(g_{\varphi}^{i\bar{j}}g_{\varphi}^{p\bar{q}}((K(P-\tilde P)^2)_iP_{p\bar{j}}P_{\bar{q}}+(K(P-\tilde P)^2)_iP_pP_{,\bar{q}\bar{j}})\big).
\end{align}\\
We have the following completion of squares:
\begin{align}
|\nabla_{\varphi}(K(P-\tilde P)^2)|^2 |\nabla_\varphi P|^2 +2Re\big(g_{\varphi}^{i\bar{j}}g_{\varphi}^{p\bar{q}}(K(P-\tilde P)^2)_iP_pP_{,\bar{q}\bar{j}}\big)+g_{\varphi}^{i\bar{q}}g_{\varphi}^{p\bar{j}}P_{,ip}P_{,\bar{q}\bar{j}}\ge  0
\end{align}\\
And 
\begin{align}
	4|\nabla_{\varphi}(K(P-\tilde P)^2)|^2|\nabla_\varphi P|^2 + 2Re\big(g_{\varphi}^{i\bar{j}}g_{\varphi}^{p\bar{q}}((K(P-\tilde P)^2)_iP_{p\bar{j}}P_{\bar{q}}) +\frac{1}{4}g_{\varphi}^{i\bar{q}}g_{\varphi}^{p\bar{j}}P_{i\bar{j}}P_{p\bar{q}}\ge  0
\end{align}\\
So we obtain:
\begin{align}
	&\Delta_{\varphi}\big(e^{K(P-\tilde P)^2}|\nabla_{\varphi}P|^2\big)e^{-K(P-\tilde P)^2} \ge  (2K-16K^2(P-\tilde P)^2)|\nabla_\varphi(P-\tilde P)|^2 |\nabla_{\varphi}P|^2  \notag \\
	&\qquad\qquad\qquad\qquad +2K(P-\tilde P)\Delta_{\varphi}(P-\tilde P)|\nabla_{\varphi}P|^2 +\frac{3}{4}g_{\varphi}^{i\bar{q}}g_{\varphi}^{p\bar{j}}P_{i\bar{j}}P_{p\bar{q}} \notag\\
	&\qquad\qquad\qquad\qquad +2\nabla_{\varphi}P\cdot \nabla_{\varphi}\Delta_{\varphi}P+g_{\varphi}^{i\bar{q}}g_{\varphi}^{p\bar{j}}(Ric(\omega_{\varphi}))_{i\bar{j}}P_pP_{\bar{q}} 
\end{align}\\
In the above,  we can estimate: 
\begin{align}
	&(2K-16K^2(P-\tilde P)^2)|\nabla_\varphi(P-\tilde P)|^2 |\nabla_{\varphi}P|^2 \notag\\
	&\quad  \ge  \big(2K-16K^2(\delta_*+\omega(r))^2 \big)\bigg( |\nabla_\varphi P|^4 -\frac{C}{r^2}|\nabla_\varphi P|^2 \bigg)
\end{align}\\
And
\begin{align}
	2K(P-\tilde P)\Delta_{\varphi}(P-\tilde P)|\nabla_{\varphi}P|^2\ge  -K(\delta_*+\omega(r)) \frac{C}{r^2} (tr_{\omega_{\varphi}}\omega_0+1)|\nabla_\varphi P|^2
\end{align}\\
Also
\begin{align}
	& g_{\varphi}^{i\bar{q}}g_{\varphi}^{p\bar{j}}(Ric(\omega_{\varphi}))_{i\bar{j}}P_pP_{\bar{q}} =g_{\varphi}^{i\bar{q}}g_{\varphi}^{p\bar{j}}(Ric(\omega_0))_{i\bar{j}}P_pP_{\bar{q}} - g_{\varphi}^{i\bar{q}}g_{\varphi}^{p\bar{j}}F_{i\bar{j}}P_pP_{\bar{q}} \notag \\
	&\qquad \ge  -C tr_{\omega_{\varphi}}\omega_0|\nabla_{\varphi}P|^2 - \Delta_{\varphi}F|\nabla_{\varphi}P|^2 - \frac{dd^cF\wedge d^c P\wedge d P\wedge \omega_{\varphi}^{n-2}}{\omega_{\varphi}^n}.
\end{align}\\
So we get:
\begin{align}
	&\Delta_{\varphi}\big(e^{K(P-\tilde P)^2}|\nabla_{\varphi}P|^2\big)e^{-K(P-\tilde P)^2}  \notag\\
	&\quad  \ge  \big(2K-16K^2(\delta_*+\omega(r))^2 \big)\bigg( |\nabla_\varphi P|^4 -\frac{C}{r^2}|\nabla_\varphi P|^2 \bigg)   \notag \\
	&\qquad +\frac{3}{4}g_{\varphi}^{i\bar{q}}g_{\varphi}^{p\bar{j}}P_{i\bar{j}}P_{p\bar{q}} -C tr_{\omega_{\varphi}}\omega_0|\nabla_{\varphi}P|^2 -K(\delta_*+\omega(r)) \frac{C}{r^2} (tr_{\omega_{\varphi}}\omega_0+1)|\nabla_\varphi P|^2 \notag\\
	&\qquad +2\nabla_{\varphi}P\cdot \nabla_{\varphi}\Delta_{\varphi}P - \Delta_{\varphi}F|\nabla_{\varphi}P|^2 - \frac{dd^cF\wedge d^cP\wedge dP\wedge \omega_{\varphi}^{n-2}}{\omega_{\varphi}^n}.
\end{align}\\
Denote $u=e^{K(P-\tilde{P})^2}|\nabla_{\varphi}P|^2$,  then the above inequality is equivalent to:
\begin{align}\label{4.133New}
	&\Delta_{\varphi}u  \ge  (2K-16K^2(\delta_*+\omega(r))^2)\bigg( |\nabla_\varphi P|^2 - \frac{C}{r^2}\bigg) u    \notag \\
	&\quad +\frac{3}{4}e^{K(P-\tilde P)^2} g_{\varphi}^{i\bar{q}}g_{\varphi}^{p\bar{j}}P_{i\bar{j}}P_{p\bar{q}} -C tr_{\omega_{\varphi}}\omega_0 u -K(\delta_*+\omega(r)) \frac{C}{r^2} (tr_{\omega_{\varphi}}\omega_0+1) u \notag\\
	&\quad+2e^{K(P-\tilde P)^2}\nabla_{\varphi}P\cdot\nabla_{\varphi}\Delta_{\varphi}P - \Delta_{\varphi}Fu - e^{K(P-\tilde P)^2}\frac{dd^cF\wedge d^cP\wedge dP\wedge \omega_{\varphi}^{n-2}}{\omega_{\varphi}^n}.
\end{align}\\

\noindent
Let $p>1$,  we have:
\begin{align}
	\Delta_\varphi u^p = p u^{p-1}\Delta_\varphi u+p(p-1)u^{p-2}|\nabla_\varphi u|^2.
\end{align}

Integrate with respect to $\omega_\varphi^n$, we have
\begin{equation}\label{4.135New}
\begin{split}
&\int_Mp(p-1)u^{p-2}|\nabla_{\varphi}u|^2\omega_{\varphi}^n=-\int_Mpu^{p-1}\Delta_{\varphi}u\omega_{\varphi}^n\\
&\le (2K-16K^2(\delta_*+\omega(r))^2)\bigg(-\int_Mpu^p|\nabla_{\varphi}P|^2\omega_{\varphi}^n+\frac{C}{r^2}\int_Mpu^p\omega_{\varphi}^n\bigg)\\
&-\int_M\frac{3p}{4}u^{p-1}e^{K(P-\tilde{P})^2}|\sqrt{-1}\partial\bar{\partial}P|^2_{\varphi}\omega_{\varphi}^n+(K(\delta_*+\omega(r))+1)\frac{C}{r^2}\int_Mpu^p(tr_{\omega_{\varphi}}\omega_0+1)\omega_{\varphi}^n\\
&-\int_M2pu^{p-1}e^{K(P-\tilde{P})^2}\nabla_{\varphi}P\cdot \nabla_{\varphi}\Delta_{\varphi}P\omega_{\varphi}^n+\int_Mpu^p\Delta_{\varphi}F\omega_{\varphi}^n\\
&+\int_Mpu^{p-1}e^{K(P-\tilde{P})^2}dd^cF\wedge d^cP\wedge dP\wedge \omega_{\varphi}^{n-2}.
\end{split}
\end{equation}
We are going to integrate by parts in the last three terms of the above.
\begin{align}\label{4.135N}
	&-\int_M2pu^{p-1}e^{K(P-\tilde P)^2}\nabla_{\varphi}P\cdot \nabla_{\varphi}\Delta_{\varphi}P\omega_{\varphi}^n =\int_M2p(p-1)u^{p-2}e^{K(P-\tilde P)^2}\nabla_{\varphi}u\cdot \nabla_{\varphi}P\Delta_{\varphi}P\omega_{\varphi}^n \notag \\
	&\quad\qquad\qquad +\int_M2pu^{p-1}e^{K(P-\tilde P)^2}2K (P-\tilde P)\left( |\nabla_{\varphi}P|^2-\nabla_\varphi\tilde P\cdot\nabla_\varphi P \right) \Delta_{\varphi}P\omega_{\varphi}^n \notag\\
	&\qquad\qquad\qquad\qquad + \int_M2pu^{p-1}e^{K(P-\tilde P)^2}(\Delta_{\varphi}P)^2\omega_{\varphi}^n.
\end{align}\\
Also
\begin{equation}\label{4.136}
\int_Mpu^p\Delta_{\varphi}F\omega_{\varphi}^n=-\int_Mp^2u^{p-1}\nabla_{\varphi}u\cdot \nabla_{\varphi}F\omega_{\varphi}^n.
\end{equation}

Next:
\begin{equation}\label{4.137}
\begin{split}
&\int_Mpu^{p-1}e^{K(P-\tilde{P})^2}dd^cF\wedge d^cP\wedge dP\wedge \omega_{\varphi}^{n-2}\\
&=-\int_Mp(p-1)u^{p-2}e^{K(P-\tilde{P})^2}du\wedge d^cF\wedge d^cP\wedge dP\wedge \omega_{\varphi}^{n-2}\\
&-\int_Mpu^{p-1}e^{K(P-\tilde{P})^2}2K(P-\tilde{P})d(P-\tilde{P})\wedge d^cF\wedge d^cP\wedge dP\wedge \omega_{\varphi}^{n-2}\\
&-\int_Mpu^{p-1}e^{K(P-\tilde{P})^2}d^cF\wedge dd^cP\wedge dP\wedge \omega_{\varphi}^{n-2}.
\end{split}
\end{equation}

Now we wish to estimate the right hand sides of (\ref{4.135N})-(\ref{4.137}).
We start with the right hand side of (\ref{4.135N}):
\begin{align}
	&\int_M2p(p-1)u^{p-2}e^{K(P-\tilde P)^2}\nabla_{\varphi}u\cdot \nabla_{\varphi}P\Delta_{\varphi}P\omega_{\varphi}^n \notag \\
	&\le  \frac{p(p-1)}{16}\int_Mu^{p-2}|\nabla_{\varphi}u|^2\omega_{\varphi}^n + C \int_M p(p-1)u^{p-1}e^{K(P-\tilde P)^2}(1+tr_{\omega_{\varphi}}\omega_0)^2\omega_{\varphi}^n.
\end{align}\\
Next
\begin{align}
	\int_M2pu^{p-1}e^{K(P-\tilde P)^2}(\Delta_{\varphi}P)^2\omega_{\varphi}^n\le  C \int_M pu^{p-1}e^{K(P-\tilde P)^2}(1+tr_{\omega_{\varphi}}\omega_0)^2\omega_{\varphi}^n.
\end{align}
Next,  keeping in mind that $u=e^{K(P-\tilde{P})^2}|\nabla_{\varphi}P|^2$ and that $|P-\tilde{P}|\le \delta_*+\omega(r)$.
\begin{align}
	&\int_M pu^{p-1}e^{K(P-\tilde P)^2} (P-\tilde P)\left( |\nabla_{\varphi}P|^2-\nabla_\varphi\tilde P\cdot\nabla_\varphi P \right) \Delta_{\varphi}P\omega_{\varphi}^n \notag\\
	& \le  \int_M pu^p |P-\tilde P||\Delta_\varphi P|\omega_\varphi^n + \int_M pu^{p-1}e^{K(P-\tilde P)^2} |P-\tilde P|\left( |\nabla_{\varphi}P|^2 +|\nabla_\varphi\tilde P|^2 \right) |\Delta_{\varphi}P| \omega_{\varphi}^n \notag\\
	&  \le  C(\delta_*+\omega(r)) \int_M pu^p  (1+tr_{\omega_{\varphi}}\omega_0)\omega_\varphi^n + \big(\delta_*+\omega(r)\big)\frac{C}{r^2} \int_M pu^{p-1} e^{K(P-\tilde P)^2}(1+tr_{\omega_{\varphi}}\omega_0)\omega_\varphi^n 
\end{align}\\

In the above,  we noted that $|\nabla_{\varphi}\tilde{P}|^2\le tr_{\omega_{\varphi}}\omega_0|\nabla \tilde{P}|^2\le tr_{\omega_{\varphi}}\omega_0\cdot \frac{C}{r^2}$.

Summing up we have 
\begin{align}\label{4.142}
	& -\int_M2pu^{p-1}e^{K(P-\tilde P)^2}\nabla_{\varphi}P\cdot \nabla_{\varphi}\Delta_{\varphi}P\omega_{\varphi}^n \le  \frac{p(p-1)}{16}\int_Mu^{p-2}|\nabla_{\varphi}u|^2\omega_{\varphi}^n \notag\\
	&\quad\quad + C \int_M p^2u^{p-1}e^{K(P-\tilde P)^2}(1+tr_{\omega_{\varphi}}\omega_0)^2\omega_{\varphi}^n +CK(\delta_*+\omega(r)) \int_M pu^p  (1+tr_{\omega_{\varphi}}\omega_0)\omega_\varphi^n \notag\\
	&\quad\quad\quad\quad\quad\quad\quad + K(\delta_*+\omega(r))\frac{C}{r^2} \int_M pu^{p-1} e^{K(P-\tilde P)^2}(1+tr_{\omega_{\varphi}}\omega_0)\omega_\varphi^n 
\end{align}

Now we look at (\ref{4.136}) and we can estimate:
\begin{equation}\label{4.143}
-\int_Mp^2u^{p-1}\nabla_{\varphi}u\cdot \nabla_{\varphi}F\omega_{\varphi}^n\le \int_M\frac{p(p-1)}{16}u^{p-2}|\nabla_{\varphi}P|^2+\int_M\frac{4p^3}{p-1}u^p|\nabla_{\varphi}F|^2\omega_{\varphi}^n.
\end{equation}
Next we are going to estimate the right hand side of (\ref{4.137}).
We have:
\begin{equation*}
\begin{split}
&-\int_Mp(p-1)u^{p-2}e^{K(P-\tilde{P})^2}du\wedge d^cF\wedge d^cP\wedge dP\wedge \omega_{\varphi}^{n-2}\\
&\le \int_Mp(p-1)u^{p-2}e^{K(P-\tilde{P})^2}|\nabla_{\varphi}u||\nabla_{\varphi}F||\nabla_{\varphi}P|^2\omega_{\varphi}^n=\int_Mp(p-1)u^{p-1}|\nabla_{\varphi}u||\nabla_{\varphi}F|\omega_{\varphi}^n\\
&\le \int_M\frac{p(p-1)u^{p-2}}{16}|\nabla_{\varphi}u|^2\omega_{\varphi}^n+\int_M4p(p-1)u^p|\nabla_{\varphi}F|^2\omega_{\varphi}^n.
\end{split}
\end{equation*}
Then
\begin{equation*}
\begin{split}
&-\int_Mpu^{p-1}e^{K(P-\tilde{P})^2}2K(P-\tilde{P})d(P-\tilde{P})\wedge d^cF\wedge d^cP\wedge dP\wedge \omega_{\varphi}^{n-2}\\
&\le \int_Mpu^{p-1}e^{K(P-\tilde{P})^2}2K(\delta_*+\omega(r))(|\nabla_{\varphi}P|+|\nabla_{\varphi}\tilde{P}|)|\nabla_{\varphi}F||\nabla_{\varphi}P|^2\omega_{\varphi}^n\\
&\le\int_Mpu^p\cdot 2K(\delta_*+\omega(r))(|\nabla_{\varphi}P|+\frac{C}{r}(tr_{\omega_{\varphi}}\omega_0)^{\frac{1}{2}})|\nabla_{\varphi}F|\omega_{\varphi}^n\\
&\le p(p-1)\int_Mu^p|\nabla_{\varphi}F|^2\omega_{\varphi}^n+\int_M\frac{2p^2}{p-1}K^2(\delta_*+\omega(r))^2u^p(|\nabla_{\varphi}P|^2+\frac{C'}{r^2}tr_{\omega_{\varphi}}\omega_0)\omega_{\varphi}^n.
\end{split}
\end{equation*}
For the last term in the right hand side of (\ref{4.137}):
\begin{equation*}
\begin{split}
-&\int_Mpu^{p-1}e^{K(P-\tilde{P})^2}d^cF\wedge dd^cP\wedge dP\wedge \omega_{\varphi}^{n-2}\le \int_M\frac{p}{8}u^{p-1}e^{K(P-\tilde{P})^2}|\sqrt{-1}\partial\bar{\partial}P|^2_{\varphi}\omega_{\varphi}^n\\
&+2p\int_Mu^{p-1}e^{K(P-\tilde{P})^2}|\nabla_{\varphi}F|^2|\nabla_{\varphi}P|^2\omega_{\varphi}^n.
\end{split}
\end{equation*}
Summing up,  we obtain that:
\begin{equation}\label{4.144}
\begin{split}
&\int_Mpu^{p-1}e^{K(P-\tilde{P})^2}dd^cF\wedge d^cP\wedge dP\wedge \omega_{\varphi}^{n-2}\\
&\le \int_M\frac{p(p-1)u^{p-2}}{16}|\nabla_{\varphi}u|^2\omega_{\varphi}^n+\int_M\frac{pu^{p-1}}{8}e^{K(P-\tilde{P})^2}|\sqrt{-1}\partial\bar{\partial}P|^2_{\varphi}\omega_{\varphi}^n\\
&+\int_MCp^2u^p|\nabla_{\varphi}F|^2\omega_{\varphi}^n+\int_MCpK^2(\delta_*+\omega(r))^2(|\nabla_{\varphi}P|^2+\frac{C'}{r^2}tr_{\omega_{\varphi}}\omega_0)\omega_{\varphi}^n
\end{split}
\end{equation}
Now we plug (\ref{4.142})-(\ref{4.144}) back to (\ref{4.135New}),  and we obtain:
\begin{equation}
\begin{split}
&\int_M\frac{3p(p-1)}{4}u^{p-2}|\nabla_{\varphi}u|^2\omega_{\varphi}^n\le -\int_M\frac{p}{2}u^{p-1}e^{K(P-\tilde{P})^2}|\sqrt{-1}\partial\bar{\partial}P|^2_{\varphi}\omega_{\varphi}^n\\
&-(2K-C''pK^2(\delta_*+\omega(r))^2)\int_Mpu^p|\nabla_{\varphi}P|^2\omega_{\varphi}^n+\frac{C_1(p)K}{r^2}\int_Mpu^p\omega_{\varphi}^n\\
&+\frac{\tilde{C}p}{r^2}\int_M(K(\delta_*+\omega(r))+K^2(\delta_*+\omega(r))^2)u^ptr_{\omega_{\varphi}}\omega_0\omega_{\varphi}^n+Cp\int_Mu^p|\nabla_{\varphi}F|^2\omega_{\varphi}^n.
\end{split}
\end{equation}

We are going to assume that $\delta_*$ and $r$ are sufficiently small so that:
\begin{equation}
2K-C''pK^2(\delta_*+\omega(r))^2\ge \frac{3K}{2}.
\end{equation}
Therefore we get
\begin{equation}\label{K8}
\frac{3K}{2} \int_M u^p|\nabla_{\varphi}P|^2 \omega_{\varphi}^n \le   \frac{C(p) }{r^2}\int_M u^p (1+tr_{\omega_{\varphi}}\omega_0)\omega_\varphi^n +C(p)\int_M u^p |\nabla_\varphi F|^2\omega_\varphi^n + \frac{K}{r^2}C(p)\int_M u^p \omega_\varphi^n.
\end{equation}\\
We also assume that $K(\delta_*+\omega(r))^2\le 1$,  so that:
\begin{equation*}
1\le e^{K(P-\tilde{P})^2}\le e^{K(\delta_*+\omega(r))^2}\le e.  
\end{equation*}
Therefore,  we get:
\begin{equation*}
\begin{split}
&K\int_M|\nabla_{\varphi}P|^{2(p+1)}\omega_{\varphi}^n\le \frac{C'(p)}{r^2}\int_M|\nabla_{\varphi}P|^{2p}(1+tr_{\omega_{\varphi}}\omega_0)\omega_{\varphi}^n+C'(p)\int_Mu^p|\nabla_{\varphi}F|^2\omega_{\varphi}^n\\
&+\frac{KC'(p)}{r^2}\int_Mu^p\omega_{\varphi}^n.
\end{split}
\end{equation*}

Then we just need to apply the following version of Young's inequality to the right hand side above and the claimed estimate will follow:
\begin{equation*}
a^pb\le \frac{p}{p+1}\eps a^{p+1}+\frac{1}{p+1}\eps^{-p}b^{p+1},\,a,\,b>0,\,p>0,\,\eps>0.
\end{equation*}
\end{proof}

\noindent
Finally, we estimate $|\nabla_\varphi F|^2$ in term of $tr_{\omega_{\varphi}}\omega_0$ and $|\nabla_\varphi P|^2$. \\
\begin{lem}[\textbf{Estimate of $|\nabla_\varphi F|^2$}]\label{Lem316}
	 Let $(\varphi,F,P)$ be a solution of Pseudo Calabi Flow on $[0,T)\times M$ satisfying the conditions in Theorem \ref{MT31}, then for $p$ large enough, there exist $\lambda_p >0$ and a constant depends on  $ p, \|F\|_0,T$, such that  
\begin{align}
	\sup_{t\in (0,T)} t^{\lambda_p} \int_M |\nabla_\varphi F|^{2p}\omega_\varphi^n \le  C(p,||F||_0,T)
\end{align}\\
\end{lem}

\begin{rem}
	This Lemma implies Theorem \ref{MT31}. 
\end{rem}

\begin{cor}
	Using same notation above, we can combine Lemma \ref{Lem315} and Lemma \ref{Lem316} to get an $L^p$ estimate of $|\nabla_\varphi P|^2$ 
	\begin{align}
		\sup_{t\in (0,T)} t^{\lambda_p} \int_M |\nabla_\varphi P|^{2p}\omega_\varphi^n \le  C(p,||F||_0,T)
	\end{align}\\
\end{cor}

\begin{proof}[Proof of Lemma \ref{Lem316}]
First, we compute:
\begin{align}
	& (\partial_t-\Delta_{\varphi})\big(e^{\delta_0F^2}|\nabla_{\varphi}F|^2\big)e^{-\delta_0F^2} =(\partial_t-\Delta_{\varphi})(\delta_0F^2)|\nabla_{\varphi}F|^2-|\nabla_{\varphi}(\delta_0F^2)|^2|\nabla_{\varphi}F|^2 \notag \\
	&\qquad \quad +(\partial_t-\Delta_{\varphi})|\nabla_{\varphi}F|^2-2\nabla_{\varphi}(\delta_0F^2)\cdot \nabla_{\varphi}(|\nabla_{\varphi}F|^2) \notag\\
	&\qquad = -2\delta_0|\nabla_{\varphi}F|^4 + 2\delta_0F(\underline{R}-tr_{\omega_{\varphi}}(Ric(\omega_0)))|\nabla_{\varphi}F|^2 + (\partial_t-\Delta_{\varphi})|\nabla_{\varphi}F|^2 \notag\\
	&\qquad \quad -|\nabla_{\varphi}(\delta_0F^2)|^2|\nabla_{\varphi}F|^2 - 2Re\big(g_{\varphi}^{i\bar{j}}g_{\varphi}^{p\bar{q}}((\delta_0 F^2)_iF_{p\bar{j}}F_{\bar{q}}+(\delta_0 F^2)_iF_pF_{,\bar{q}\bar{j}})\big) 
\end{align}\\
We can find that
\begin{align}
	\partial_t|\nabla_{\varphi}F|^2=(-1)g_{\varphi}^{i\bar{q}}g_{\varphi}^{p\bar{j}}(\partial_t\varphi)_{i\bar{j}}F_pF_{\bar{q}}+2\nabla_{\varphi}F\cdot \nabla_{\varphi}\partial_t F \notag\\
	= - g_{\varphi}^{i\bar{q}}g_{\varphi}^{p\bar{j}}(F+P)_{i\bar{j}}F_pF_{\bar{q}}+2\nabla_{\varphi}F\cdot \nabla_{\varphi}\partial_t F
\end{align}\\
Also
\begin{align}
	\Delta_{\varphi}|\nabla_{\varphi}F|^2 = g_{\varphi}^{i\bar{q}}g_{\varphi}^{p\bar{j}}(F_{,ip}F_{,\bar{j}\bar{q}}+F_{i\bar{j}}F_{p\bar{q}}) + 2\nabla_{\varphi}F\cdot \nabla_{\varphi}\Delta_{\varphi}F + g_{\varphi}^{i\bar{q}}g_{\varphi}^{p\bar{j}}(Ric(\omega_{\varphi}))_{i\bar{j}}F_pF_{\bar{q}} \notag\\
	=g_{\varphi}^{i\bar{q}}g_{\varphi}^{p\bar{j}}(F_{,ip}F_{,\bar{j}\bar{q}}+F_{i\bar{j}}F_{p\bar{q}}) + 2\nabla_{\varphi}F\cdot \nabla_{\varphi}\Delta_{\varphi}F + g_{\varphi}^{i\bar{q}}g_{\varphi}^{p\bar{j}} \left( (Ric(\omega_0))_{i\bar{j}}-F_{i\bar j}\right) F_pF_{\bar{q}}
\end{align}\\
Therefore
\begin{align}
	& (\partial_t-\Delta_{\varphi})|\nabla_{\varphi}F|^2 =-g_{\varphi}^{i\bar{q}}g_{\varphi}^{p\bar{j}}(P_{i\bar{j}}+Ric_{0,i\bar j}) F_pF_{\bar{q}} \notag\\
	&\qquad\qquad +2\nabla_{\varphi}F\cdot \nabla_{\varphi}(\partial_t-\Delta_{\varphi})F -g_{\varphi}^{i\bar{q}}g_{\varphi}^{p\bar{j}}(F_{,ip}F_{,\bar{j}\bar{q}}+F_{i\bar{j}}F_{p\bar{q}}).
\end{align}\\
Notice that there is a completion of square
\begin{align}
	|\nabla_{\varphi}(\delta_0F^2)|^2|\nabla_{\varphi}F|^2 + 2Re\big(g_{\varphi}^{i\bar{j}}g_{\varphi}^{p\bar{q}}(\delta_0 F^2)_iF_pF_{,\bar{q}\bar{j}} \big) +g_{\varphi}^{i\bar{q}}g_{\varphi}^{p\bar{j}}F_{,ip}F_{,\bar{j}\bar{q}} \ge  0
\end{align}\\
So that we get (after the completion of square):
\begin{align}
	& (\partial_t-\Delta_{\varphi})\big(e^{\delta_0F^2}|\nabla_{\varphi}F|^2\big)e^{-\delta_0F^2} \le  -2\delta_0|\nabla_{\varphi}F|^4 + 2\delta_0 F (\underline{R}-tr_{\omega_{\varphi}}(Ric(\omega_0))) |\nabla_{\varphi}F|^2 \notag \\
	&\qquad \qquad -g_{\varphi}^{i\bar{q}}g_{\varphi}^{p\bar{j}}(P_{i\bar{j}}+Ric_{0,i\bar j}) F_pF_{\bar{q}} +2\nabla_{\varphi}F\cdot \nabla_{\varphi}(\partial_t-\Delta_{\varphi})F \notag\\
	&\qquad\qquad -g_{\varphi}^{i\bar{q}}g_{\varphi}^{p\bar{j}}F_{i\bar{j}}F_{p\bar{q}} + 2Re\big(g_{\varphi}^{i\bar{j}}g_{\varphi}^{p\bar{q}}((\delta_0 F^2)_iF_{p\bar{j}}F_{\bar{q}}\big) \notag\\
	&\quad \le  -2\delta_0|\nabla_{\varphi}F|^4 + C(2\delta_0 ||F||_0 +1)(1+tr_{\omega_{\varphi}}\omega_0) |\nabla_\varphi F|^2 -g_{\varphi}^{i\bar{q}}g_{\varphi}^{p\bar{j}}F_{i\bar{j}}F_{p\bar{q}} - g_{\varphi}^{i\bar{q}}g_{\varphi}^{p\bar{j}}P_{i\bar{j}} F_pF_{\bar{q}} \notag\\
	&\qquad\qquad  +2\nabla_{\varphi}F\cdot \nabla_{\varphi}(\partial_t-\Delta_{\varphi})F + 2Re\big(g_{\varphi}^{i\bar{j}}g_{\varphi}^{p\bar{q}}((\delta_0 F^2)_iF_{p\bar{j}}F_{\bar{q}}\big) 
\end{align}\\
We also have the following completion of square  
\begin{align}
	4|\nabla_{\varphi}(\delta_0 F^2)|^2|\nabla_\varphi F|^2 + 2Re\big(g_{\varphi}^{i\bar{j}}g_{\varphi}^{p\bar{q}}((\delta_0 F^2)_iF_{p\bar{j}}F_{\bar{q}})\big) +\frac{1}{4}g_{\varphi}^{i\bar{q}}g_{\varphi}^{p\bar{j}}F_{i\bar{j}}F_{p\bar{q}}\ge  0
\end{align}\\
Then we have 
\begin{align}
	& (\partial_t-\Delta_{\varphi})\big(e^{\delta_0F^2}|\nabla_{\varphi}F|^2\big)e^{-\delta_0F^2} \le  -(2\delta_0-16\delta_0^2 ||F||_0^2)|\nabla_\varphi F|^4 \notag\\
	&\qquad\qquad + C(2\delta_0 ||F||_0 +1)(1+tr_{\omega_{\varphi}}\omega_0) |\nabla_\varphi F|^2 \notag\\
	& \qquad\qquad -\frac{3}{4}g_{\varphi}^{i\bar{q}}g_{\varphi}^{p\bar{j}}F_{i\bar{j}}F_{p\bar{q}} - g_{\varphi}^{i\bar{q}}g_{\varphi}^{p\bar{j}}P_{i\bar{j}} F_pF_{\bar{q}} +2\nabla_{\varphi}F\cdot \nabla_{\varphi}(\partial_t-\Delta_{\varphi})F 
\end{align}\\

\noindent
Set $u=e^{\delta_0 F^2}|\nabla_\varphi F|^2$, then the inequality above becomes  
\begin{align}
	& (\partial_t-\Delta_{\varphi})u\le  -(2\delta_0-16\delta_0^2 ||F||_0^2)|\nabla_\varphi F|^2 u + C(2\delta_0 ||F||_0 +1)(1+tr_{\omega_{\varphi}}\omega_0)u  \notag\\
	&\quad\quad -\frac{3}{4} e^{\delta_0 F^2} g_{\varphi}^{i\bar{q}}g_{\varphi}^{p\bar{j}}F_{i\bar{j}}F_{p\bar{q}} - e^{\delta_0 F^2} g_{\varphi}^{i\bar{q}}g_{\varphi}^{p\bar{j}}P_{i\bar{j}} F_pF_{\bar{q}} + 2e^{\delta_0 F^2} \nabla_{\varphi}F\cdot \nabla_{\varphi}(\partial_t-\Delta_{\varphi})F \notag\\
	& \le  - (2\delta_0-16\delta_0^2 ||F||_0^2)|\nabla_\varphi F|^2 u + C(2\delta_0 ||F||_0 +1)(1+tr_{\omega_{\varphi}}\omega_0)u -\frac{3}{4} e^{\delta_0 F^2} g_{\varphi}^{i\bar{q}}g_{\varphi}^{p\bar{j}}F_{i\bar{j}}F_{p\bar{q}} \notag\\
	&\quad\quad - e^{\delta_0 F^2}\Delta_{\varphi}P|\nabla_{\varphi}F|^2 - e^{\delta_0 F^2}\frac{dd^c P\wedge d^c F\wedge d F\wedge \omega_{\varphi}^{n-2}}{\omega_{\varphi}^n} + 2e^{\delta_0 F^2} \nabla_{\varphi}F\cdot \nabla_{\varphi}\Delta_\varphi P
\end{align}\\
We insert the above inequality back to (\ref{K3}), and then we get 
\begin{align}
	&\partial_t\big( \int_M u^p\omega_\varphi^n\big)\le  - \int_M \frac{p(p-1)}{2} u^{p-2} |\nabla_\varphi u|^2_\varphi\omega_\varphi^n +\int_M u^p(|Ric|tr_\varphi\omega_0+\underline R)\omega_\varphi^n \notag\\
	&\qquad\qquad\qquad\qquad +\int_M pu^{p-1}\left( (\partial_t-\Delta_\varphi)u+  \frac{1}{2(p-1)}|\nabla_\varphi F|^2 u  \right) \omega_\varphi^n \notag\\
	 &\qquad \le  - \int_M \frac{p(p-1)}{2} u^{p-2} |\nabla_\varphi u|^2_\varphi\omega_\varphi^n - \frac{3}{4}\int_M pu^{p-1}e^{\delta_0 F^2} g_{\varphi}^{i\bar{q}}g_{\varphi}^{p\bar{j}}F_{i\bar{j}}F_{p\bar{q}}\omega_\varphi^n \notag\\
	 & +\int_M pu^{p}\left( -(2\delta_0-16\delta_0^2 ||F||_0^2-\frac{1}{2(p-1)})|\nabla_\varphi F|^2 +C(2\delta_0 ||F||_0 +1)(1+tr_{\omega_{\varphi}}\omega_0)\right)\omega_\varphi^n \notag\\
	 & \quad -\int_M pu^{p-1}e^{\delta_0 F^2}dd^c P\wedge d^c F\wedge d F\wedge \omega_{\varphi}^{n-2} + 2\int_M pu^{p-1}e^{\delta_0 F^2} \nabla_{\varphi}F\cdot \nabla_{\varphi}\Delta_\varphi P \omega_\varphi^n 
\end{align}\\
Just like what  we did with the estimate of $|\nabla_\varphi P|^2$, we are going to integrate by parts in the last two terms
\begin{align}
	& \int_M2pu^{p-1}e^{\delta_0 F^2}\nabla_{\varphi}F \cdot \nabla_{\varphi}\Delta_{\varphi}P\omega_{\varphi}^n = - \int_M2p(p-1)u^{p-2}e^{\delta_0 F^2}\nabla_{\varphi}u\cdot \nabla_{\varphi}F\, \Delta_{\varphi}P\omega_{\varphi}^n \notag \\
	&\quad - \int_M2pu^{p-1}e^{\delta_0 F^2}2\delta_0 F |\nabla_{\varphi}F|^2 \Delta_{\varphi}P\omega_{\varphi}^n - \int_M2pu^{p-1}e^{\delta_0 F^2}(\Delta_{\varphi}F)(\Delta_{\varphi}P)\omega_{\varphi}^n.
\end{align}\\
In the above,  we have:
\begin{align}
	&\int_M2p(p-1)u^{p-2}e^{\delta_0 F^2}\nabla_{\varphi}u\cdot \nabla_{\varphi}F\, \Delta_{\varphi}P\omega_{\varphi}^n \notag \\
	&\le  \frac{p(p-1)}{16}\int_Mu^{p-2}|\nabla_{\varphi}u|^2\omega_{\varphi}^n +C\int_M p(p-1)u^{p-1}e^{\delta_0 F^2}(1+tr_{\omega_{\varphi}}\omega_0)^2\omega_{\varphi}^n 
\end{align}\\
Next
\begin{align}
	& \int_M2pu^{p-1}e^{\delta_0 F^2}2\delta_0 F |\nabla_{\varphi}F|^2 \Delta_{\varphi}P\omega_{\varphi}^n = \int_M 4p\delta_0 u^p F\Delta_\varphi P\omega_\varphi^n \notag\\
	&\qquad\qquad\qquad \le  C\delta_0||F||_0 \int_M pu^p  (1+tr_{\omega_{\varphi}}\omega_0)\omega_\varphi^n 
\end{align}\\
Next
\begin{align}
	& \int_M2pu^{p-1}e^{\delta_0 F^2}(\Delta_{\varphi}F)(\Delta_{\varphi}P)\omega_{\varphi}^n \notag\\
	&\qquad\qquad \le  \frac{1}{4} \int_M pu^{p-1}e^{\delta_0 F^2}(\Delta_{\varphi}F)^2\omega_{\varphi}^n + 4\int_M pu^{p-1}e^{\delta_0 F^2}(\Delta_{\varphi}P)^2\omega_{\varphi}^n \notag\\
	&\qquad\qquad \le  \frac{1}{4} \int_M pu^{p-1}e^{\delta_0 F^2}|D_\varphi^2 F|^2 \omega_{\varphi}^n + C\int_M pu^{p-1}e^{\delta_0 F^2}(1+tr_{\omega_{\varphi}}\omega_0)^2\omega_{\varphi}^n 
\end{align}\\
Then we estimate 
\begin{align}
	&\int_M pu^{p-1}e^{\delta_0 F^2}dd^c P\wedge d^c F\wedge d F\wedge \omega_{\varphi}^{n-2} \notag \\
	&=-\int_Mp(p-1)u^{p-2}e^{\delta_0F^2}du\wedge d^cP\wedge d^cF\wedge dF\wedge \omega_{\varphi}^{n-2} \notag \\
	&\quad-\int_Mpu^{p-1}e^{\delta_0F^2}2\delta_0 F dF\wedge d^cP\wedge d^cF\wedge dF\wedge \omega_{\varphi}^{n-2} \notag \\
	&\quad +\int_M pu^{p-1}e^{\delta_0F^2}d^cP\wedge dd^cF\wedge dF\wedge \omega_{\varphi}^{n-2} \notag\\
	&= \int_Mp(p-1)u^{p-2}e^{\delta_0F^2}\left( -(\nabla_\varphi u\cdot\nabla_\varphi F)(\nabla_\varphi F\cdot\nabla_\varphi P) + (\nabla_\varphi u\cdot \nabla_\varphi P) |\nabla_\varphi F|^2 \right)  \omega_{\varphi}^{n} \notag \\
	&\quad -0+\int_M pu^{p-1}e^{\delta_0F^2}\left(-(\nabla_\varphi P\cdot \nabla_\varphi F)\Delta_\varphi F+ g_\varphi^{i\bar j}g_\varphi^{p\bar q}F_{i\bar q}F_{p}P_{\bar j} \right)  \omega_{\varphi}^{n} \notag\\
	&\le  \int_M 2p(p-1)u^{p-1}|\nabla_\varphi u|\cdot|\nabla_\varphi P| \omega_\varphi^n + \int_M pu^{p-1}e^{\delta_0F^2}\left( 16|\nabla_\varphi F|^2|\nabla_\varphi P|^2 +\frac{1}{16} (\Delta_\varphi F)^2\right) \omega_\varphi^n \notag\\
	&\quad + \int_M pu^{p-1}e^{\delta_0F^2}(\frac{1}{16}g_\varphi^{i\bar j}g_\varphi^{p\bar q}F_{i\bar q}F_{p\bar j}+16 |\nabla_\varphi P|^2|\nabla_\varphi F|^2)\omega_\varphi^n \notag\\
	&\le  \frac{p(p-1)}{16}\int_Mu^{p-2}|\nabla_{\varphi}u|^2\omega_{\varphi}^n  + 32\int_M p^2u^pe^{\delta_0F^2} |\nabla_\varphi P|^2 \omega_\varphi^n \notag\\
	&\qquad\qquad  + \int_M pu^{p-1}e^{\delta_0F^2}(\frac{1}{8}g_\varphi^{i\bar j}g_\varphi^{p\bar q}F_{i\bar q}F_{p\bar j})\omega_\varphi^n
\end{align}\\
Combining all the above together, we have 
\begin{align}
	&\partial_t\big( \int_M u^p\omega_\varphi^n\big)\le  - \int_M \frac{p(p-1)}{8} u^{p-2} |\nabla_\varphi u|^2_\varphi\omega_\varphi^n - \frac{1}{4}\int_M pu^{p-1}e^{\delta_0 F^2} g_{\varphi}^{i\bar{q}}g_{\varphi}^{p\bar{j}}F_{i\bar{j}}F_{p\bar{q}}\omega_\varphi^n \notag\\
	 & +\int_M pu^{p}\left( -(2\delta_0-16\delta_0^2 ||F||_0^2-\frac{1}{2(p-1)})|\nabla_\varphi F|^2 +C(2\delta_0 ||F||_0 +1)(1+tr_{\omega_{\varphi}}\omega_0)\right)\omega_\varphi^n \notag\\
	 & \quad +C\int_M p(p-1)u^{p-1}e^{\delta_0 F^2}(1+tr_{\omega_{\varphi}}\omega_0)^2\omega_{\varphi}^n +C\delta_0||F||_0 \int_M pu^p  (1+tr_{\omega_{\varphi}}\omega_0)\omega_\varphi^n \notag\\
	 &\quad + C\int_M pu^{p-1}e^{\delta_0 F^2}(1+tr_{\omega_{\varphi}}\omega_0)^2\omega_{\varphi}^n + 32\int_M p^2u^pe^{\delta_0F^2} |\nabla_\varphi P|^2 \omega_\varphi^n 
\end{align}\\
The constant $\delta_0$ is chosen to be small enough that 
\begin{align}
	\delta_0 (||F||_0^2+||F||_0)<\frac{1}{2}
\end{align}
And $p$ large enough that 
\begin{align}
	\frac{1}{p-1} < \delta_0
\end{align}\\
Then we get
\begin{align}
	&\partial_t\big( \int_M u^p\omega_\varphi^n\big)\le  -\delta_0 \int_M u^{p+1}\omega_\varphi^n \notag\\
	 &\quad\quad + C(p)\left( \int_M u^p (tr_{\omega_{\varphi}}\omega_0+1)\omega_\varphi^n + \int_M u^{p-1} (tr_{\omega_{\varphi}}\omega_0+1)^2 \omega_\varphi^n + \int_M u^{p}|\nabla_\varphi P|^2 \omega_\varphi^n \right) \notag\\
	&\quad\quad \le  -\frac{\delta_0}{2} \int_M u^{p+1}\omega_\varphi^n + C(p,\delta_0 )\left( \int_M (tr_{\omega_{\varphi}}\omega_0+1)^{p+1} \omega_\varphi^n + \int_M |\nabla_\varphi P|^{2(p+1)} \omega_\varphi^n \right)
\end{align}
where the last line above is obtained by using Young's inequality. \\

\noindent
Recall Lemma \ref{Lem315}, we already have an estimate of $|\nabla_\varphi P|^2$ in term of $|\nabla_\varphi F|^2$, so we have 
\begin{align}
	& \partial_t\big( \int_M u^p\omega_\varphi^n\big)\le  -\frac{\delta_0}{2} \int_M u^{p+1}\omega_\varphi^n + \frac{C'(p)}{K-1}\int_M u^{p+1}\omega_\varphi^n \notag\\
	&\quad\quad\quad\quad\quad\quad + C(p) \int_M (tr_{\omega_{\varphi}}\omega_0+1)^{p+1} \omega_\varphi^n + C(p,K,r,\delta_0)
\end{align}\\
By choosing $K$ sufficiently large so that 
\begin{align}
	\frac{C'(p)}{K-1} < \frac{\delta_0}{4}
\end{align}\\
Also notice that 
\begin{align}
	tr_{\omega_{\varphi}}\omega_0+1\le  C_1 (n+\Delta \varphi)^n+C_2
\end{align}\\
We then have 
\begin{align}
	\partial_t\big( \int_M u^p\omega_\varphi^n\big)\le  -c \int_M u^{p+1}\omega_\varphi^n + C_3 \int_M (n+\Delta\varphi)^{n(p+1)} + C_4
\end{align}\\
\noindent
Let \( \lambda_p =1+n^2(p+1) \), we then compute:
\begin{align}
	& \partial_t \left( t^{\lambda_p} \int_M u^p \, \omega_{\varphi}^n \right) \le  \lambda_p t^{\lambda_p - 1} \int_M u^p \, \omega_{\varphi}^n - c t^{\lambda_p} \int_M u^{p +1} \, \omega_{\varphi}^n \notag\\
	&\qquad\qquad\qquad\qquad\qquad\qquad + C_3t^{\lambda_p} \int_M (n+\Delta\varphi)^{n(p+1)} \omega_{\varphi}^n + C_4t^{\lambda_p}  \notag \\
	&\qquad\qquad\qquad = t^{n^2(p+1)-p} \left( \lambda_p t^p \int_M u^p \, \omega_{\varphi}^n -ct^{p+1}\int_M u^{p+1}\,\omega_\varphi^n \right) \notag\\
	&\qquad\qquad\qquad\qquad\qquad\qquad + C_3t^{\lambda_p} \int_M (n+\Delta\varphi)^{n(p+1)} \omega_{\varphi}^n + C_4t^{\lambda_p} \notag\\
	&\qquad\qquad\qquad = t^{n^2(p+1)-p} \left( \lambda_p t^p \int_M u^p \, \omega_{\varphi}^n -c(t^p \int_M u^p \, \omega_{\varphi}^n)^{1+\frac{1}{p}} \right) \notag\\
	&\qquad\qquad\qquad\qquad\qquad\qquad + C_3t^{\lambda_p} \int_M (n+\Delta\varphi)^{n(p+1)} \omega_{\varphi}^n + C_4t^{\lambda_p}
\end{align}\\
Note that the first item of right-hand side above is bounded on any finite time interval. By Lemma \ref{Lem313}, we know the second item is also bounded. We integrate in $t$, and see that $ t^{\lambda_p} \int_M u^p \, \omega_{\varphi}^n $ is bounded from above on any finite time interval, which depends only on $ p, \|F\|_0$. i.e. 
\begin{align}
	\sup_{t\in (0,T)} t^{\lambda_p} \int_M |\nabla_\varphi F|^{2p}\omega_\varphi^n \le  C(p,||F||_0,T)
\end{align}\\
\end{proof}

\subsection{When initial volume form is small }

The goal of this subsection is to prove the following theorem. 

\begin{thm}\label{MT32}
	Let $\varphi$ be a solution to the PCF on $M\times [0,T)$.  Let $p>1$,  then there exists $\delta>0$ sufficiently small,  depending on $p$,  $n$ and the background metric,  such that if 
		\begin{align}
			\max_{t\in [0,T)}||F(\cdot,t)||_0\le \delta,
		\end{align}
	then for any $0<\epsilon <T$,  one has 
		\begin{align}
			\sup_{t\in [\epsilon ,T)}||\nabla_{\varphi} F(\cdot,t)||_{L^p(\omega_{\varphi}^n)}\le C,
		\end{align}
	where $C$ depends on $p$,  $n$,  $\epsilon$,  $T$,  and the background metric.\\
\end{thm}

\noindent
Doing a same argument as Corollary \ref{Cor31} and \ref{Cor32}, we have the following result. 

\begin{cor}\label{Cor33}
	Let $\varphi$ be a solution of Pseudo Calabi Flow on $[0,T)\times M,\ (t<\infty)$ with $\sup_{t\in [0,t)}||F(\cdot,t)||_0<\delta$ where $\delta$ is the constant in Theorem \ref{MT32}, then the Pseudo Calabi Flow can be extended beyond T, i.e. $\varphi\in C^\infty \big([0,T+\eps) ,C^{\infty}(M)\big)$ for a small constant $\eps$. \\
\end{cor}

\noindent
We start with estimating $||\nabla_{\varphi}P||_{L^p}$ in terms of $tr_{\omega_{\varphi}}\omega_0$ and $|\nabla_{\varphi}F|$.\\

\begin{lem}[\textbf{Estimate of $|\nabla_\varphi P|^2$}]\label{Lem321}
	Let $(\varphi,F,P)$ satisfy the following equations 
	\begin{align}
		\begin{cases}
			F= \log \frac{\omega_\varphi^n}{\omega_0^n} \\
			\Delta_\varphi P=-tr_\varphi Ric_{\omega_0} +\underline R
		\end{cases}
	\end{align}
	with $P$ normalized as $\int_M P\,\omega_\varphi^n=0$. If $||P||_0\le V'$, then there exist a constant $C$ depends on $V', p$ and background metric such that 
	\begin{align}
		\int_M |\nabla_\varphi P |^{p+1}\omega_\varphi^n \le C\left( \int_M (tr_\varphi \omega_0 +1)^{p+1}\omega_\varphi^n + \int_M |\nabla_\varphi F|^{2(p+1)}\omega_\varphi^n \right)
	\end{align}\\
\end{lem}

\begin{proof}
Let $\delta_1>0$,  we compute:
\begin{align}
	& \Delta_{\varphi}\big(e^{\delta_1P^2}|\nabla_{\varphi}P|^2\big)e^{-\delta_1P^2} =\left( \Delta_{\varphi}(\delta_1P^2)+|\nabla_{\varphi}(\delta_1P^2)|^2 \right)(|\nabla_\varphi P|^2)  \notag \\
	&\qquad\qquad\qquad \quad +\Delta_{\varphi}(|\nabla_{\varphi}P|^2) +2\nabla_{\varphi}(\delta_1P^2)\cdot \nabla_{\varphi}(|\nabla_{\varphi}P|^2).
\end{align}\\
In the above,  we have:
\begin{align}
	& \Delta_{\varphi}(|\nabla_{\varphi}P|^2)=g_{\varphi}^{i\bar{q}}g_{\varphi}^{p\bar{j}}P_{,iq}P_{,\bar{p}\bar{j}}+g_{\varphi}^{i\bar{q}}g_{\varphi}^{p\bar{j}}P_{i\bar{j}}P_{p\bar{q}} \notag\\
	&\qquad +2\nabla_{\varphi}P\cdot \nabla_{\varphi}\Delta_{\varphi}P+g_{\varphi}^{i\bar{q}}g_{\varphi}^{p\bar{j}}(Ric(\omega_{\varphi}))_{i\bar{j}}P_pP_{\bar{q}}.
\end{align}\\
Also
\begin{align}
2\nabla_{\varphi}(\delta_1P^2)\cdot \nabla_{\varphi}(|\nabla_{\varphi}P|^2)=2Re\big(g_{\varphi}^{i\bar{j}}g_{\varphi}^{p\bar{q}}((\delta_1P^2)_iP_{p\bar{j}}P_{\bar{q}}+(\delta_1P^2)_iP_pP_{,\bar{q}\bar{j}})\big).
\end{align}\\
We have the following completion of squares:
\begin{align}
|\nabla_{\varphi}(\delta_1P^2)|^2 |\nabla_\varphi P|^2 +2Re\big(g_{\varphi}^{i\bar{j}}g_{\varphi}^{p\bar{q}}(\delta_1P^2)_iP_pP_{,\bar{q}\bar{j}}\big)+g_{\varphi}^{i\bar{q}}g_{\varphi}^{p\bar{j}}P_{,ip}P_{,\bar{q}\bar{j}}\ge  0
\end{align}\\
And 
\begin{align}
	4|\nabla_{\varphi}(\delta_1P^2)|^2|\nabla_\varphi P|^2 + 2Re\big(g_{\varphi}^{i\bar{j}}g_{\varphi}^{p\bar{q}}((\delta_1P^2)_iP_{p\bar{j}}P_{\bar{q}}) +\frac{1}{4}g_{\varphi}^{i\bar{q}}g_{\varphi}^{p\bar{j}}P_{i\bar{j}}P_{p\bar{q}}\ge  0
\end{align}\\
So we obtain:
\begin{align}
	&\Delta_{\varphi}\big(e^{\delta_1P^2}|\nabla_{\varphi}P|^2\big)e^{-\delta_1P^2} \ge  (2\delta_1-16\delta_1^2P^2) |\nabla_{\varphi}P|^4 +2\delta_1P\Delta_{\varphi}P|\nabla_{\varphi}P|^2 \notag \\
	&\quad +\frac{3}{4}g_{\varphi}^{i\bar{q}}g_{\varphi}^{p\bar{j}}P_{i\bar{j}}P_{p\bar{q}}+2\nabla_{\varphi}P\cdot \nabla_{\varphi}\Delta_{\varphi}P+g_{\varphi}^{i\bar{q}}g_{\varphi}^{p\bar{j}}(Ric(\omega_{\varphi}))_{i\bar{j}}P_pP_{\bar{q}} 
\end{align}\\
In the above,  we can estimate:
\begin{align}
	2\delta_1P\Delta_{\varphi}P\ge  -C\delta_1||P||_0 (tr_{\omega_{\varphi}}\omega_0+1).
\end{align}\\
Also
\begin{align}
	&g_{\varphi}^{i\bar{q}}g_{\varphi}^{p\bar{j}}(Ric(\omega_{\varphi}))_{i\bar{j}}P_pP_{\bar{q}} =g_{\varphi}^{i\bar{q}}g_{\varphi}^{p\bar{j}}(Ric(\omega_0))_{i\bar{j}}P_pP_{\bar{q}} - g_{\varphi}^{i\bar{q}}g_{\varphi}^{p\bar{j}}F_{i\bar{j}}P_pP_{\bar{q}} \notag \\
	&\qquad \ge  -C tr_{\omega_{\varphi}}\omega_0|\nabla_{\varphi}P|^2 - \Delta_{\varphi}F|\nabla_{\varphi}P|^2 - \frac{dd^cF\wedge d^c P\wedge d P\wedge \omega_{\varphi}^{n-2}}{\omega_{\varphi}^n}.
\end{align}\\
So we get:
\begin{align}
	&\Delta_{\varphi}\big(e^{\delta_1P^2}|\nabla_{\varphi}P|^2\big)e^{-\delta_1P^2} \notag\\
	&\qquad \ge  (2\delta_1-16\delta_1^2||P||_0^2) |\nabla_{\varphi}P|^4 -C\delta_1||P||_0 (tr_{\omega_{\varphi}}\omega_0+1)  +\frac{3}{4}g_{\varphi}^{i\bar{q}}g_{\varphi}^{p\bar{j}}P_{i\bar{j}}P_{p\bar{q}}   \notag \\
	&\qquad\quad -C tr_{\omega_{\varphi}}\omega_0|\nabla_{\varphi}P|^2 +2\nabla_{\varphi}P\cdot \nabla_{\varphi}\Delta_{\varphi}P - \Delta_{\varphi}F|\nabla_{\varphi}P|^2 - \frac{dd^cF\wedge d^cP\wedge dP\wedge \omega_{\varphi}^{n-2}}{\omega_{\varphi}^n}.
\end{align}\\
Denote $u=e^{\delta_1 P^2}|\nabla_{\varphi}P|^2$,  then the above inequality is equivalent to:
\begin{align}
	&\Delta_{\varphi}u \notag \ge  (2\delta_1-16\delta_1^2||P||_0^2) |\nabla_{\varphi}P|^2u -Ctr_{\omega_{\varphi}}\omega_0 u \notag\\
	&\quad  + \frac{3}{4} e^{\delta_1P^2}g_{\varphi}^{i\bar{q}}g_{\varphi}^{p\bar{j}}P_{i\bar{j}}P_{p\bar{q}} - C\delta_1||P||_0e^{\delta_1P^2}(tr_{\omega_{\varphi}}\omega_0 +1)  \notag \\
	&\quad+2e^{\delta_1P^2}\nabla_{\varphi}P\cdot\nabla_{\varphi}\Delta_{\varphi}P - \Delta_{\varphi}Fu - e^{\delta_1P^2}\frac{dd^cF\wedge d^cP\wedge dP\wedge \omega_{\varphi}^{n-2}}{\omega_{\varphi}^n}.
\end{align}\\

\noindent
Let $p>1$,  we have:
\begin{align}
	\Delta_\varphi u^p = p u^{p-1}\Delta_\varphi u+p(p-1)u^{p-2}|\nabla_\varphi u|^2
\end{align}\\
Integral with respect to $\omega_\varphi^n$, we have
\begin{align}
	&\int_M p(p-1)u^{p-2}|\nabla_{\varphi}u|^2\omega_{\varphi}^n = - \int_M pu^{p-1}\Delta_\varphi u \omega_\varphi^n  \notag \\
	&\le \int_M pu^p\left( -(2\delta_1-16\delta_1^2||P||_0^2)|\nabla_{\varphi}P|^2+Ctr_{\omega_{\varphi}}\omega_0\right) \omega_{\varphi}^n - \int_M\frac{3}{4} pu^{p-1}e^{\delta_1P^2}g_{\varphi}^{i\bar{q}}g_{\varphi}^{p\bar{j}}P_{i\bar{j}}P_{p\bar{q}}\omega_{\varphi}^n \notag \\
	&\quad + \int_M pu^{p-1}e^{\delta_1P^2} C\delta_1||P||_0 (tr_{\omega_{\varphi}}\omega_0+1)\omega_\varphi^n      -\int_M2pu^{p-1}e^{\delta_1P^2}\nabla_{\varphi}P\cdot \nabla_{\varphi}\Delta_{\varphi}P\omega_{\varphi}^n \notag \\
	&\quad +\int_M pu^p\Delta_{\varphi}F\omega_{\varphi}^n + \int_M pu^{p-1}e^{\delta_1P^2}dd^cF\wedge d^cP\wedge dP\wedge \omega_{\varphi}^{n-2}.
\end{align}\\
We are going to integrate by parts in the last two terms:
\begin{align}
	&-\int_M2pu^{p-1}e^{\delta_1P^2}\nabla_{\varphi}P\cdot \nabla_{\varphi}\Delta_{\varphi}P\omega_{\varphi}^n =\int_M2p(p-1)u^{p-2}e^{\delta_1P^2}\nabla_{\varphi}u\cdot \nabla_{\varphi}P\Delta_{\varphi}P\omega_{\varphi}^n \notag \\
	&\quad+\int_M2pu^{p-1}e^{\delta_1P^2}2\delta_1 P |\nabla_{\varphi}P|^2 \Delta_{\varphi}P\omega_{\varphi}^n + \int_M2pu^{p-1}e^{\delta_1P^2}(\Delta_{\varphi}P)^2\omega_{\varphi}^n.
\end{align}\\
In the above,  we have:
\begin{align}
	&\int_M2p(p-1)u^{p-2}e^{\delta_1P^2}\nabla_{\varphi}u\cdot \nabla_{\varphi}P\Delta_{\varphi}P\omega_{\varphi}^n \notag \\
	&\le \frac{p(p-1)}{16}\int_Mu^{p-2}|\nabla_{\varphi}u|^2\omega_{\varphi}^n+C \int_M p(p-1)u^{p-1}e^{\delta_1P^2}(1+tr_{\omega_{\varphi}}\omega_0)^2\omega_{\varphi}^n.
\end{align}\\
Next
\begin{align}
	\int_M2pu^{p-1}e^{\delta_1P^2}(\Delta_{\varphi}P)^2\omega_{\varphi}^n\le C \int_M pu^{p-1}e^{\delta_1P^2}(1+tr_{\omega_{\varphi}}\omega_0)^2\omega_{\varphi}^n.
\end{align}\\
Next
\begin{align}
	&\int_M2pu^{p-1}e^{\delta_1P^2}2\delta_1P |\nabla_{\varphi}P|^2 \Delta_{\varphi}P\omega_{\varphi}^n = \int_M 4p\delta_1 u^p P\Delta_\varphi P\omega_\varphi^n \notag\\
	& \qquad \le \delta_1C ||P||_0\int_M pu^p  (1+tr_{\omega_{\varphi}}\omega_0)\omega_\varphi^n 
\end{align}\\
Moreover
\begin{align}
	&\int_Mpu^{p-1}e^{\delta_1P^2}dd^cF\wedge d^cP\wedge dP\wedge \omega_{\varphi}^{n-2} \notag \\
	&=-\int_Mp(p-1)u^{p-2}e^{\delta_1P^2}du\wedge d^cF\wedge d^cP\wedge dP\wedge \omega_{\varphi}^{n-2} \notag \\
	&\quad-\int_Mpu^{p-1}e^{\delta_1P^2}2\delta_1P dP\wedge d^cF\wedge d^cP\wedge dP\wedge \omega_{\varphi}^{n-2} \notag \\
	&\quad +\int_Mpu^{p-1}e^{\delta_1P^2}d^cF\wedge dd^cP\wedge dP\wedge \omega_{\varphi}^{n-2} \notag\\
	&= \int_Mp(p-1)u^{p-2}e^{\delta_1P^2}\left( -(\nabla_\varphi u\cdot\nabla_\varphi P)(\nabla_\varphi F\cdot\nabla_\varphi P) + (\nabla_\varphi u\cdot \nabla_\varphi F) |\nabla_\varphi P|^2 \right)  \omega_{\varphi}^{n} \notag \\
	&\quad-0+\int_M pu^{p-1}e^{\delta_1P^2}\left(-(\nabla_\varphi P\cdot \nabla_\varphi F)\Delta_\varphi P+ g_\varphi^{i\bar j}g_\varphi^{p\bar q}P_{i\bar q}P_{p}F_{\bar j} \right)  \omega_{\varphi}^{n} \notag\\
	&= -\int_Mp(p-1)u^{p-2}e^{\delta_1P^2} (\nabla_\varphi u\cdot\nabla_\varphi P)(\nabla_\varphi F\cdot\nabla_\varphi P)\omega_\varphi^n - \int_M (p-1)u^p\Delta_{\varphi}F\omega_{\varphi}^n \notag\\
	&\quad + \int_M pu^{p-1}e^{\delta_1P^2}\left(-(\nabla_\varphi P\cdot \nabla_\varphi F)\Delta_\varphi P+ g_\varphi^{i\bar j}g_\varphi^{p\bar q}P_{i\bar q}P_{p}F_{\bar j} \right)  \omega_{\varphi}^{n}
\end{align}\\
Hence we have 
\begin{align}
	&\int_Mpu^{p-1}e^{\delta_1P^2}dd^cF\wedge d^cP\wedge dP\wedge \omega_{\varphi}^{n-2} + \int_Mpu^p\Delta_{\varphi}F\omega_{\varphi}^n \notag\\
	& = -\int_Mp(p-1)u^{p-2}e^{\delta_1P^2} (\nabla_\varphi u\cdot\nabla_\varphi P)(\nabla_\varphi F\cdot\nabla_\varphi P)\omega_\varphi^n + \int_M u^p\Delta_{\varphi}F\omega_{\varphi}^n \notag\\
	&\quad + \int_M pu^{p-1}e^{\delta_1P^2}\left(-(\nabla_\varphi P\cdot \nabla_\varphi F)\Delta_\varphi P+ g_\varphi^{i\bar j}g_\varphi^{p\bar q}P_{i\bar q}P_{p}F_{\bar j} \right)  \omega_{\varphi}^{n} \notag\\
	&\le \int_M p(p-1)u^{p-1}|\nabla_\varphi u|\cdot|\nabla_\varphi F| \omega_\varphi^n -\int_M pu^{p-1}\nabla_{\varphi}u\cdot \nabla_{\varphi}F\omega_{\varphi}^n \notag\\
	& +C \int_M pu^{p-1}e^{\delta_1P^2}(1+tr_{\omega_{\varphi}}\omega_0) |\nabla_{\varphi}F|\cdot |\nabla_{\varphi}P| \omega_\varphi^n + \int_M pu^{p-1}e^{\delta_1P^2}g_\varphi^{i\bar j}g_\varphi^{p\bar q}P_{i\bar q}P_{p}F_{\bar j}   \omega_{\varphi}^{n} \notag\\
	&\le \int_M p^2u^{p-1}|\nabla_\varphi u|\cdot|\nabla_\varphi F| \omega_\varphi^n + C\int_M pu^{p-1}e^{\delta_1P^2}(1+tr_{\omega_{\varphi}}\omega_0)(|\nabla_\varphi P|^2+|\nabla_\varphi F|^2)\omega_\varphi^n \notag\\
	&\quad + \int_M pu^{p-1}e^{\delta_1P^2}(\frac{1}{16}g_\varphi^{i\bar j}g_\varphi^{p\bar q}P_{i\bar q}P_{p\bar j}+16 |\nabla_\varphi P|^2|\nabla_\varphi F|^2)\omega_\varphi^n \notag \\
	&\le \int_M\frac{p(p-1)}{16}u^{p-2}|\nabla_{\varphi}u|^2\omega_{\varphi}^n+\int_M 32p^2 u^p|\nabla_{\varphi}F|^2\omega_{\varphi}^n + C\int_M pu^{p}(1+tr_{\omega_{\varphi}}\omega_0)\omega_\varphi^n \notag\\
	&\quad +C\int_M pu^{p-1}e^{\delta_1P^2}|\nabla_\varphi F|^2 (1+tr_{\omega_{\varphi}}\omega_0) \omega_\varphi^n + \int_M pu^{p-1}e^{\delta_1P^2}\frac{1}{16}g_\varphi^{i\bar j}g_\varphi^{p\bar q}P_{i\bar q}P_{p\bar j} \omega_\varphi^n
\end{align}\\
We then obtain:
\begin{align}
	&\int_M \frac{p(p-1)}{2} u^{p-2}|\nabla_{\varphi}u|^2\omega_{\varphi}^n \le -\int_M pu^p(2\delta_1-16\delta_1^2||P||_0^2)|\nabla_{\varphi}P|^2 \omega_{\varphi}^n \notag\\
	&\quad +C(p) \int_M u^{p-1}e^{\delta_1P^2} \left((1+tr_{\omega_{\varphi}}\omega_0)^2 + (1+tr_{\omega_{\varphi}}\omega_0)|\nabla_\varphi F|^2\right) \omega_\varphi^n \notag\\
	&\quad + C(p)\delta_1||P||_0 \int_M u^p (1+tr_{\omega_{\varphi}}\omega_0)\omega_\varphi^n +C(p)\int_M u^p |\nabla_\varphi F|^2\omega_\varphi^n  
\end{align}\\
By choosing $\delta_1$ sufficiently small so that 
\begin{align}
	16\delta_1||P||_0^2 <1,\ \delta_1||P||_0 <1,\ e^{\delta_1P^2}<2
\end{align}\\
then we have 
\begin{align}\label{d18}
	\delta_1 \int_M u^p|\nabla_{\varphi}P|^2 \omega_{\varphi}^n \le  C(p) \int_M u^p (1+tr_{\omega_{\varphi}}\omega_0)\omega_\varphi^n +C(p)\int_M u^p |\nabla_\varphi F|^2\omega_\varphi^n \notag\\
	\quad + C(p) \int_M u^{p-1}e^{\delta_1P^2} \left((1+tr_{\omega_{\varphi}}\omega_0)^2 + (1+tr_{\omega_{\varphi}}\omega_0)|\nabla_\varphi F|^2\right) \omega_\varphi^n
\end{align}\\
Using Young's inequality, we get
\begin{align}
	u^p A\le \frac{p}{p+1}u^{p+1}+\frac{1}{p+1}A^{p+1}
\end{align}\\
We can choose $A$ to be $C'(1+tr_{\omega_{\varphi}}\omega_0)$ or $C'|\nabla_\varphi F|^2$ as in the last line. Also 
\begin{align}
	u^{p-1}AB\le \frac{p-1}{p+1}u^{p+1}+\frac{1}{p+1}A^{p+1}+\frac{1}{p+1}B^{p+1}
\end{align}\\
And we choose $A=\sqrt{C'}(1+tr_{\omega_{\varphi}}\omega_0)$ , $B=\sqrt{C'}|\nabla_\varphi F|^2$. Here $C'$ is a very large constant depend on $\delta_1$,
\begin{align}
	C'= 10 C(p)\delta_1^{-1}
\end{align}
Then the right hand side of (\ref{d18}) will become
\begin{align}
	\frac{4}{10}\delta_1 \int_M u^{p+1}\omega_\varphi^n + \delta_1^{-p}C(p)\left( \int_M (1+tr_{\omega_{\varphi}}\omega_0)^{p+1}\omega_\varphi^n + \int_M |\nabla_\varphi F|^{2(p+1)}\omega_\varphi^n \right)
\end{align}\\
Finally, we get our estimate 
\begin{align}
	\int_M u^{p+1}\omega_\varphi^n \le \delta_1^{-p-1}C(p)\left( \int_M (1+tr_{\omega_{\varphi}}\omega_0)^{p+1}\omega_\varphi^n + \int_M |\nabla_\varphi F|^{2(p+1)}\omega_\varphi^n \right)
\end{align}\\
Where $\delta_1$ essentially depends on $||P||_0$, $C(p)$ depends on $p$ and background metric. \\ 
\end{proof}

\noindent
Lastly, we estimate $|\nabla_\varphi F|^2$. 
 \begin{lem}[$W^{1,p}$ bound of $F$]\label{Lem322}
 	Let $\varphi$ be a solution to PCF on $M\times [0,T)$.  Denote $F=\log\frac{\omega_{\varphi}^n}{\omega_0^n}$.  For any $p>1$,  there exists $V_0>0$ sufficiently small,  depending only on $p$,  the background metric,  such that if $||F||_{0,M\times [0,T)}\le V_0$,  then one has:
	\begin{align}\label{equa3210}
		\sup_{t\in(0,T)}t^{p+1}\int_M|\nabla_{\varphi}F|^{2p}\omega_{\varphi}^n\le C, 
	\end{align}
	where $C$ depends on $p$,  the background metric,  $n$ and $T$.\\
 \end{lem}
 
 \begin{rem}
 	The inequality \eqref{equa3210} immediately implies Theorem \ref{MT32}. \\
 \end{rem}
 
 \begin{proof}
By Lemme \ref{Lem316}, we have:
\begin{align}
	& (\partial_t-\Delta_{\varphi})\big(e^{KF^2}|\nabla_{\varphi}F|^2\big)e^{-KF^2} =(\partial_t-\Delta_{\varphi})(KF^2)|\nabla_{\varphi}F|^2-|\nabla_{\varphi}(KF^2)|^2|\nabla_{\varphi}F|^2 \notag \\
	&\quad\qquad\qquad\qquad +(\partial_t-\Delta_{\varphi})|\nabla_{\varphi}F|^2-2\nabla_{\varphi}(KF^2)\cdot \nabla_{\varphi}(|\nabla_{\varphi}F|^2) \notag\\
	&\qquad\qquad = -2K|\nabla_{\varphi}F|^4 + 2KF(\underline{R}-tr_{\omega_{\varphi}}(Ric(\omega_0)))|\nabla_{\varphi}F|^2 + (\partial_t-\Delta_{\varphi})|\nabla_{\varphi}F|^2 \notag\\
	& \quad\qquad\qquad -|\nabla_{\varphi}(KF^2)|^2|\nabla_{\varphi}F|^2 - 2Re\big(g_{\varphi}^{i\bar{j}}g_{\varphi}^{p\bar{q}}((K F^2)_iF_{p\bar{j}}F_{\bar{q}}+(K F^2)_iF_pF_{,\bar{q}\bar{j}})\big) 
\end{align}\\
Integral on $M$, we have (also similar to computation in Lemma \ref{Lem316}) 
\begin{align}
	&\partial_t\big( \int_M u^p\omega_\varphi^n\big)\le - \int_M \frac{p(p-1)}{8} u^{p-2} |\nabla_\varphi u|^2_\varphi\omega_\varphi^n - \frac{1}{8}\int_M pu^{p-1}e^{K F^2} g_{\varphi}^{i\bar{q}}g_{\varphi}^{p\bar{j}}F_{i\bar{j}}F_{p\bar{q}}\omega_\varphi^n \notag\\
	 & +\int_M pu^{p}\left( -(2K-16K^2 ||F||_0^2-\frac{1}{2(p-1)})|\nabla_\varphi F|^2 +C(2K||F||_0 +1)(1+tr_{\omega_{\varphi}}\omega_0)\right)\omega_\varphi^n \notag\\
	 & \quad +C\int_M p(p-1)u^{p-1}e^{K F^2}(1+tr_{\omega_{\varphi}}\omega_0)^2\omega_{\varphi}^n +CK||F||_0 \int_M pu^p  (1+tr_{\omega_{\varphi}}\omega_0)\omega_\varphi^n \notag\\
	 &\quad + C\int_M pu^{p-1}e^{K F^2}(1+tr_{\omega_{\varphi}}\omega_0)^2\omega_{\varphi}^n + 32\int_M p^2u^pe^{KF^2} |\nabla_\varphi P|^2 \omega_\varphi^n 
\end{align}\\
By choosing $V$ small enough (depending on $K$)
\begin{align}
	16K V^2 <\frac{1}{2}
\end{align}
and by Young's inequality, we have 
\begin{align}
	&\partial_t\big( \int_M u^p\omega_\varphi^n\big)\le -(\frac{2}{3} K-C) \int_M u^{p+1}\omega_\varphi^n \notag\\
	&\qquad\qquad  +C(p,||P||_0,||F||_0) \int_M \big( |\nabla_\varphi P|^{2(p+1)} +(tr_{\omega_{\varphi}}\omega_0 +1)^{p+1}\big) \omega_\varphi^n
\end{align}\\
By lemma \ref{Lem31} and lemma \ref{Lem321}, we then have
\begin{align}
	&\partial_t\big( \int_M u^p\omega_\varphi^n\big)\le -\big(\frac{2}{3} K-C(p,||P||_0,||F||_0)\big) \int_M u^{p+1}\omega_\varphi^n +C(p,V) 
\end{align}\\
Choose $K$ sufficiently large, the first result in the lemma follows. 
\begin{align}
 		\partial_t\left( \int_M |\nabla_\varphi F|^{2p}\omega_\varphi^n\right) \le -K \left( \int_M |\nabla_\varphi F|^{2(p+1)}\omega_\varphi^n\right) + C
 	\end{align}\\
Adding time variable, we have 
	\begin{align}
		&\partial_t\left( t^{p+1} \int_M |\nabla_\varphi F|^{2p}\omega_\varphi^n\right) \le (p+1)t^p  \int_M |\nabla_\varphi F|^{2p}\omega_\varphi^n \notag\\
		&\qquad\qquad\qquad\qquad\qquad\qquad\qquad  -K \left( t^{p+1} \int_M |\nabla_\varphi F|^{2(p+1)}\omega_\varphi^n\right) + C t^{p+1} \notag\\
		&\quad\quad \le (p+1)\left( t^p  \int_M |\nabla_\varphi F|^{2p}\omega_\varphi^n \right) -K \left( t^{p} \int_M |\nabla_\varphi F|^{2p}\omega_\varphi^n\right)^{1+\frac{1}{p}} + C t^{p+1} 
	\end{align}\\
	Then by standard calculus, we get 
	\begin{align}
		\sup_{(0,T)} t^{p+1} \int_M |\nabla_\varphi F|^{2p}\omega_0^n \le C(K,V,p,T)
	\end{align}\\
\end{proof}

\section{When the K\"ahler metric is close to cscK in volume form \\}
The goal of this section is to prove Theorem \ref{t1.2}. 
The key point is to use Corollary \ref{Cor33} to get improvement of regularity,  so that we get closeness to a cscK metric in the smooth topology,  and then we may invoke Theorem \ref{t1.1} part (1) to conclude the long time existence and convergence of the pseudo Calabi flow.

As a first step,  we show that the closeness of the volume form to cscK will imply a uniform lower bound on the life span of the pseudo Calabi flow.  More precisely,  we have:

\begin{lem}\label{l5.1New}
Let $(M,\omega_0)$ be a compact K\"ahler manifold.  Assume that $|Ric(\omega_0)|_{\omega_0}\le \Lambda$.  Then for any $V>0$,  there exists $\eps_0>0$,  $t_0>0$,  such that for any $\varphi_0\in PSH(M,\omega_0)\cap C^{\infty}(M)$ with 
\begin{equation}\label{5.1New}
1-\eps_0\le \frac{\omega_{\varphi_0}^n}{\omega_0^n}\le 1+\eps_0,
\end{equation}
the pseudo Calabi flow starting from $\varphi_0$ exists on $M\times [0,t_0]$.  Moreover, 
\begin{equation*}
\max_{t\in [0,t_0]}||\log\frac{\omega_{\varphi(t)}^n}{\omega_0^n}||_0\le V.
\end{equation*}
Here $\eps_0,\,t_0$ depends only on $\Lambda$ and $V$.
\end{lem}
\begin{proof}
Assume that starting from $\varphi_0$,  the pseudo Calabi flow has smooth solution on $M\times [0,T_{\max})$ where $T_{\max}$ is the maximal life span of the smooth solution.  We first show that for any $V>0$ there exists $t_0(V)>0$ (with the said dependence),    such that if the initial data satisfies (\ref{5.1New}) with some small enough $\eps_0$ and  $T_{\max}\le t_0(V)$ ,  then one has $\sup_{t\in [0,T_{\max})}||\log\frac{\omega_{\varphi(t)}^n}{\omega_0^n}||_0\le V$.  Then we see from Corollary \ref{Cor33} that if we choose $V= \delta$ (here $\delta$ is given by Corollary \ref{Cor33}),  then one must have $T_{\max}\ge t_0(V)$,  since Corollary \ref{Cor33} allows us to extend the flow if the log of volume ratio is $\le \delta$.

Now it only remains to see the existence of such $t_0(V)$.  Choose $\eps_0>0$ sufficiently small so that $V\ge 2\eps_0$.  Let $C_1$ be the constant given by Theorem \ref{MT21} such that $||P||_0\le C_1$ if one has $e^{-V}\le \frac{\omega_{\varphi}^n}{\omega_0^n}\le e^V$ and $|Ric(\omega_0)|_{\omega_0}\le \Lambda$.  Let $\alpha_1,\,\alpha_2$ be the constants given by Theorem \ref{MT22} with $\max_{t\in [0,T)}||P(t)||_0\le C_1$,  $|Ric(\omega_0)|_{\omega_0}\le \Lambda$.  Let $t_0(V)$ be sufficiently small so that:
\begin{equation*}
\frac{3V}{4}\ge \frac{V}{2}e^{\alpha_3t}+\alpha_4t.
\end{equation*}
Now we need to show that 
\begin{equation}\label{5.2New}
\sup_{t\in [0,T_{\max})}||\log\frac{\omega_{\varphi(t)}^n}{\omega_0^n}||_0\le V.
\end{equation}
Note that with $t=0$,  one has the above $<V$.  If (\ref{5.2New}) fails,  then there exists $t_1>0$,  which is the smallest $t$ such that $||\log\frac{\omega_{\varphi(t)}^n}{\omega_0^n}||_0=V.$  In particular,  one has: $e^{-V}\le \frac{\omega_{\varphi(t)}^n}{\omega_0^n}\le e^V$ for $t\in [0,t_1]$.  Therefore,  we may conclude using Theorem \ref{MT21} that $\max_{t\in[0,t_1]}||P(t)||_0\le C_1$.  Then Theorem \ref{MT22} gives:
\begin{equation*}
||F(t_1)||_0\le ||F(0)||_0e^{\alpha_3t_1}+\alpha_4t_1\le \frac{V}{2}e^{\alpha_3t_1}+\alpha_4t\le \frac{3V}{4}.
\end{equation*}
This is in contradiction with our assumption that $||\log\frac{\omega_{\varphi(t)}^n}{\omega_0^n}||_0=V$.
\end{proof}

As a consequence,  we see that:
\begin{cor}\label{c5.2}
Let $(M,\omega_0)$ be a compact K\"ahler manifold.  Assume that $|Ric(\omega_0)|_{\omega_0}\le \Lambda$.  Then there exists $\eps_0>0,\,t_0>0$,  such that for any $\varphi_0\in PSH(M,\omega_0)\cap C^{\infty}(M)$ with
\begin{equation*}
1-\eps_0\le \frac{\omega_{\varphi_0}^n}{\omega_0^n}\le 1+\eps_0,
\end{equation*}
one has
\begin{equation*}
\sup_{t\in (\eps,t_0]}||\partial_t^kD^l\varphi(\cdot,t)||_0\le C_{\eps,k,l}.
\end{equation*}
Here $D^l$ means $l$-th order differentiation in the manifold direction.
\end{cor}
\begin{proof}
Let $p_n$ be given by Lemma \ref{l4.6N}.  Then we go to Theorem \ref{MT32},  and let $\delta_*$ be the $\delta$ given by Theorem \ref{MT32} with $p=p_n$.  Next we take $V=\delta_*$ in Lemma \ref{l5.1New},  so that we get $\eps_0>0$,  $t_0>0$ such that one has:
\begin{equation*}
\max_{t\in [0,t_0]}||\log\frac{\omega_{\varphi(t)}^n}{\omega_0^n}||_0\le C,\,\,\max_{t\in [\eps,t_0]}||\nabla_{\varphi}F(\cdot,t)||_{L^p(\omega_{\varphi}^n)}\le C_{\eps}.
\end{equation*}
From Lemma \ref{l4.6N},  we see that all the higher derivatives of $\varphi$ can be bounded on the time interval $t\in [\eps',t_0]$ for $\eps'>\eps$.
\end{proof}

Now we are ready to prove Theorem \ref{t1.2}.
\begin{proof}
(Of Theorem \ref{t1.2}) Let $t_0>0$ be given by Corollary \ref{c5.2}.  We just need to show that for any $\delta>0$,  there exists $\eps_0>0$,  such that for any $\varphi_0\in PSH(M,\omega_0)\cap C^{\infty}(M)$ with $1-\eps_0\le \frac{\omega_{\varphi_0}^n}{\omega_0^n}\le 1+\eps_0$ normalized with $\int_M\varphi_0\omega_{\varphi_0}^n=0$,  one has $||\varphi(\cdot,t_0)||_{2,\alpha}\le \delta$.  

Assume that the above claim is false,  then there exists $\delta_*>0$,  such that we can find a sequence $\varphi_i\in PSH(M,\omega_0)\cap C^{\infty}(M)$,  with $\frac{\omega_{\varphi_i}^n}{\omega_0^n}\rightarrow 1$,  normalized with $\int_M\varphi_i\omega_{\varphi_i}^n=0$,  but $||\varphi_i(\cdot,t_0)||_{2,\alpha}\ge \delta_*$.  First from the stability result of complex Monge-Ampère equations,  we have that:
\begin{equation*}
\varphi_i\rightarrow 0\text{ uniformly}.
\end{equation*}
Denote $\varphi_i(t)$ to be the pseudo Calabi flow starting from $\varphi_i$.  From Corollary \ref{c5.2} we know that $\varphi_i(t)$ exists for $t\in [0,t_0]$.  Moreover,  the derivatives of $\varphi_i(t)$ are uniformly bounded on $(\eps,t_0]$ for any $\eps>0$.  Therefore,  we can take a subsequence $\varphi_{i_k}$ such that $\varphi_{i_k}\rightarrow \varphi_{\infty}$ smoothly on $(0,t_0]\times M$.  In particular,  $\varphi_{\infty}(t)$ also solves pseudo Calabi flow on $(0,t_0]\times M$.  We wish to show that $\varphi_{\infty}(t)\equiv 0$ on $(0,t_0]\times M$.  To see this,  we just need to show that $\varphi_{\infty}(t)$ is continuous at $t=0$ with respect to $C^0$ norm and $\varphi_{\infty}(0)=0$.  Then 
\begin{equation*}
K(\varphi_{\infty}(t))\le K(\varphi_{\infty}(0))=K(0)=\inf_{\varphi\in PSH(M,\omega_0)}K(\varphi).
\end{equation*}
Therefore,  we see that $\varphi_{\infty}(t)$ is cscK for $t\in [0,t_0]$.  On the other hand,  $\varphi_{\infty}(t)$ solves the pseudo Calabi flow equation for $0<t<t_0$,  therefore $\varphi_{\infty}$ is stationary,  and we can conclude that $\varphi_{\infty}(t)=\varphi_{\infty}(0)=0$.

It only remains to see the continuity of $\varphi_{\infty}(t)$ at $t=0$ and $\varphi_{\infty}(0)=0$.  We just need to show that $\varphi_i(t)$ is equi-continuous at $t=0$.  For this we use that $\partial_t\varphi_i=P_i+F_i$.  From Lemma \ref{l5.1New} we see that $F_i$ is uniformly bounded on $[0,t_0]$.  On the other hand, we may use Theorem \ref{MT21} to conclude that $P_i$ is also uniformly bounded on $[0,t_0]$.  Therefore,  $\partial_t\varphi_i$ is also uniformly bounded on $[0,t_0]$.  Therefore the result follows.
\end{proof}

\section{When $M$ is K\"ahler-Einstein \\}

In this section, we consider the case when $M$ admit a K\"ahler-Einstein (KE) metric. Without loss of generality,  we may assume that $c_1(M)=\lambda[\omega_0]$ where $\lambda=\pm 1$ or 0. \\

\begin{prop}
	Let $(M,\tilde\omega)$ be a K\"ahler-Einstein manifold, then the PCF $\omega(t)$ in class $[\tilde\omega]$ coincides with normalized K\"ahler-Ricci flow. \\
\end{prop}

\begin{proof}
Consider the Pseudo Calabi flow for a time--dependent K\"ahler potential $\varphi(t)$:
\begin{align}
	\omega_\varphi(t)=\omega_0+\sqrt{-1}\,\partial\bar\partial\varphi(t),\qquad \partial_t\varphi(t)=-f(t)
\end{align}\\
where $f(t)$ solves the Poisson equation
\begin{equation}\label{eq51}
\Delta_{\varphi} f \;=\; R_\varphi-\underline{R}.
\end{equation}\\
Since $Ric_\omega = \lambda\omega$, we have
\begin{align}
	[Ric(\omega_\varphi) -\lambda \omega_\varphi]=0 \in H^{1,1}(M,\mathbb{R}).
\end{align}\\
Hence there exists a time--dependent real function $h_\varphi$, a Ricci potential such that
\begin{equation}\label{eq52}
	Ric(\omega_\varphi)-\lambda\omega_\varphi \;=\; \sqrt{-1}\,\partial\bar\partial h_\varphi, \qquad \int_M e^{h_\varphi}\,\omega_\varphi^n=Vol(M),
\end{equation}
where the normalization fixes $h_\varphi$ up to an additive constant. Taking the $\omega_\varphi$--trace of \eqref{eq52} yields
\begin{align}
	\Delta_{\varphi} h_\varphi = tr_{\omega_\varphi}\!\big(Ric(\omega_\varphi)-\lambda\omega_\varphi\big) = R_\varphi-\lambda n
\end{align}\\
On the class $[\omega_\varphi]= \lambda \,c_1(M)$, the average scalar curvature equals
\begin{align}
	\underline R =\frac{n\,c_1(M)\cdot [\omega_\varphi]^{n-1}}{[\omega_\varphi]^n} = \lambda\,n
\end{align}
Therefore
\begin{align}
	\Delta_{\varphi} h_\varphi = R_\varphi-\underline R
\end{align}\\
Comparing with \eqref{eq51}, we see that $f$ and $h_\varphi$ solve the same Poisson equation, hence
\begin{align}
	f \;=\; h_\varphi + C(t)
\end{align}
for some function of time $C(t)$ (fixed by the chosen normalization). Applying $\sqrt{-1}\,\partial\bar\partial$-lemma, it gives
\begin{align}
	\sqrt{-1}\,\partial\bar\partial f = \sqrt{-1}\,\partial\bar\partial h_\varphi = Ric(\omega_\varphi)-\lambda \omega_\varphi.
\end{align}\\
By definition of the Pseudo Calabi Flow,
\begin{align}
	\partial_t\omega_\varphi = \sqrt{-1}\,\partial\bar\partial(\partial_t\varphi) = -\,\sqrt{-1}\,\partial\bar\partial f = -\,Ric(\omega_\varphi)\;+\lambda \;\omega_\varphi
\end{align}
which is precisely the normalized K\"ahler--Ricci flow. \\
\end{proof}

\noindent
The (normalized) K\"ahler-Ricci flow have be well researched in past decades. According to the work by \cite{Cao1985} \cite{TiZh2011}, we have the following theorem. \\ 

\begin{thm}\label{TMCT}
	Let $(M,\tilde\omega)$ be a K\"ahler-Einstein manifold. Then any normalized K\"ahler-Ricci flow $\omega(t)$ in the class $[\tilde\omega]$ always exist on $M\times [0,+\infty)$ and converges to a K\"ahler-Einstein metric. \\  
	\end{thm}

According to Lebrun-Simanca \cite{Lesi1994},  the set of DeRham classes swept out by the K\"ahler forms of extremal K\"ahler metrics is an open set in $H^{1,1}(M,\bR)$.  Moreover,  there exists an $h^{1,1}$-dimensional smooth family of extremal K\"ahler metrics near $\omega_0$ where $\omega_0$ is the K\"ahler-Einstein metric.  We denote these metrics to be $\omega_s$.  Since we assumed that $Aut_0(M,J)=0$,  we see that these extremal metrics are actually cscK metrics.  We have the following version of 
Theorem \ref{t1.1} part (2):
\begin{thm}\label{t6.2N}
There exists $\eps_0>0$ and a neighborhood $\mathcal{U}$ of $[\omega_0]$ in $H^{1,1}(M,\bR)$,  such that for any cscK metric $\omega_s$ whose class is in $\mathcal{U}$ and any $\varphi_s\in PSH(M,\omega_s)\cap C^{\infty}(M)$ with $||\varphi_s||_{2,\alpha}\le \eps_0$,  the pseudo Calabi flow in the class $[\omega_s]$ with initial data $\varphi_s$ exists on $M\times [0,+\infty)$ and converges to a cscK metric as $t\rightarrow \infty$.
\end{thm}
We wish to reduce the proof of Theorem \ref{t1.3} to Theorem \ref{t6.2N} and we just need to show that:
\begin{prop}\label{p6.2}
Let $(M,J,\omega_0)$ be a compact KE manifold with $Aut_0(M,J)=0$.  Then for any $\Gamma>1$,  $\delta_0>0$,  there exists $T_0>0$ and a neighborhood $\mathcal{U}_1$ such that for any class $[\omega]\in \mathcal{U}_1$ and any metric $\omega_1\in [\omega]$ with $\frac{1}{\Gamma}\le \frac{\omega_1^n}{\omega_0^n}\le \Gamma$,  the pseudo Calabi flow in the class $[\omega]$ starting from $\omega_1$ exists on $M\times[0,T_0]$ and satisfies $||\varphi(T_0)||_{2,\alpha}\le \delta_0$,  where $\varphi(T_0)$ is the K\"ahler potential for the pseudo Calabi flow at $t=T_0$ (under the background metric $\omega_s$).  
\end{prop}

To see that Proposition \ref{p6.2} implies Theorem \ref{t1.3},  we just need to choose $\delta_0=\eps_0$ where $\eps_0$ is given by Theorem \ref{t6.2N}.  Now it only remains to prove Proposition \ref{p6.2},  and it follows from the following lemmas.

The first step is the improvement of regularity.  Let $\mathcal{U}$ be the neighborhood of the canonical class swept out by the K\"ahler classes of cscK K\"ahler metrics.  For $[\omega]\in \mathcal{U}$,  and $\Gamma>1$,  we may denote:
\begin{equation*}
\mathcal{A}_{[\omega],\Gamma}=\{\varphi\in PSH(M,\omega_s)\cap C^{\infty}(M):\frac{1}{\Gamma}\le \frac{(\omega_s+\sqrt{-1}\partial\bar{\partial}\varphi)^n}{\omega_s^n}\le \Gamma\}.
\end{equation*}
Here $\omega_s$ is the cscK metric in $[\omega]$,  which depends smoothly on the class.
\begin{lem}\label{l6.3}
Let $(M,J,\omega_0)$ be a compact K\"ahler-Einstein manifold with $Aut_0(M,J)=0$.  Then for any $\Gamma>1$,  there exists $t_0>0$ and a neighborhood $\mathcal{U}_1$ of the canonical class such that for any class $[\omega]\in \mathcal{U}_1$ and any $\varphi\in \mathcal{A}_{[\omega],\Gamma}$,  the pseudo Calabi flow in the class $[\omega]$ starting from $\varphi$ exists on $M\times [0,t_0]$.  Moreover,  
\begin{enumerate}
\item For any $0<\eps<t_0$,  there exists a constant $C_{\eps,k,l}>0$
\begin{equation*}
\max_{t\in [\eps,t_0]}||\partial_t^kD^l\varphi(\cdot,t)||_0\le C_{\eps,k,l}.
\end{equation*}
\item For any $\eps_0>0$,  there exists a neighborhood $\mathcal{U}_2\subset \mathcal{U}_1$ of the canonical class,  and $\delta_0>0$,  such that for any two K\"ahler classes $[\omega_1],\,[\omega_2]\in \mathcal{U}_2$,  and any two functions $\varphi_i\in \mathcal{A}_{[\omega_i],\Gamma}$ with $||\varphi_1-\varphi_2||_0\le \delta_0$,  then one has:
\begin{equation*}
||\varphi_1(\cdot,t_0)-\varphi_2(\cdot,t_0)||_{10,\alpha}\le \eps_0.
\end{equation*}
Here $\varphi_i(\cdot,t)$ denotes the pseudo Calabi flow in the class $[\omega_i]$ with initial value $\varphi_i$.
\end{enumerate}
\end{lem}
The second step is somewhat expected,  which roughly says that if two K\"ahler classes are close enough and the two initial K\"ahler potentials are close enough under smooth topology,  then pseudo Calabi flow starting from them will stay close on a finite time interval.  More precisely:

\begin{lem}\label{l6.4}
Denote $\omega_0$ to be the K\"ahler-Einstein metric.  Let $\varphi_0\in PSH(M,\omega_0)\cap C^{\infty}(M)$.  Let $T_0>0$,  $\eps_1>0$,  then there exists a neigborhood $\mathcal{U}_3$ of $[\omega_0]$ and $\eps_2>0$,  such that for any K\"ahler class $[\omega]\in \mathcal{U}_3$,  any $\varphi_s\in PSH(M,\omega_s)\cap C^{\infty}(M)$ with $||\varphi_s-\varphi_0||_{10,\alpha}\le \eps_2$,  the pseudo Calabi flow with initial data $\varphi_s$ in the class $[\omega]$ exists on $M\times [0,T_0]$.  Moreover,  $||\varphi_s(T_0)-\varphi_0(T_0)||_{2,\alpha}\le \eps_1$.  Here $\varphi_s(t)$ denotes the solution to the pseudo Calabi flow with $\varphi_s$ being the initial data,  and $\varphi_0(t)$ denotes the normalized K\"ahler-Ricci flow.  Also $\omega_s$ is the cscK metric in the class $[\omega]$.
\end{lem}

First we explain how to use Lemma \ref{l6.3} and \ref{l6.4} to prove Proposition \ref{p6.2}:
\begin{proof}[Proof (of Proposition \ref{p6.2},  using Lemma \ref{l6.3} and \ref{l6.4})]
 Denote $\omega_{s_0}$ to be the cscK metric in the class $[\omega]$ and we write $\omega_1=\omega_{s_0}+\sqrt{-1}\partial\bar{\partial}\varphi_{s_0}$,  $\sup_M\varphi_{s_0}=0$.  We also denote $F_{s_0}=\log\frac{(\omega_{s_0}+\sqrt{-1}\partial\bar{\partial}\varphi_{s_0})^n}{\omega_{s_0}^n}$.  By choosing $\mathcal{U}_1$ small enough,  we may assume that $\frac{1}{2\Gamma}\le e^{F_{s_0}}\le 2\Gamma$.  We define $\varphi_0$ to be the solution to the following complex Monge-Ampère equation in the canonical class:
\begin{equation*}
(\omega_0+\sqrt{-1}\partial\bar{\partial}\varphi_0)^n=e^{F_{s_0}+c_{s_0}}\omega_0^n.
\end{equation*}
Here $c_{s_0}$ is a constant given by $e^{c_{s_0}}=\frac{\int_M\omega_0^n}{\int_Me^{F_{s_0}}\omega_0^n}$.  Since $F_{s_0}$ satisfies $\int_M\omega_{s_0}^n=\int_Me^{F_{s_0}}\omega_{s_0}^n$,  and $\omega_{s_0}\rightarrow \omega_0$ when $\mathcal{U}_1$ is chosen small enough.  Therefore,  $c_{s_0}\rightarrow 0$ as $\mathcal{U}_1$ is chosen small enough.  Therefore,  from the $C^0$ stability of the solution to the complex Monge-Ampère equation,  we may conclude that $||\varphi_{s_0}-\varphi_0||_0\le \delta_0$,  for any prescribed $\delta_0>0$,  as long as $\mathcal{U}_1$ is chosen small enough. 

Then we may apply Lemma \ref{l6.3},  part (2) to conclude that for any $\eps_0>0$,  one can make
\begin{equation*}
||\varphi_{s_0}(\cdot,t_0)-\varphi_0(\cdot,t_0)||_{10,\alpha}\le \eps_0,
\end{equation*}
by choosing $\delta_0>0$ small enough,  and also shrink $\mathcal{U}_1$ if necessary.  Then we use Lemma \ref{l6.3} and consider the pseudo Calabi flow and the normalized K\"ahler-Ricci flow for $t\in [t_0,t_0+T_0]$,  we would be able to make $||\varphi_s(T_0)-\varphi_0(T_0)||_{2,\alpha}$ as small as desired,  for any $T_0>0$ prescribed.  On the other hand,  with $T_0>0$ large enough,  we know that $\varphi_0(T_0)$ will be close to a constant in $C^{2,\alpha}$ norm,  using the convergence result of normalized K\"ahler-Ricci flow (Theorem \ref{TMCT}).   With such choice of $T_0$ and use the smallness of $||\varphi_s(T_0)-\varphi_0(T_0)||_{2,\alpha}$,  we also get the closeness of $\varphi_s(T_0)$ to a constant in $C^{2,\alpha}$ topology.  Then the existence of pseudo Calabi flow on $M\times [0,+\infty)$ and convergence to the cscK metric is given by Theorem \ref{t6.2N}.
\end{proof}

At this time,  it only remains to prove Lemma \ref{l6.3} and \ref{l6.4}. We first prove Lemma \ref{l6.3}.  Again,  we may assume that $c_1(M)=\lambda[\omega_0]$ where $\lambda=\pm 1$ or 0.  
First we show that
\begin{lem}\label{l6.5}
Assume that $(M,\omega_0)$ be a compact K\"ahler manifold with $c_1(M)=\lambda[\omega_0]$.   
Let $\Gamma>0,\,\delta_*>0$,  then there exists a neighorhood $\mathcal{U}_4$ of $[\omega_0]$,  such that for any $[\omega]\in \mathcal{U}_4$ and any $\varphi\in \mathcal{A}_{[\omega],\Gamma}$,  one has:
\begin{equation*}
|P-\lambda(\varphi-\frac{1}{vol([\omega])}\int_M\varphi\omega_{\varphi}^n)|\le \delta_*.
\end{equation*}
Here $\Delta_{\varphi}P=\underline{R}([\omega])-tr_{\omega_{\varphi}}(Ric(\omega_s))$,  $\omega_s$ denotes the cscK metric in the class $[\omega]$,  and $P$ is normalized so that $\int_MP\omega_{\varphi}^n=0$.
\end{lem}
Note that since $\frac{\omega_{\varphi}^n}{\omega_0^n}$ is bounded from above,  we see that one actually has $\varphi\in C^{\alpha}(M)$ (Demailly \cite{DDGH2008}).  Therefore,  this allows us to show that $P$ is $\delta_*$-close to a continuous function with modulus of continuity $\omega(r)=Cr^{\alpha}$ in the sense of Definition \ref{Def31}.  (with $\tilde{P}=c+\lambda\varphi$ with $c$ a suitable constant.)

\begin{proof}
(of Lemma \ref{l6.5}) We choose $\omega_s$ to be the background (reference) metric in the class $[\omega]$. 
Note that $\Delta_{\varphi}P=\underline{R}([\omega])-tr_{\omega_{\varphi}}(Ric(\omega_s))$.  On the other hand,  $\Delta_{\varphi}\varphi=n-tr_{\omega_{\varphi}}\omega_s$,  we see that:
\begin{equation*}
\Delta_{\varphi}(P-\lambda\varphi)=\underline{R}([\omega])-\lambda n-tr_{\omega_{\varphi}}(Ric(\omega_s)-\lambda \omega_s).
\end{equation*}
Without loss of generality,  we may normalize $P$ and $\varphi$ so that $\int_MP\omega_{\varphi}^n=0,\,\int_M\varphi\omega_{\varphi}^n=0$,  so that one has:
\begin{equation*}
(P-\lambda\varphi)(x)=\int_MG_{\omega_{\varphi}}(x,y)\big(tr_{\omega_{\varphi}}(\lambda\omega_s-Ric(\omega_s)+\underline{R}([\omega])-\lambda n)\omega_{\varphi}^n(y).
\end{equation*}
Here $G_{\omega_{\varphi}}(x,y)$ denotes the Green's function associated with the metric $\omega_{\varphi}$.  

Since $\omega_s$ is a smooth family of cscK metrics,  and $\underline{R}([\omega_0])=\lambda n,\,Ric(\omega_0)=\lambda\omega_0$,  we see that with $\mathcal{U}_4$ chosen small enough,  one can make:
\begin{equation*}
|\underline{R}([\omega])-\lambda n|\le \eps,\,\,-\eps\omega_s\le \lambda\omega_s-Ric(\omega_s)\le \eps\omega_s.
\end{equation*}
From \cite{GPSS2024},  Theorem 1.1,  we know that  there exists $K>0$,  which depends on $\Gamma$,  and is uniform once $\mathcal{U}_4$ is small enough:
\begin{equation}\label{6.11N}
\int_M|G_{\omega_{\varphi}}(x,y)|\omega_{\varphi}^n(y)\le K,\,\,\inf_{y\in M}G(x,y)\ge -K,\,\,\text{ for any $x\in M$}.
\end{equation}
Therefore,  one can write:
\begin{equation*}
\begin{split}
&(P-\lambda\varphi)(x)=\int_M(G_{\omega_{\varphi}}(x,y)+K)(tr_{\omega_{\varphi}}(\lambda \omega_s-Ric(\omega_s))+\underline{R}([\omega])-\lambda n)\omega_{\varphi}^n(y)\\
&-K\int_M(tr_{\omega_{\varphi}}(\lambda \omega_s-Ric(\omega_s))+\underline{R}([\omega])-\lambda n)\omega_{\varphi}^n(y).
\end{split}
\end{equation*}
Therefore one can estimate:
\begin{equation*}
\begin{split}
&|(P-\lambda \varphi)(x)|\le \int_M(G_{\omega_{\varphi}}(x,y)+K)\eps(tr_{\omega_{\varphi}}\omega_s+1)\omega_{\varphi}^n(y)+K\int_M\eps(tr_{\omega_{\varphi}}\omega_s+1)\omega_{\varphi}^n(y)\\
&=\int_M(G_{\omega_{\varphi}}(x,y)+K)\eps(n+1-\Delta_{\varphi}\varphi)\omega_{\varphi}^n(y)+K\int_M\eps(n+1-\Delta_{\varphi}\varphi)\omega_{\varphi}^n(y)\\
&=\eps\bigg((n+1)\int_MG_{\omega_{\varphi}}(x,y)\omega_{\varphi}^n(y)-\varphi(x)+2(n+1)Kvol([\omega])\bigg).
\end{split}
\end{equation*}
In view of (\ref{6.11N}),  one can make the above $\le \delta_*$ if $\eps$ is chosen small enough.
\end{proof}

Next,  we wish to establish an analogue of Lemma \ref{l5.1New} which establishes uniform lower bound of the life span of the pseudo Calabi flow for K\"ahler classes in a neighborhood of $[\omega_0]$.  More precisely we have:
\begin{lem}\label{l6.6}
Let $(M,\omega_0)$ be a compact K\"ahler-Einstein manifold.  For any $\Gamma>1$,  any $\delta_*>0$,  there exists a neighborhood $\mathcal{U}_5$ of $[\omega_0]$ in $H^{1,1}(M,\bR)$ and $t_0>0$,  such that for any $[\omega]\in \mathcal{U}_5$,  any $\varphi\in\mathcal{A}_{[\omega],\Gamma}$,  the pseudo Calabi flow starting from $\varphi$ in the class $[\omega]$ exists on $M\times [0,t_0]$.  Moreover,  $\varphi(t)\in \mathcal{A}_{[\omega],2\Gamma}$ for any $t\in [0,t_0]$.
\end{lem}
\begin{proof}
First we show that one has uniform lower and upper bound for $\frac{\omega_{\varphi(t)}^n}{\omega_0^n}$ on $t\in [0,t_0']$ for some $t_0'>0$.

For this we can choose $\mathcal{U}_5$ small enough,  so that for any $[\omega]\in \mathcal{U}_5$,  one has $\frac{1}{1.01}\le \frac{\omega_s^n}{\omega_0^n}\le 1.01$ so that for any $\varphi\in \mathcal{A}_{[\omega],\Gamma}$ one has 
$\frac{1}{1.01\Gamma}\le \frac{\omega_{\varphi}^n}{\omega_s^n}\le 1.01\Gamma$.  Here $\omega_s$ is the cscK metric in $[\omega]$.

We now apply Lemma \ref{l6.5},  with $\Gamma$ replaced by $10\Gamma$,  $\delta_*=1$,  then Lemma \ref{l6.5} gives us a neighborhood $\mathcal{U}_4$ and we may assume that $\mathcal{U}_5\subset \mathcal{U}_4$.  Now we go back to Theorem \ref{MT22},  with the following choice of constants:
\begin{equation*}
V_1=10\sup_{[\omega]\in \mathcal{U}_4,\,\varphi\in\mathcal{A}_{[\omega],\Gamma}}||\varphi||_0+10.
\end{equation*}
Let $\alpha_3,\,\alpha_4$ be the two constants given by Theorem \ref{MT22}.  Now we choose $t_0'$ small enough so that for any $t\in [0,t_0']$ one has:
\begin{equation*}
\log(1.5\Gamma)\ge \log(1.01\Gamma)e^{\alpha_3t}+\alpha_4t,\,\,t(\log(2\Gamma)+V_1)\le 1.
\end{equation*}
Let $\varphi(t)$ denotes the solution to the pseudo Calabi flow starting from $\varphi\in\mathcal{A}_{[\omega],\Gamma}$,  which exists on $M\times [0,T_{\max})$.  We show that $\varphi(t)\in \mathcal{A}_{[\omega],2\Gamma}$ if $t\le t_0',\,t<T_{\max}$.  Assume otherwise,  then we can find $0<t_0''<\min(t_0',T_{\max})$ to be the supremum of $t_*$ so that $\varphi(t)\in \mathcal{A}_{[\omega],2\Gamma}$ for $t\in [0,t_*]$.  We know that for $t\in [0,t_0'']$,  one has $\frac{1}{2\Gamma}\le \frac{\omega_{\varphi(t)}^n}{\omega_0^n}\le 2\Gamma$ so that $\frac{1}{2.02\Gamma}\le \frac{\omega_{\varphi}^n}{\omega_s^n}\le 2.02\Gamma$.

Next we use Lemma \ref{l6.5} to bound $P$: for $t\in[0,t_0'']$,
\begin{equation*}
\begin{split}
&|P(t)|\le 2||\varphi||_0+1\le 2||\varphi_0||_0+2\int_0^t(||F(t')||_0+||P(t')||_0)dt'+1\\
&\le 2\sup_{[\omega]\in \mathcal{U}_4,\varphi\in \mathcal{A}_{[\omega],\Gamma}}||\varphi||_0+1+2t\log(2.02\Gamma)+2\int_0^t||P(t')||_0dt'\\
&\le 2(\sup_{[\omega]\in \mathcal{U}_4,\varphi\in \mathcal{A}_{[\omega],\Gamma}}||\varphi||_0+1)+2\int_0^t||P(t')||_0dt'.
\end{split}
\end{equation*}
Therefore,  we may use Gronwall's inequality to get that:
\begin{equation*}
||P(t)||_0\le (2\sup_{[\omega]\in \mathcal{U}_4,\varphi\in \mathcal{A}_{[\omega],\Gamma}}||\varphi||_0+1)(2e^{2t}-1)\le V_1,\,\,t\in [0,t'].
\end{equation*}
On the other hand,  if we use Theorem \ref{MT22},  we get:
\begin{equation*}
||F(t)||_0\le ||F(0)||_0e^{\alpha_3t}+\alpha_4t\le \log(1.01\Gamma)e^{\alpha_3t}+\alpha_4t\le \log(1.5\Gamma).
\end{equation*}
That is,  $\frac{1}{1.5\Gamma}\le \frac{\omega_{\varphi(t)}^n}{\omega_s^n}\le 1.5\Gamma$.  Therefore,  $\frac{1}{1.5\times 1.01\Gamma}\le \frac{\omega_{\varphi(t)}^n}{\omega_0^n}\le 1.5\times 1.01\Gamma$,  $t\in[0,t_0'']$.  Since $1.5\times 1.01<2$,  we see that $\varphi(t)\in \mathcal{A}_{[\omega],2\Gamma}$ for some $t>t_0''$,  which contradicts that $t_0''$ is the supremum.

Next we show that $T_{\max}$ has a uniform lower bound.  For this we need to use Corollary \ref{Cor31}.  Let $C_1=\log(2.02\Gamma)$.  We have shown that for any $0\le t<T_{\max}$ and $t\le t_0''$,  one has $||F(t)||_0\le C_1$.  Let $\delta_*>0$ be the constant given by Corollary \ref{Cor31}.  Next we apply Lemma \ref{l6.5} with $\Gamma$ replaced by $2\Gamma$,  we get that for a possibly smaller neighborhood $\mathcal{U}_4'$,  any $[\omega]\in \mathcal{U}_4'$ and any $\varphi\in \mathcal{A}_{[\omega],2\Gamma}$,  $P$ is $\delta_*$-close to the function $\lambda(\varphi-\frac{1}{vol([\omega])}\int_M\varphi\omega_{\varphi}^n)$,  which has modulus of continuity $\omega(r)=Cr^{\alpha}$,  which follows  from the uniform boundedness of the volume form on $[0,t_0'']$.  Then we can conclude from Corollary \ref{Cor31} that one can bound all the derivatives of $\varphi$ for $t\in (\eps,t_0'')$,  when the neighborhood to be $\mathcal{U}_4'$ as explained above.
\end{proof}
So far we have proved part (1) of Lemma \ref{l6.3}.
Now we prove the second part.
\begin{proof}
(Lemma \ref{l6.3},  part (2)) We wish to argue by contradiction.  Assume that it is false,  then for some $\eps_0>0$,  we may find K\"ahler classes $[\omega_{1,i}],\,[\omega_{2,i}]$ which converges to $[\omega_0]$, as well as $\varphi_{1,i}\in \mathcal{A}_{[\omega_{1,i}],\Gamma}$,  $\varphi_{2,i}\in \mathcal{A}_{[\omega_{2,i}],\Gamma}$,  with $||\varphi_{1,i}-\varphi_{2,i}||\rightarrow 0$ as $i\rightarrow \infty$,  but still 
\begin{equation}\label{6.12}
||\varphi_{1,i}(\cdot,t_0)-\varphi_{2,i}(\cdot,t_0)||_{10,\alpha}\ge \eps_0.
\end{equation}
From Lemma \ref{l6.6} and Corollary \ref{Cor31},  we see that there exists $t_0>0$ such that with $i$ large enough,  $\varphi_{1,i}(t)$ and $\varphi_{2,i}(t)$ has uniform bounds on all derivatives on $M\times [\eps,t_0]$ for any $\eps>0$.  

Therefore,  we may find a subsequence $i_k$ such that we may define (smooth limit):
\begin{equation*}
\varphi_{1,\infty}(t)=\lim_{k\rightarrow \infty}\varphi_{1,i_k}(t),\,\,\,\varphi_{2,\infty}(t)=\lim_{k\rightarrow \infty}\varphi_{2,i_k}(t).
\end{equation*}
From the pseudo Calabi flow equation and that we have uniform $C^0$ bound for 
\begin{equation*}
	F_{1,i_k}(t),\,F_{2,i_k}(t),\,P_{1,i_k}(t),\,P_{2,i_k}(t),\,t\in [0,t_0]
\end{equation*} 
we see that $\varphi_{1,i_k}(t),\,\varphi_{2,i_k}(t)$ is equi-continuous at $t=0$ under $C^0$ norm.  

Therefore,  we can observe that:
\begin{itemize}
\item $\varphi_{1,\infty}(t),\,\varphi_{2,\infty}(t)$ are smooth and solve the normalized K\"ahler-Ricci flow on $M\times (0,t_0]$.
\item $[0,t_0]\ni t\mapsto \varphi_{1,\infty}(t),\,\varphi_{2,\infty}(t)$ is continuous under $C^0$ norm and $\varphi_{1,\infty}(0)=\varphi_{2,\infty}(0)$.
\end{itemize}
The following lemma shows that $\varphi_{1,\infty}(t)=\varphi_{2,\infty}(t)$.  On the other hand,  if we pass (\ref{6.12}) to limit (which we can because $\varphi_{1,i_k},\,\varphi_{2,i_k}$ converges smoothly),  we see that:
\begin{equation*}
||\varphi_{1,\infty}(t_0)-\varphi_{2,\infty}(t_0)||_{10,\alpha}\ge \eps_0.
\end{equation*}
This is a contradiction.  
\end{proof}
The above proof requires the following result:
\begin{lem}
Let $(M,\omega_0)$ be a compact K\"ahler-Einstein manifold with $c_1(M)=\lambda[\omega_0]$ where $\lambda=1$ or $-1$ or $0$.  Let $\varphi\in C(M)\cap PSH(M,\omega_0)$.  Let $\varphi_1,\,\varphi_2\in C^{\infty}(M\times (0,T))$ for some $T>0$ and solve the normalized K\"ahler-Ricci flow on $M\times (0,T)$.  Assume also that $[0,T]\ni t\mapsto \varphi_1(\cdot,t),\,\varphi_2(\cdot,t)$ are continuous under $C^0$ norm,  with $\varphi_1|_{t=0}=\varphi_2|_{t=0}=\varphi$.  Then we have $\varphi_1=\varphi_2$ on $M\times [0,T]$.
\end{lem}
\begin{proof}
By assumption,  we know that both $\varphi_1,\,\varphi_2$ solve the following equation:
\begin{equation*}
\partial_t\varphi=\log\frac{\omega_{\varphi}^n}{\omega_0^n}+\lambda\varphi.
\end{equation*}
Plugging in $\varphi_1,\,\varphi_2$ into the above equation,  and take difference,  we see that on $M\times (0,T)$:
\begin{equation}\label{6.13}
\partial_t(\varphi_1-\varphi_2)=\sum_{i,j}a_{i\bar{j}}\partial_{i\bar{j}}(\varphi_1-\varphi_2)+\lambda(\varphi_1-\varphi_2).
\end{equation}
Here $a_{i\bar{j}}$ is uniformly elliptic on $M\times [\eps,T]$.  From (\ref{6.13}),  we see that if we put $\psi=e^{-\lambda t}(\varphi_1-\varphi_2)$,  we get:
\begin{equation*}
\partial_t\psi=\sum_{i,j}a_{i\bar{j}}\partial_{i\bar{j}}\psi.
\end{equation*}
Therefore,  for every $\eps>0$,  we may use maximum principle to conclude that:
\begin{equation*}
\max_{M\times [\eps,T]}\psi= \max_M\psi(\cdot, \eps),\,\min_{M\times[\eps,T]}\psi=\min_M\psi(\cdot,\eps).
\end{equation*}
That is,  we obtain:
\begin{equation*}
\max_{M\times [\eps,T]}e^{-\lambda t}|\varphi_1-\varphi_2|\le e^{-\lambda \eps}\max_M|\varphi_1(\eps)-\varphi_2(\eps)|.
\end{equation*}
Let $\eps\rightarrow 0$,  and we use that $\varphi_1(\eps),\,\varphi_2(\eps)$ tends to $\varphi$ uniformly,  so that the right hand side above tends to 0 as $\eps\rightarrow 0$.  We get that $\varphi_1=\varphi_2$.
\end{proof}
So far we have finished the proof of Lemma \ref{l6.3}.

Now we prove Lemma \ref{l6.4},  to which we wish to apply the implicit function theorem and we first explain the set-up of the implicit function theorem.

Denote $d=h^{1,1}(M,\bR)$,  then there is a neighborhood $\mathcal{U}$ of $[\omega_0]$ such that each $[\omega]\in \mathcal{U}$ contains a unique cscK metric (given our assumption that $Aut_0(M,J)=0$),  which we denote it to be $\omega_s,\,s\in \bR^d$.  Moreover,  $\omega_s$ depends smoothly on $s$.  

On the other hand,  without loss of generality,  we may assume that $\int_M\varphi_0\omega_{\varphi_0}^n=0$.  
We can solve the normalized K\"ahler-Ricci flow starting from $\varphi_0$,  and we denote $\tilde{\varphi}_0(t)$ to be the solution,  namely:
\begin{equation*}
\partial_t\tilde{\varphi}_0=\log\frac{\omega_{\tilde{\varphi}_0(t)}^n}{\omega_0^n}+\lambda\tilde{\varphi}_0,\,\,\,\,\tilde{\varphi}_0(\cdot,0)=\varphi_0.
\end{equation*}
Then we put $\varphi_0(t)=\tilde{\varphi}_0(t)+c(t)$ where $c(t)$ is a function of $t$ with $c(0)=0$ so that $\int_M\varphi_0(t)\omega_{\varphi_0(t)}^n=0$.  We know that $\varphi_0\in C^{\infty}(M\times [0,T])$.  

Denote $C^{10+\alpha,5+\frac{\alpha}{2}}(M\times [0,T])$ to be the parabolic H\"older space,  namely the space of functions on $M\times [0,T]$,  such that 
\begin{equation*}
\sum_{k+2l\le 10}\sup_{(x_1,t_1),\,(x_2,t_2)\in M\times [0,T]}\frac{|D^k\partial_t^l\varphi(x_1,t_1)-D^k\partial_t^l\varphi(x_2,t_2)|}{(d(x_1,x_2)+|t_1-t_2|^{\frac{1}{2}})^{\alpha}}<\infty.
\end{equation*}
In the above $D^k$ means $k$-th order derivative in the $M$ direction.  We choose a sufficiently small neighborhood $\mathcal{V}$ of $[\omega_0]$ in $\bR^d$,  and a sufficiently small neighborhood $U_1$ of $\varphi_0(t)$ in $C^{10+\alpha,5+\frac{\alpha}{2}}(M\times [0,T])$,  such that for any $[\omega]\in \mathcal{V}$,  any $\varphi\in U_1$,  one has:
\begin{equation*}
\omega_s+\sqrt{-1}\partial\bar{\partial}\varphi>0,\,\,\text{for any $t\in [0,T]$}.
\end{equation*}
Here $\omega_s$ is the cscK metric on $[\omega]$.
Now we define the nonlinear map:
\begin{equation*}
\begin{split}
\mathcal{F}&:\mathcal{V}\times U_1\rightarrow C^{8+\alpha,4+\frac{\alpha}{2}}(M\times [0,T])\times C^{10,\alpha}(M)\\
&([\omega],\,\,\varphi)\mapsto \big(\partial_t\varphi-\log\frac{\omega_{\varphi}^n}{\omega_s^n}-\Delta_{\varphi}^{-1}(\underline{R}([\omega])-tr_{\omega_{\varphi}}(Ric(\omega_s))),\,\,\varphi|_{t=0}\big)
\end{split}
\end{equation*}
In the above $\omega_s$ is the cscK metric on $[\omega]$ and $\omega_{\varphi}=\omega_s+\sqrt{-1}\partial\bar{\partial}\varphi$.  Moreover,  $\Delta_{\varphi}^{-1}(\cdots)$ is normalized so that its integral with respect to $\omega_{\varphi}^n$ is zero.  We just need to show that:
\begin{prop}
The Frechet derivative of $\mathcal{F}$ with respect to $\varphi$ evaluated at $([\omega_0],\varphi_0(t))$ is invertible.  
\end{prop}
Therefore,  we may conclude from Implicit Function Theorem that there exists a neighborhood $\mathcal{V}'$ of $[\omega_0]$,  and a neighborhood $V_2$ of $\varphi_0$ in $C^{10,\alpha}(M)$,  such that for any $[\omega]\in \mathcal{V}'$ and any $\tilde{\varphi}\in V_2$,  there exists a unique $\varphi\in C^{10+\alpha,5+\frac{\alpha}{2}}(M\times [0,T])$ such that $\varphi$ solves the pseudo Calabi flow with initial data $\tilde{\varphi}$.  Denote this map to be $\mathcal{G}$.  This map $\mathcal{G}:\mathcal{V}'\times V_2\rightarrow C^{10+\alpha,5+\frac{\alpha}{2}}(M\times [0,T])$ is continuous.  This would imply the statement of Lemma \ref{l6.4}.

Now we need to compute the Frechet derivative of $\mathcal{F}$ with respect to $\varphi$,  and one can compute:
\begin{equation}
\begin{split}
\delta_{\varphi}\mathcal{F}([\omega_0],\varphi_0(t))&:C^{10+\alpha,5+\frac{\alpha}{2}}(M\times [0,T])\rightarrow C^{8+\alpha,4+\frac{\alpha}{2}}(M\times [0,T])\times C^{10,\alpha}(M)\\
&v\mapsto \big(\partial_tv-\Delta_{\varphi_0(t)}v-Q(v),\,v|_{t=0}\big).
\end{split}
\end{equation}
In the above,  $Q(v)$ satisfies:
\begin{equation}
\Delta_{\varphi}Q(v)=\lambda\langle\sqrt{-1}\partial\bar{\partial}v,\omega_{\varphi_0(t)}\rangle_{\omega_{\varphi_0(t)}},\,\,\,\int_MQ(v)\omega_{\varphi_0}^n+\int_M\lambda \varphi_0\cdot n\omega_{\varphi_0}^{n-1}\wedge \sqrt{-1}\partial\bar{\partial}v=0.
\end{equation}
Here $\langle\sqrt{-1}\partial\bar{\partial}v,\omega_{\varphi_0(t)}\rangle_{\omega_{\varphi_0(t)}}=g_{\varphi_0(t)}^{i\bar{q}}g_{\varphi_0(t)}^{p\bar{j}}v_{i\bar{j}}(\omega_{\varphi_0(t)})_{p\bar{q}}$.
So it only remains to show that:
\begin{lem}\label{l6.9}
For any $f\in C^{8+\alpha,4+\frac{\alpha}{2}}(M\times [0,T])$ and any $g\in C^{10,\alpha}(M)$,  there exists a unique $v\in C^{10+\alpha,5+\frac{\alpha}{2}}(M\times [0,T])$ such that:
\begin{equation}\label{Linear}
\partial_tv-\Delta_{\varphi_0(t)}v-Q(v)=f,\,\,v|_{t=0}=g.
\end{equation}
Here $Q(v)$ solves $\Delta_{\varphi}Q(v)=\lambda\langle\sqrt{-1}\partial\bar{\partial}v,\omega_{\varphi_0(t)}\rangle_{\omega_{\varphi_0(t)}}$ and $\int_MQ(v)\omega_{\varphi_0(t)}^n+\int_M\lambda\varphi_0\cdot n\omega_{\varphi_0}^{n-1}\wedge \sqrt{-1}\partial\bar{\partial}v=0$.
\end{lem}
Lemma \ref{l6.9} has essentially being proved in Chen-Zheng \cite{Chen-Zheng},  Proposition 5.4.  The precise version of what they proved (stated using our notations):
\begin{prop}\label{p6.10}
For any $f\in C^0([0,T],C^{\alpha}(M))$ and any $g\in C^{2,\alpha}(M)$,  there exists a unique solution $v\in C^{2,1}(M\times [0,T])$ (first order differentiable in $t$ and second order differentiable along $M$) with $\max_{[0,T]}\big(||\partial_t\varphi||_{C^{\alpha}(M)}+||\varphi||_{C^{2,\alpha}(M)}\big)<\infty$ that solves the problem (\ref{Linear}).
\end{prop}
The only difference betwen Lemma \ref{l6.9} and Proposition \ref{p6.10} is just in terms of the regularity requirements.  However,  if we assume that $f\in C^{8+\alpha,4+\frac{\alpha}{2}}(M\times [0,T])$ and $g\in C^{10,\alpha}(M)$,  it is easy to apply standard parabolic regularity results to see $v\in C^{10+\alpha,5+\frac{\alpha}{2}}(M\times [0,T])$.

\end{document}